\documentclass[11pt]{article}

\usepackage[utf8]{inputenc}   
\usepackage[english]{babel}
\usepackage{amsmath,amssymb,amsfonts,amsthm}
\usepackage{xcolor}

\usepackage{graphicx}
\usepackage{epstopdf}

\usepackage[scale=0.8]{geometry}

\theoremstyle{plain}
\newtheorem{theorem}{Theorem}
\newtheorem{lemma}[theorem]{Lemma}
\newtheorem{proposition}[theorem]{Proposition}

\theoremstyle{definition}
\newtheorem{definition}[theorem]{Definition}

\theoremstyle{remark}
\newtheorem{remark}[theorem]{Remark}
\newtheorem{remarks}[theorem]{Remarks}

\newcounter{a}
\newcommand{\C}{\textbf{Cond}}
\newenvironment{cond}{\par \refstepcounter{a}  
  {\upshape \textbf{Cond~\thea}:}}{\par}

\newcommand{\ind}[1]{1_{#1}}
\newcommand{\di}[1]{\text{d} #1}

\newcommand{\prob}[1]{\mathbb{P}\left(#1\right)}
\newcommand{\esp}[1]{\mathbb{E}\left[ #1 \right]} 
\newcommand{\eqd}{\stackrel{(d)}{=}}
\newcommand{\borel}[1]{\mathcal{B}\left(#1\right)}

\newcommand{\NN}{\mathbb{N}}
\newcommand{\ZZ}{\mathbb{Z}}
\newcommand{\LL}{\mathbb{L}}
\newcommand{\RR}{\mathbb{R}}

\newcommand{\ZZL}{{\ZZ/2L\ZZ}}

\newcommand{\tore}[1]{\ZZ / #1 \ZZ}
\newcommand{\cyl}[1]{\mathcal{C}_#1}

\newcommand{\mprob}[1]{\mathcal{P} \left( #1 \right)}

\newcommand{\trans}{\Phi}

\newcommand{\Paths}{\Pi}
\newcommand{\Front}{F}
\newcommand{\BB}{B}
\newcommand{\Bridges}{\mathcal{B}}
\newcommand{\Time}{\mathcal{T}}

\newcommand{\TT}{T}
\newcommand{\TL}{\tau}

\newcommand{\meanspeed}{c}

\newcommand{\edgetime}{\zeta}

\newcommand{\card}[1]{\text{card}\left\{ #1 \right\}}

\title{Generalised directed last passage percolation:\\ invariant laws on the cylinders \thanks{This work is supported by the ANR/FNS-16-CE93-0003 grant MALIN.}}
\author{Jérôme Casse \thanks{CEREMADE, CNRS UMR 7534, PSE research university, Université Paris-Dauphine, Place de Lattre de Tassigny, 75775 Paris 16, France. Email: jerome.casse@dauphine.psl.eu.}}

\begin{document}

\maketitle


\tableofcontents

\newpage
\begin{abstract}
  The directed last passage percolation (LPP) on the quarter-plane is a growing model. To come into the growing set, a cell needs that the cells on its bottom and on its left to be in the growing set, and then to wait a random time.
  
  We present here a generalisation of directed last passage percolation (GLPP). In GLPP, the waiting time of a cell depends on the difference of the coming times of its bottom and left cells. We explain in this article the physical meaning of this generalisation.

  In this first work on GLPP, we study them as a growing model on the cylinders rather than on the quarter-plane, the eighth-plane or the half-plane. We focus, mainly, on the law of the front line. In particular, we prove, in some integrable cases, that this law could be given explicitly as a function of the parameters of the model.

  These new results are obtained by the use of probabilistic cellular automata (PCA) to study LPP and GLPP.
\end{abstract}

\par
\emph{MSC Classes:} 60K35, 82B23\par

\begin{figure}[h]
  \begin{center}
    \includegraphics[width = 0.5\textwidth]{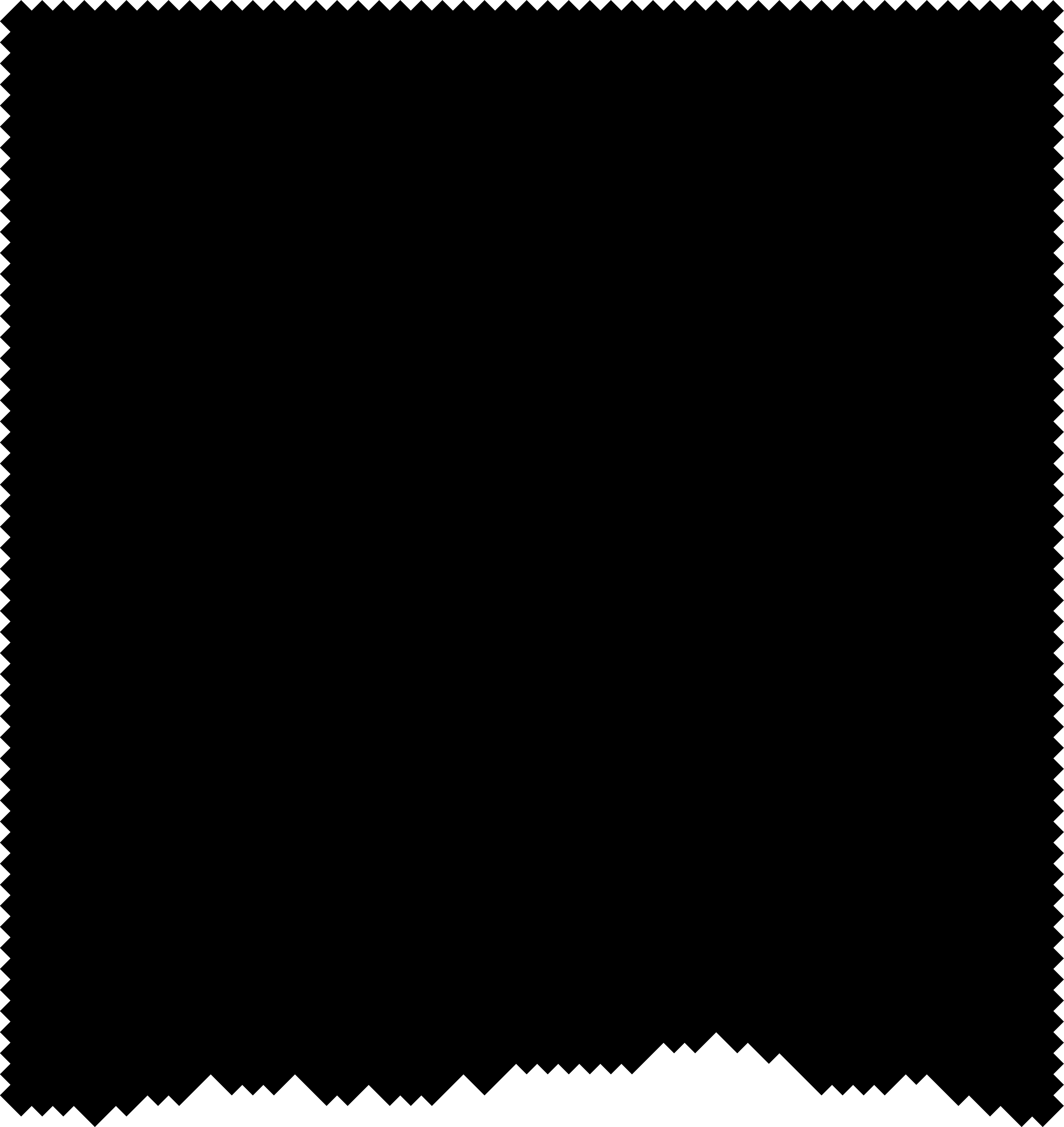}
  \end{center}
  \caption{The growing set (in black) of the LPP on the cylinder of size $50$ at time $200$. The set is growing down. Our main object of study here is the asymptotic law of the bottom line and the asymptotic height of the growing set when the time goes to infinity.}
\end{figure}

\section{Introduction}
In the preamble of this article, we fix the following notations: $\NN = \{0,1,\dots\}$, $\NN^* = \{1,2 , \dots\}$, $\RR_+ = [0,\infty)$ and $\RR_+^* = (0,\infty)$.
\paragraph{Last Passage Percolation (LPP):}
The \emph{directed Last Passage Percolation} (LPP) is a random lattice growth model. It has been introduced by Rost on the quarter-plane~\cite{Rost81}. To any vertex $z = (x,y) \in \NN^2$, we associate a random variable $\xi_{z}$. The $(\xi_{z})_{z \in \NN^2}$ are i.i.d.\ and distributed according to a probability measure $\mu$ on $\NN^*$ or $\RR_+^*$. Now, to any vertex $z \in \NN^2$, we associate the value
\begin{equation}
  \TL(z) = \max_{\pi = (z_0,\dots,z_k) \in \Paths(z)} \sum_{l=0}^{k} \xi_{z_l}
\end{equation}
where $\Paths(z)$ is the set of directed paths from $(0,0)$ to $z$, that is:
\begin{equation}
  \Paths(z) = \left\{ (z_0,\dots,z_k) \in (\NN^2)^{k+1} : \forall i < k, z_{i+1}-z_i \in  \{(1,0), (0,1)\}, z_0 = (0,0), z_k = z \right\}.
\end{equation}

Another way to define $\TL$ is by induction: for any $(x,y) \in \NN^2$,
\begin{equation}
  \TL((x,y)) = \max(\TL((x-1,y)),\TL((x,y-1))) + \xi_{(x,y)}
\end{equation}
with $\TL((-1,y)) = \TL((x,-1)) = 0$ for any $x,y \in \NN$.

\begin{remarks} \label{rem:banalite}
  \begin{itemize}
    \item Other names for the LPP on the quarter-plane are \emph{corner growth} models, point-to-point LPP, full-space LPP and it is very related to the PolyNuclear Growth (PNG) model and the TASEP (parallel or not). See~\cite{Martin06,Seppalainen09,Rost81,CEP96,BR00,PS02,Johansson00,Baik03,Johansson03,Ferrari04,SI04-1,SI04-2,SI05,BDMMZ01,BFP10,Johansson19,Johansson19-4,JR19} and~\cite[Chapter 2]{OccelliT19} for many references on those models. They are also related to the Hammersley lines models~\cite{AD95,Seppalainen97,CG05,CG06,BEGG16}.
    \item In the previous definition, taking ``$\min$'' instead of ``$\max$'' defines the \emph{directed First Passage Percolation} (FPP) on the quarter-plane.
    \item If we associate the random variables $\xi$ to edges instead of vertices, we define another model of LPP, called directed edge-LPP. In Section~\ref{sec:edge-classical}, we give some few more details about it.
    \end{itemize} 
\end{remarks}

\begin{remark}[Physical meaning of LPP] \label{rem:physsense}
  In our mind, the directed First Passage Percolation represents the time needed for a piece of ground to be wet when it starts raining at time $0$. And, in the directed Last Passage Percolation, the time $\TL(z)$ is the time needed for the piece of ground $z$ to be dry when it stops raining at time $0$. Indeed, to be dry, the piece of ground $(x,y)$ needs both pieces of ground $(x-1,y)$ and $(x,y-1)$ to be dry and, then, waits a random time $\xi_{(x,y)}$ to become dry.
\end{remark}

An interesting object in that model is the increasing sequence of sets of vertices
\begin{equation}
  \left(C_t = \{z \in \NN^2 : \TL(z) \leq t\}\right)_{t \in \NN \text{ or } \RR_+}.
\end{equation}

Throughout this article, we denote by $\mprob{\NN^*}$ the set of probability measures on $\NN^*$ whose support is $\NN^*$, i.e.\ if $\mu \in \mprob{\NN^*}$, then, for any $i \in \NN^*$, $\mu(i) \neq 0$.
\begin{theorem}[see~{\cite[Proposition~2.1]{Martin06}},~{\cite[Theorem~2.1]{Seppalainen09}}] \label{thm:classicalQuarter}
For any $\mu \in \mprob{\NN^*}$, there exists a deterministic function $f_\mu: (0,\infty)^2 \to [0,\infty]$ such that, for all $(x,y) \in (0,\infty)^2$,
\begin{equation}
  f_\mu(x,y) = \lim_{n \to \infty} \frac{\TL((\lfloor nx \rfloor,\lfloor ny \rfloor))}{n}  \text{ a.s.}
\end{equation}
Either $f_\mu = \infty$ or $f_\mu < \infty$ on all of $(0,\infty)^2$. In the latter case, $f_\mu$ is superadditive, concave, continuous, homogeneous, and symmetric on $(0,\infty)^2$. $f_\mu$ is non decreasing on both arguments.
\end{theorem}

Proof of this theorem is done by using a (superadditive version) of the Kingman's subadditive ergodic theorem~\cite{Kingman73}. The value of $f_\mu$ is not explicit, except when $\mu$ is a geometrical law or an exponential law. In that case,
\begin{proposition}[\cite{Rost81},~\cite{CEP96}] \label{thm:LPPint}
\begin{itemize}
  \item $f_{\mu}(x,y) = (\sqrt{x}+\sqrt{y})^2$ if $\mu$ is an exponential law of parameter $1$, and
  \item $\displaystyle f_{\mu}(x,y) = \frac{(1-p)x + 2 \sqrt{xy(1-p)} + (1-p)y}{p}$ if $\mu$ is a geometrical law of success parameter $p$ on $\NN^*$ (i.e.\ $\mu(i) = p (1-p)^{i-1}$ for any $i \in \NN^*$).
  \end{itemize}
\end{proposition}
See also~\cite{Martin06,Seppalainen09} and~\cite[Chapter 2]{OccelliT19}, but be careful there is sometimes confusion about the result concerning the geometrical law on the literature. Indeed, if we take $\mu$ to be a geometrical law of success parameter $p$ on $\NN$ [so $\left(\xi_z\right)_{z \in \NN^2}$ are now random variables on $\NN$ instead of $\NN^*$] (i.e.\ $\mu(i) = p (1-p)^i$ for any $i \in \NN$), then
\begin{equation}
  f_{\mu}(x,y) = \frac{x + 2 \sqrt{xy(1-p)} + y}{p}.
\end{equation}

In the sequel, we will refer to these explicit cases as ``integrable LPP''.\par

Besides,
the fluctuations around these explicit values have been studied, as well as the multipoint distribution~\cite{Johansson00,Johansson03,Johansson19,Johansson19-4,JR19,BDMMZ01,BFP10,BL16,BL18,Liu18,BL19-7,BL19-12,PS02}. They are related to the GUE Tracy-Widom distribution and $\text{Airy}_2$ processes. Hence, the LPP is in the KPZ (Kardar-Parisi-Zhang) universality class~\cite{Johansson00,PS02}. For many more details about KPZ universality of the LPP, we refer the interested reader to~\cite{KPZ86,Quastel11,Corwin12,Hairer13,Corwin18,OccelliT19} and references therein.\par

In the following of this article, we consider only discrete time (the support of $\mu$ is $\NN^*$) to get simpler mathematical expressions of ideas and formulas, and also to clarify the discussions. The case where the support is $\RR_+^*$ is done in Section~\ref{sec:cont}. The ideas are the same, but with more technical details.\par

\paragraph{Probabilistic Cellular Automata (PCA):}
The main new idea that leads to this article is the observation that LPP are related to Probabilistic Cellular Automata (PCA).\par
A PCA is a quadruplet $(E,\LL,N,T)$ where
\begin{itemize}
\item $E$ is a discrete space, 
\item $\LL$ is a discrete lattice,
\item $N = (z_1,\dots,z_n)$ is a finite subset of $\LL$,
\item $\TT = (\TT(s_1,\dots,s_{|N|};t))_{s_1,\dots,s_{|N|},t \in E}$ is a transition matrix from $E^{|N|}$ to $E$, meaning that $\TT$ satisfies the two following conditions:
  \begin{itemize}
  \item for any $\displaystyle s_1,\dots,s_{|N|},t \in E$, $\displaystyle \TT(s_1,\dots,s_{|N|};t) \in [0,1]$ and
  \item for any $\displaystyle s_1,\dots,s_{|N|} \in E$, $\displaystyle \sum_{t \in E} \TT(s_1,\dots,s_{|N|};t) = 1$.
  \end{itemize}
\end{itemize}
Each of this quadruplet $(E,\LL,N,T)$ allows to define a stochastic dynamic on $E^\LL$ in the following way: for any $s = (s_z)_{z \in \LL} \in E^{\LL}$, for any finite subset $L \subset \LL$, the probability that the image $U = (U_z)_{z \in \LL}$ of $s$ on $L$ by the dynamic is, for any $(u_z)_{z \in L} \in E^{L}$,
\begin{equation}
  \prob{(U_z = u_z)_{z \in L} | s} = \prod_{z \in L} \TT(s_{z+z_1},\dots,s_{z+z_N};u_z). 
\end{equation}
Hence, we know all the finite-dimensional laws of the random variable $U$ and so, by Kolmogorov's extension theorem, the law of $U$ itself. The random variable $U$ is then the image of $s$ by the stochastic dynamic associated to the PCA $A$, shorted in ``$U$ is the image of $s$ by $A$'' in the sequel. Moreover, $s$ could be a random variable of law $\phi$ on $E^\LL$, then $U$, the image of $s$ by $A$, is a random variable of law $\psi$ on $E^\LL$. Another point of view on the random dynamic associated to $A$ is to see it as a deterministic dynamic on the set of probability measures on $E^\LL$ that maps $\phi$ to $\psi$.
\medskip

Now, for any $\mu \in \mprob{\NN^*}$, we define $A_{\mu}$ the PCA where $E = \ZZ$, $\LL= \ZZ$, $N = \{0,1\}$, and $\TT_\mu$ is defined by: for any $s,t,u \in \ZZ$,
\begin{equation}
  \TT_\mu(s,t;u) = \mu(u - \max(s,t)).
\end{equation}

The first observation that leads to this article is:
\begin{lemma}
Let $(\TL(z))_{z \in \NN}$ be a LPP of parameter $\mu$, then, for any $(x,y) \in \NN^2$, for any $s,t,u \in \NN$,
\begin{equation}
  \prob{\TL((x,y)) = u | \TL((x-1,y)) =s, \TL((x,y-1)) = t} = \TT_\mu(s,t;u).
\end{equation}
\end{lemma}

The second observation is that ``integrable LPP'' correspond to cases where the PCA are \emph{integrable} (a precise definition of integrable PCA in our context is given in Section~\ref{sec:intPCA}). And, the reverse is true, if the PCA is integrable, then the corresponding LPP is integrable.\par

\begin{remark}
  We can also link the directed First Passage Percolation (the same definition than LPP but consider ``min'' instead of ``max'') with parameter $\mu$ and PCA by considering the following transition matrix: for any $s,t,u \in \ZZ$,
  \begin{equation}
    \TT_\mu(s,t;u) = \mu(u - \min(s,t)).
  \end{equation}
  Moreover, by using PCA with memory $2$ as defined in~\cite{CM19}, we can link them to FPP on the triangular lattice. Unfortunately, PCA linked with FPP are not integrable.
\end{remark}

At that point, the idea is to do something similar to what has been done on TASEP in~\cite{CM19}. It is to find integrable PCA that do not model the classical LPP as defined before, but another model that could be seen as a variant/generalisation. Moreover, we want to give, at least in some cases, a physical meaning to this generalisation. Now, we present this new generalisation and its physical meaning.\par

\paragraph{Generalised directed Last Passage Percolation (GLPP):}
Let $(\mu_{\Delta})_{\Delta \in \NN} \in \mprob{\NN^*}^{\NN}$ be a sequence of random probability measures on $\NN^*$. To any vertex $z \in \NN^2$, we attach a sequence of random variables $\xi_{z} = \left( \xi^{(\Delta)}_{z} \right)_{\Delta \in \NN}$ such that, for any $\Delta \in \NN$, $\xi^{(\Delta)}_{z} \sim \mu_\Delta$, and $\left(\xi_{z}\right)_{z \in \NN^2}$ are independent. From this, we define recursively $\TL((x,y))$ by
\begin{itemize}
\item $\TL((0,0)) = \xi^{(0)}_{(0,0)}$,
\item $\TL((x,0)) = \TL((x-1,0)) + \xi^{(\TL((x-1,0)))}_{(x,0)}$,
\item $\TL((0,y)) = \TL((0,y-1)) + \xi^{(\TL((0,y-1)))}_{(0,y)}$,
\item $\TL((x,y)) = \max(\TL((x-1,y)),\TL((x,y-1))) + \xi^{(|\TL((x-1,y)) - \TL((x,y-1))|)}_{(x,y)}$.
\end{itemize}\par
Remark that, in that model, neither the independence of $\left(\xi^{(\Delta)}_z\right)_{\Delta \in \NN}$, neither the identical distribution of $\left(\xi_z\right)_{z \in \NN^2}$ are required.\par

If, for any $\Delta \in \NN$, $\mu_\Delta = \mu_0$, then we obtain the classical LPP on the quarter-plane.\par

\begin{remarks}
  \begin{itemize}
  \item The physical meaning of LPP, as we express in Remark~\ref{rem:physsense}, is preserved and even improved.
    Indeed, in our generalisation, the time $\xi_{(x,y)}$ to dry depends on $|\TL((x-1,y)) - \TL((x,y-1))|$, the difference of drying times of $(x-1,y)$ and $(x,y-1)$. We think that it is more realistic: suppose that $\TL((x-1,y))$ is much bigger than $\TL((x,y-1))$, then, during a time $\Delta = \TL((x,y-1)) - \TL((x-1,y))$, $(x,y)$ receives water only from $(x,y-1)$, so when $(x,y-1)$ is dried, $(x,y)$ has less water that if $\Delta = 0$. This implies that, with this interpretation, $\left(\xi^{(\Delta)}_{(x,y)} \right)_{\Delta \in \NN}$ should be decreasing stochastically in $\Delta$.
  \item The directed edge-LPP can also be viewed as a GLPP, see Lemma~\ref{lem:edge-classical} in Section~\ref{sec:edge-classical}.
\end{itemize}
\end{remarks}

For any $\mu = (\mu_\Delta)_{\Delta \in \NN}$, the GLPP is related to the PCA whose transition matrix $\TT_{\mu}$ is, for any $s,t,u \in \ZZ$,
\begin{equation}
  \TT_{\mu}(s,t;u) = \mu_{|s-t|}(u - \max(s,t)),
\end{equation}
see Lemma~\ref{lem:model} to understand formally this relation. This PCA is integrable (as defined in Section~\ref{sec:intPCA}) if $\mu = \left(\mu_\Delta\right)_{\Delta \in \NN}$ satisfies the following condition
\begin{cond} \label{cond:int}
  for any $\Delta \in \NN$, for any $t \in \NN^*$,
  \begin{equation}
    \mu_\Delta(t) = \frac{\sqrt{\mu_0(t) \mu_0(t+\Delta)}}{\sum_{s \in \NN^*} \sqrt{\mu_0(s) \mu_0(s+\Delta)}}. \label{eq:int}
  \end{equation}
\end{cond}
The denominator is finite (less than $1$) due to Cauchy-Schwarz inequality.\par

\begin{remarks} \label{rem:dim}
  \begin{itemize}
  \item The GLPP is a model parameterised by $\mprob{\NN^*}^{\NN}$, but where "only" $\mprob{\NN^*}$ are integrable. In fact, in the classical case, a similar reduction of the model happens: the classical LPP can be parameterised by any measure $\mu \in \mprob{\NN^*} \simeq [0,1]^{\NN^*}$, but integrability happens when $\mu$ is a geometrical law that could be parameterised by its success parameter, an element of $[0,1]$. Hence, in both cases, ``a power $\NN$ is lost'' between the set of all models and the set of integrable ones.
  \item If the two following conditions hold on the same time: \C~\ref{cond:int} and, for any $\Delta \in \NN$, $\mu_\Delta = \mu_0$, then $\mu_0$ is a geometrical law; and the reverse is true (see Proposition~\ref{lem:classical} on Section~\ref{sec:classical}). Hence, we could not expect an improvement of the integrability conditions of the classical LPP by our methods; but, in the same time, our methods do not forget any ``integrable LPP''.
  \end{itemize} 
\end{remarks}

Our first objective was to generalise Theorem~\ref{thm:LPPint} to this new integrable condition. Unfortunately, for now, we do not succeed. Nevertheless, some simulations and conjectures are given in Section~\ref{sec:exQuarter}.\par
\smallskip
In this article, we are interested in this GLPP, not on the quarter-plane, but on the cylinders.\par
This is not the first time that LPP are not studied on the quarter-plane. In the literature, there are models of LPP on the half-plane (also called LPP line-to-point or Polynuclear Growth Model)~\cite{BR00,Johansson03,PS02,Ferrari04,SI04-1,JR19}, and models of LPP on the eighth-plane (under the name half-plane LPP, they are called half-plane because TASEP related to the LPP are on the half-line, see~\cite{BR01-1,BR01-2,BR01-3,BBNV18,BBCS18,BFO19} and~\cite[Chapter 2]{OccelliT19}. Results about LPP on cylinders can be deduced from the results about TASEP on rings. Recently, many results about the fluctuation of TASEP on rings have been obtained~\cite{BL16,BL18,Liu18,BL19-7,Liu19,BL19-12}.\par
\smallskip
Our aim in this article is to define the GLPP whose dynamics are more complex than the usual ones of LPP models as explained in Remarks~\ref{rem:dim}, and to show that some of them (those that satisfy~\C~\ref{cond:int}) have the potential to be studied as deeply as the LPP with exponential or geometric weights are.  In this first work on the subject, we focus on the front line of GLPP in cylinders. We see in Theorem~\ref{thm:int}~and~Proposition~\ref{prop:int-c} that the invariant probability measures of the front lines have more complex forms that the ones of LPP, see~Propositions~\ref{prop:clasCyl} and~\ref{prop:clasCylR}. 

\paragraph{Content:}
In Section~\ref{sec:LPPcyl}, we define the GLPP (with discrete time) on the cylinders and we express the four main results of this paper: Theorems~\ref{thm:HMC},~\ref{thm:meanspeed},~\ref{thm:int} and~\ref{thm:int-speed}. In Section~\ref{sec:proof}, we prove these four theorems. In Section~\ref{sec:ideas}, we explain how we are able to conjecture Theorem~\ref{thm:int} by using PCA. In Section~\ref{sec:cont}, we treat the continuous case that is when $\left( \mu_{\Delta} \right)_{\Delta \in \RR_+} \in \mprob{\RR_+^*}^{\RR_+}$ is a family of probability measures on $\RR_+^*$. In Section~\ref{sec:examples}, we present how our results on the cylinders apply to classical LPP and directed edge-LPP, and we discuss about the GLPP on the quarter-plane. Finally, in Section~\ref{sec:open}, we express and summarise some open questions on GLPP and some potential directions for future researches.

\section{GLPP on cylinders} \label{sec:LPPcyl}
\subsection{Definition} \label{sec:gLPPdef}
Let $L \in \NN^*$ be an integer and $\mu = (\mu_\Delta)_{\Delta \in \NN} \in \mprob{\NN^*}^{\NN}$ be a sequence of probability measures on $\NN^*$ with full support. The \emph{Generalised directed Last Passage Percolation} (GLPP) on the cylinder of size $L$ with parameter $\mu$ is a growing model on
\begin{equation}
  \cyl{L} = \{ (x,y) : x \in \tore{2L}, y \in \NN : x+y = 0 \mod 2\}
\end{equation}
such that, to each cell $(x,y) \in \cyl{L}$, we associate (by induction) a number $\TL((x,y))$ such that
\begin{align}
  & \TL((2x,0)) = 0 \text{ for any } x \in \tore{L}, \label{eq:LPPinit} \\
  & \TL((x,y))= \max(\,\TL((x-1,y-1))\, , \, \TL((x+1,y-1))\,) + \xi_{(x,y)} \label{eq:LPPdyn}
\end{align}
where $\xi_{(x,y)} \sim \mu_{|\TL(x-1,y-1) - \TL(x+1,y-1)|}$ and $\left( \xi_{(x,y)} \right)_{(x,y) \in \cyl{L}}$ are independent. \par

Our object of study is the curve $\Front_n$ that splits $\{z \in \cyl{L} : \TL(z) \leq n\}$ and $\{z \in \cyl{L}: \TL(z) > n\}$.\par 
In particular, we are interested in the law of $\Front_n$ when $n \to \infty$. For any $n \in \NN$, $\Front_n$ is an element of $\Bridges_L$, the set of bridges of size $2L$ whose steps are $+1$ or $-1$:
\begin{equation}
  \Bridges_L = \left\{ (b_i)_{i \in \ZZL} \in \{-1,+1\}^{\ZZL} : \sum_{i \in \ZZL} b_i =0 \right\}.
\end{equation}
Moreover, we define, for any $n \in \NN$ and any edge $i \in \ZZL$ of $F_n$, $t_{n,i} = n - \TL(z_i)$ where $z_i$ is the face adjacent to the edge $i$ and such that $\TL(z_i) \leq n$. It is denoted by $\tilde{\Front}_n = (\Front_n,(t_{n,i})_{i \in \ZZL})$. This is illustrated in Figure~\ref{fig:TF}. In the following, it is easier to work with $\tilde{F}_n$ than directly with $F_n$ (see Theorem~\ref{thm:HMC} in Section~\ref{sec:ergo}). Hence, many results are stated on $\tilde{\Front}_n$ and then deduced on $\Front_n$.\par

\begin{figure}
  \begin{center}
    \includegraphics{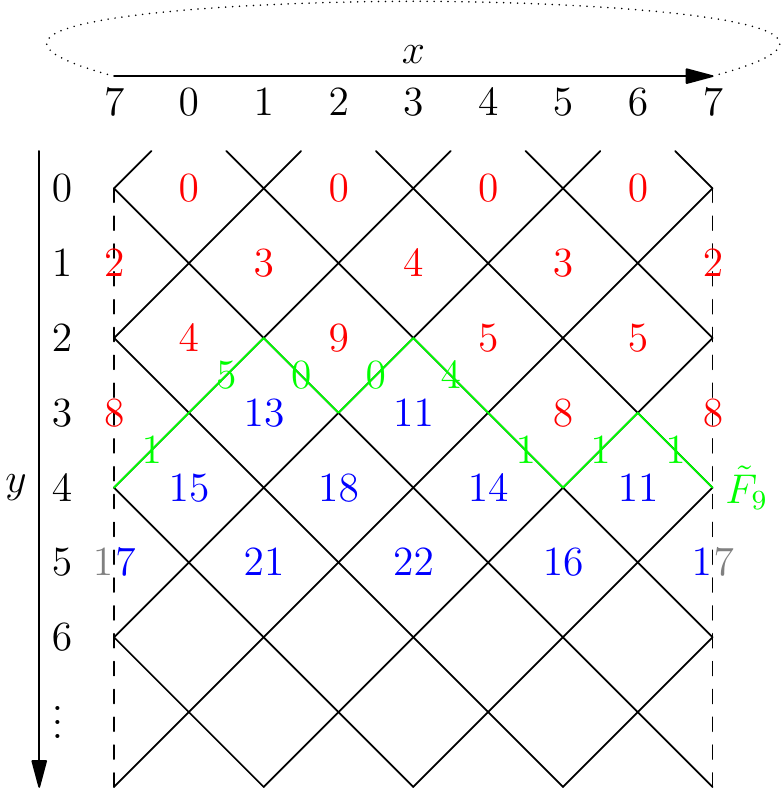}
  \end{center}
  \caption{Example of a LPP with $L=4$. The green line is $\Front_9 = (+1,-1,+1,-1,-1,+1,-1,+1)$ and, with the numbers on edges of $\Front_9$, we obtain $\tilde{\Front_9} = (\Front_9, (5,0,0,4,1,1,1,1))$.} \label{fig:TF}
\end{figure}

Few words about the set $\tilde{\Bridges}_L$ in which the random variable $\tilde{\Front}_n$ takes its values. For any $b \in \Bridges_L$, we define the set
\begin{align} 
\Time_b = \{(t_i & : i \in \ZZL) \in \NN^{\ZZL} : \nonumber \\
& \text{if } b_i = b_{i+1} = 1, \text{then } t_i<t_{i+1}; \nonumber \\
& \text{if } b_i = b_{i+1} = -1, \text{then } t_i>t_{i+1}; \nonumber \\
& \text{if } b_i = -1 \text{ and } b_{i+1} = 1, \text{then } t_i=t_{i+1}\} \label{eq:timesetCont}.
\end{align}
Then, for any $n$, $\tilde{\Front}_n$ is necessary an element of 
\begin{equation}
\tilde{\Bridges}_L = \{(b,t) : b \in \Bridges_L, t\in \Time_b\}.
\end{equation}

We also are interested in the asymptotic mean speed $\meanspeed_L$ of this front line that is
\begin{equation}
  c_L = \lim_{n \to \infty} \frac{\max \left\{y : \esp{\TL\left(\left(\frac{1 - (-1)^y}{2},y\right)\right)} \leq n \right\}}{n}.
\end{equation}
By a change of variable, $c_L$ could be rewritten as
\begin{equation}
  c_L = \lim_{y \to \infty} \frac{y}{\esp{\TL\left(\left(\frac{1 - (-1)^y}{2},y\right)\right)}}.
\end{equation}

For later, in relation to $c_L$, we introduce the notation $\edgetime(e)$ that is the time spend by the edge $e$ into the front line
\begin{equation} \label{eq:edgetime}
  \edgetime(e) = \TL(z'_e) - \TL(z_e),
\end{equation}
where $z'_e$ and $z_e$ are the two faces adjacent to the edge $e$ such that $\TL(z'_e) > \TL(z_e)$.


\begin{remarks}
  \begin{itemize}
  \item Due to invariance by horizontal translation of the model, $c_L$ does not depend on $x$ that's why we have chosen $x=0$ or $x=1$ here. Moreover, $c_L$ exists when $(\tilde{\Front}_n)_{n \in \NN}$ is ergodic, see Theorem~\ref{thm:HMC} in Section~\ref{sec:ergo} for a sufficient condition on $\mu$.
  \item In the definition of the model, we have chosen $\mu_\Delta \in \mprob{\NN^*}$. We could take it in $\mprob{\NN}$ allowing $\mu_\Delta(0) \in (0,1)$: it corresponds, for the growing process, to add several cells in the same time slot if one of them allows another to come. In the following, we do not study this case, even if some of our results apply (we just need to change the definition of $\Time_b$ by allowing equality in the two cases where $b_i=b_{i+1}$). The reason is that it complicates significantly some proofs. In Remark~\ref{rem:nono-proof}, we explain in details the issues of taking $\mu_\Delta(0) \neq 0$.
  \end{itemize}
\end{remarks}

\subsection{Ergodicity of the front line} \label{sec:ergo}
First, the following condition on $\left(\mu_\Delta\right)_{\Delta \in \NN}$ permits to assure the ergodicity of $(\tilde{\Front}_n)_{n \in \NN}$ and so of $(\Front_n)_{n \in \NN}$:
\begin{cond} \label{cond:conv}
  there exists $\alpha> 0$ such that
  \begin{equation}
    \inf_{\Delta \in \NN, t \in \NN^*} \frac{\mu_\Delta(t)}{\sum_{s \geq t} \mu_\Delta(s)} \geq \alpha.
  \end{equation}
\end{cond}

\begin{theorem} \label{thm:HMC}
  For any $\left( \mu_\Delta \right)_{\Delta \in \NN} \in \mprob{\NN^*}^\NN$, $(\tilde{\Front}_n)_{n \in \NN}$ is a Markov chain, and so $(\Front_n)_{n \in \NN}$ is a hidden Markov chain. Moreover, if~\C~\ref{cond:conv} holds, then they are ergodic. 
\end{theorem}

The proof of this theorem is done in Section~\ref{sec:proofHMC}.\par
\medskip

\begin{remarks}
  \begin{itemize}
  \item For those who are not familiar with hidden Markov chains: a process $\left(H_n \right)_{n\in \NN}$ is called a \emph{hidden Markov chain} on a set $E$, if there exists $(\tilde{H}_n)_{n \in \NN}$ a Markov chain on a set $\tilde{E}$ and a function $\pi : \tilde{E} \to E$, such that, for any $n \in \NN$, $H_n = \pi(\tilde{H}_n)$. Hence, if $(\tilde{F}_n)_{n \in \NN}$ is a Markov chain, by projection on the first coordinate, $(\Front_n)_{n \in \NN}$ is a hidden Markov chain.\par

  \item \C~\ref{cond:conv} is sufficient to obtain the ergodicity, but probably not optimal. We could expect weaker conditions by finest control on $\mu_\Delta$, in particular, by controlling the behaviour of $\mu_\Delta$ according to $\Delta$.
  \end{itemize}
\end{remarks}

When~\C~\ref{cond:conv} holds, we denote by $\tilde{\nu}_L$ the unique invariant law of the Markov chain $(\tilde{\Front}_n)_{n \in \NN}$ and $\nu_L$ the one of $(\Front_n)_{n \in \NN}$. Obviously, for any $b \in \Bridges_L$,
\begin{equation}
  \nu_L(b) = \sum_{t \in \Time_b} \tilde{\nu}_L((b,t)).
\end{equation}
We also obtain the asymptotic mean speed $\meanspeed_L$ of the front line as a function of $\tilde{\nu}_L$.
\begin{theorem} \label{thm:meanspeed}
  Let $\mu = (\mu_\Delta)_{\Delta \in \NN} \in \mprob{\NN^*}^\NN$ be such that~\C~\ref{cond:conv} holds. We denote by $\tilde{\nu}_L$ the invariant measure of the Markov chain $(\tilde{\Front}_n)_{n \in \NN}$. Let $(B,(T_1,\dots,T_{2L})) \sim \tilde{\nu}_L$. The asymptotic mean speed $\meanspeed_L$ of the front line of the LPP with parameter $\mu$ on the cylinder of size $L$ is
  \begin{equation}
    \meanspeed_L = \frac{1}{\displaystyle 2 \mathbb{E}\left[T_1\right] + 1} = \frac{1}{\displaystyle 1 + 2 \sum_{(b,(t_1,\dots,t_{2L})) \in \tilde{\Bridges}_L} t_1\, \tilde{\nu}_L((b,t))}.
  \end{equation}
\end{theorem}
The proof of this theorem is done in Section~\ref{sec:proofmeanspeed}. Moreover, for this theorem, it is necessary that: for any $\Delta \in \NN$, $\mu_\Delta(0)=0$.\par
\medskip

In the integrable case (when $\mu$ satisfies~\C~\ref{cond:int}), we have an explicit expression of $\tilde{\nu}_L$, and so of $\nu_L$ and $\meanspeed_L$, as a function of $\mu_0$.\par

\subsection{Integrable GLPP}
First, remark that the set of $\mu = \left( \mu_{\Delta} \right)_{\Delta \in \NN} \in \mprob{\NN^*}^{\NN}$ that satisfy~\C~\ref{cond:int} is parameterised by $\mu_0 \in \mprob{\NN^*}$. Indeed, from any $\mu_0 \in \mprob{\NN^*}$, we can define by~\C~\ref{cond:int} a unique sequence $\mu = (\mu_\Delta)_{\Delta \in \NN}$ in $\mprob{\NN^*}^{\NN}$. Hence, in the following, when we study the integrable case, we reduce the set of parameters $\mu = (\mu_\Delta)_{\Delta \in \NN} \in \mprob{\NN^*}^\NN$ to $\mu_0 \in \mprob{\NN^*}$.\par

When~\C~\ref{cond:int} holds, \C~\ref{cond:conv} becomes the following one on $\mu_0$:
  \begin{cond} \label{cond:conv2}
    there exists $\alpha > 0$ such that
    \begin{equation}
      \sup_{t \in \NN^*}\frac{\mu_0(t)}{\sum_{s\geq t} \mu_0(s)} \geq \alpha.
    \end{equation}
  \end{cond}

\begin{lemma} \label{lem:cc2}
  Let $\mu_0 \in \mprob{\NN^*}$ be such that~\C~\ref{cond:conv2} holds. Define $\mu = (\mu_\Delta)_{\Delta \in \NN} \in \mprob{\NN^*}^\NN$ using~\C~\ref{cond:int}. Then \C~\ref{cond:conv} holds for $\mu$. 
\end{lemma}

\begin{proof}
  For any $\Delta \in \NN$, $t \in \NN^*$,
  \begin{align*}
    \frac{\mu_\Delta(t)}{\sum_{s \geq t} \mu_\Delta(s)} & = \frac{\sqrt{\mu_0(t) \mu_0(t+\Delta)}}{\sum_{s \geq t} \sqrt{\mu_0(s) \mu_0(s+\Delta)}}  \\
                                                        & = \sqrt{\frac{\mu_0(t)}{\sum_{s \geq t} \mu_0(s)}} \sqrt{\frac{\mu_0(t+\Delta)}{\sum_{s \geq t} \mu_0(s+\Delta)}} \frac{\sqrt{\sum_{s \geq t} \mu_0(s)} \sqrt{\sum_{s \geq t} \mu_0(s+\Delta)}}{\sum_{s \geq t} \sqrt{\mu_0(s) \mu_0(s+\Delta)}}\\
                                                        & \geq \sqrt{\alpha} \times \sqrt{\alpha} \times 1= \alpha.
  \end{align*}
  The last inequality comes from~\C~\ref{cond:conv2} twice and from the Cauchy-Schwarz inequality.
\end{proof}

Before to give the third main theorem of this article, we need to introduce one notation.
For any $b \in \Bridges_L$, for any $t \in \Time_b$, set
 \begin{align} 
   W_{(b,t)} = &  \left( \prod_{i: b_i=b_{i+1}} \frac{\sqrt{\mu_0(|t_{i+1}-t_i|)}}{\sum_{s \geq 1} \sqrt{\mu_0(s)}} \right) \times \nonumber \\
             & \quad \left( \prod_{i: b_i=1,b_{i+1}=-1} \left(\sum_{s > \max(t_i,t_{i+1})} \frac{\sqrt{\mu_0(s-t_i)}}{\sum_{s \geq 1} \sqrt{\mu_0(s)}} \frac{\sqrt{\mu_0(s-t_{i+1})}}{\sum_{s \geq 1} \sqrt{\mu_0(s)}}\right)  \left( \sum_{s \geq 1} \mu_{|t_i-t_{i+1}|}(\min(t_i,t_{i+1}) + s) \right) \right). \label{eq:wbt} 
 \end{align}

 \begin{remark}
   We could obtain a little simplification for $W_{(b,t)}$ with
   \begin{align} 
     W_{(b,t)} = &  \left( \prod_{i: b_i=b_{i+1}} \sqrt{\mu_0(|t_{i+1}-t_i|)} \right) \times \nonumber \\
                 & \quad \left( \prod_{i: b_i=1,b_{i+1}=-1} \left(\sum_{s > \max(t_i,t_{i+1})} \sqrt{\mu_0(s-t_i)} \sqrt{\mu_0(s-t_{i+1})} \right)  \left( \sum_{s \geq 1} \mu_{|t_i-t_{i+1}|}(\min(t_i,t_{i+1}) + s) \right) \right). \label{eq:wbt-simple} 
   \end{align}
   Indeed, the two forms are proportional according to the factor $\left( \sum_{s \geq 1} \sqrt{\mu_0(s)} \right)^{2L}$ that does not depend on $(b,t)$. Moreover, this last form could be simplified in
   \begin{equation} \label{eq:wbt-usimple}
     W_{(b,t)} = \left( \prod_{i: b_i=b_{i+1}} \sqrt{\mu_0(|t_{i+1}-t_i|)} \right) \left( \prod_{i: b_i=1,b_{i+1}=-1} \left(\sum_{s \geq 1} \sqrt{\mu_0(s+t_i)} \sqrt{\mu_0(s+t_{i+1})} \right) \right)
   \end{equation}
   because
   \begin{align}
     & \left(\sum_{s > \max(t_i,t_{i+1})} \sqrt{\mu_0(s-t_i)} \sqrt{\mu_0(s-t_{i+1})} \right)  \left( \sum_{s \geq 1} \mu_{|t_i-t_{i+1}|}(\min(t_i,t_{i+1}) + s) \right) \label{eq:deb} \\
     = & \left(\sum_{s > \max(t_i,t_{i+1})} \sqrt{\mu_0(s-t_i)} \sqrt{\mu_0(s-t_{i+1})} \right)  \left( \sum_{s \geq 1} \frac{\sqrt{\mu_0(s+t_i) \mu_0(s+t_{i+1})}}{\sum_{u \geq 1} \sqrt{\mu_0(u) \mu_0(u+|t_{i}-t_{i+1}|)}}  \right) \text{ (by~\C~\ref{cond:int})} \\
     = &  \frac{\displaystyle \left(\sum_{s > \max(t_i,t_{i+1})} \sqrt{\mu_0(s-t_i)} \sqrt{\mu_0(s-t_{i+1})} \right)  \left(\sum_{s \geq 1} \sqrt{\mu_0(s+t_i) \mu_0(s+t_{i+1})}\right)}{\displaystyle \sum_{v \geq 1 + \max(t_i,t_{i+1})} \sqrt{\mu_0(v-\max(t_i,t_{i+1}))} \sqrt{\mu_0(v-\max(t_i,t_{i+1})+|t_{i}-t_{i+1}|)}} \\
     = & \sum_{s \geq 1} \sqrt{\mu_0(s+t_i)} \sqrt{\mu_0(s+t_{i+1})}. \label{eq:fin}
   \end{align}

   We give these three alternative forms for $W_{(b,t)}$ and not just one because they are all useful in the following. The third one~\eqref{eq:wbt-usimple} is the simplest and its expression is a function of the only $\mu_0$, the parameter of integrable GLPP. The second one~\eqref{eq:wbt-simple} permits to get an easier proof of Theorem~\ref{thm:int}, see Section~\ref{sec:proofint}. Finally, the first one~\eqref{eq:wbt} is the easiest to conjecture by the method explained in Section~\ref{sec:ideas}, and, in particular, in Section~\ref{sec:PCAnu}. 
 \end{remark}
 
\begin{theorem} \label{thm:int}
  For any $\mu_0 \in \mprob{\NN^*}$ such that~\C~\ref{cond:conv2} holds. Define $\mu = (\mu_\Delta)_{\Delta \in \NN}$ by~\C~\ref{cond:int}. In that case, for any $L \in \NN^*$, for any $b \in \Bridges_L$, for any $t \in \Time_b$,
  \begin{equation}
    \tilde{\nu}_L{(b,t)} = \frac{1}{Z_L} W_{(b,t)} \label{eq:conj}
  \end{equation}
  with $\displaystyle Z_L = \sum_{b \in \Bridges_L} \sum_{t \in \Time_b} W_{(b,t)}$. And so
  \begin{equation}
    \nu_L(b) = \frac{1}{Z_L} \sum_{t \in \Time_b} W_{(b,t)}.
  \end{equation}
\end{theorem}

Now, we could ask about the limit of $\nu_L$ when $L \to \infty$ on these integrable cases. Currently, it is an open problem.\par

Moreover, we are also able to give the mean speed $\meanspeed_L$ of this front line.
\begin{theorem} \label{thm:int-speed}
  For any $\mu_0 \in \mprob{\NN^*}$ such that~\C~\ref{cond:conv2} holds. Define $(\mu_\Delta)_{\Delta \in \NN}$ by~\C~\ref{cond:int}. In that case, for any $L\in \NN^*$, the asymptotic mean speed of the GLPP with parameter $\mu$ on the cylinder of size $L$ is
  \begin{equation} \label{eq:int-speed}
    \meanspeed_L = \frac{\displaystyle \sum_{(b,t) \in \tilde{\Bridges}_L} W_{(b,t)} }{\displaystyle \sum_{(b,t) = (b,(t_1,\dots,t_{2L})) \in \tilde{\Bridges}_L} (2 t_1 +1) W_{(b,t)}}.
  \end{equation}
\end{theorem}
As before, it is important here that $\mu_0(0) =0$.
In Section~\ref{sec:intLPP}, we give another expression for $c_L$ in~\eqref{eq:int-speed-alt}.\par

The proofs of Theorems~\ref{thm:int} and~\ref{thm:int-speed} are done in Section~\ref{sec:proofint}, and they are some computations once the values of $W_{(b,t)}$ and \C~\ref{cond:int} are conjectured. In fact, the most difficult part is to establish these conjectures that is done in Section~\ref{sec:ideas} applying the theory of PCA.

\section{GLPP on cylinders: discrete time} \label{sec:proof}
\subsection{Proof of Theorem~\ref{thm:HMC}} \label{sec:proofHMC}
\paragraph{Proof that $(\tilde{\Front}_n)$ is a Markov chain:} 
The dynamic of $\left(\tilde{\Front}_n\right)_{n \in \NN}$ is the following one: if, at time $n$, we have $(b_n,t_n) = ((b_{n,i})_{i \in \ZZL},(t_{n,i})_{i \in \ZZL})$, then, for any $j \in \ZZL$, for any $k \in \{-1,1\}$ such that
\begin{enumerate}
\item $b_{n,j} = k = b_{n,j+k}$, then $b_{n+1,j} = b_{n,j} = k$ and $t_{n+1,j} = t_{n,j} + 1$ ;
\item $b_{n,j} = k = - b_{n,j+k}$ (i.e. $b_n$ has a local maximum between $j$ and $j+k$), then 
\begin{enumerate}
\item $(b_{n+1,j},b_{n+1,j+k},t_{n+1,j},t_{n+1,j+k}) = (-k,k,0,0)$ with probability $\displaystyle p_{m}^{\delta} = \frac{\mu_\delta(1 + m)}{\sum_{s \geq 1} \mu_\delta(s + m)}$ where $m = \min \left( t_{n,j},t_{n,j+k} \right)$ and $\delta= \left| t_{n,j}-t_{n,j+k} \right|$;
\item $(b_{n+1,j},b_{n+1,j+k},t_{n+1,j},t_{n+1,j+k}) = (k,-k,t_{n,j}+1,t_{n,j+k}+1)$ with probability $\displaystyle 1 - p_{m}^{\delta}$ where $m = 1 + \min \left( t_{n,j},t_{n,j+k} \right)$ and $\delta= \left| t_{n,j}-t_{n,j+k} \right|$.
\end{enumerate}
\end{enumerate}

\begin{figure}
  \begin{center}
    \begin{tabular}{|c|c|}
      \hline
      $n$ & $n+1$ \\\hline
      $k=1$ & $k=1$ \\
      \includegraphics{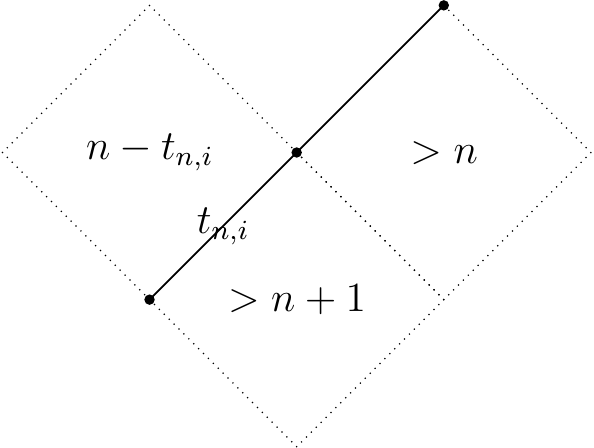} & \includegraphics{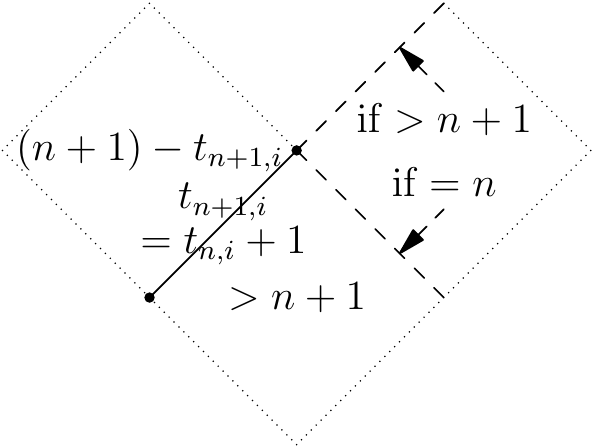}\\
          & w.p. $1$ \\\hline
      $k=-1$ & $k=-1$ \\
      \includegraphics{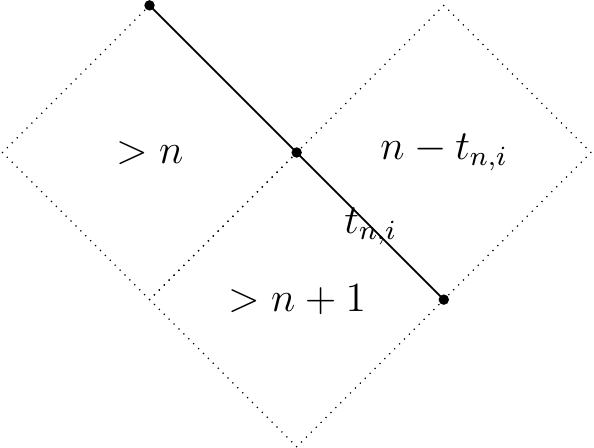} & \includegraphics{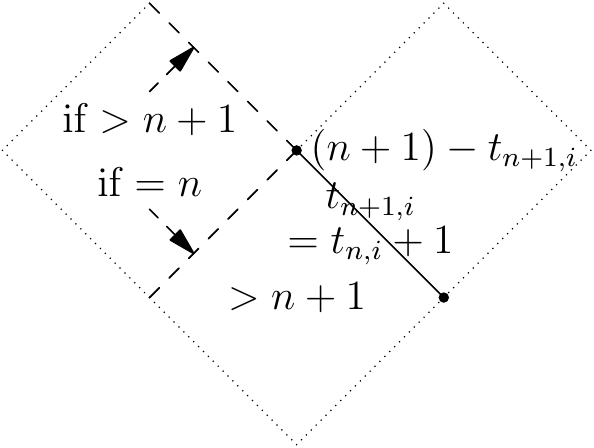} \\
          & w.p. $1$ \\\hline
    \end{tabular}
  \end{center}
  \caption{Case 1 of the dynamic}
  \label{fig:dyn1}
\end{figure}

\begin{figure}
\begin{center}
    \begin{tabular}{|c|c|}
      \hline
      $n$ & $n+1$ \\\hline
      \includegraphics{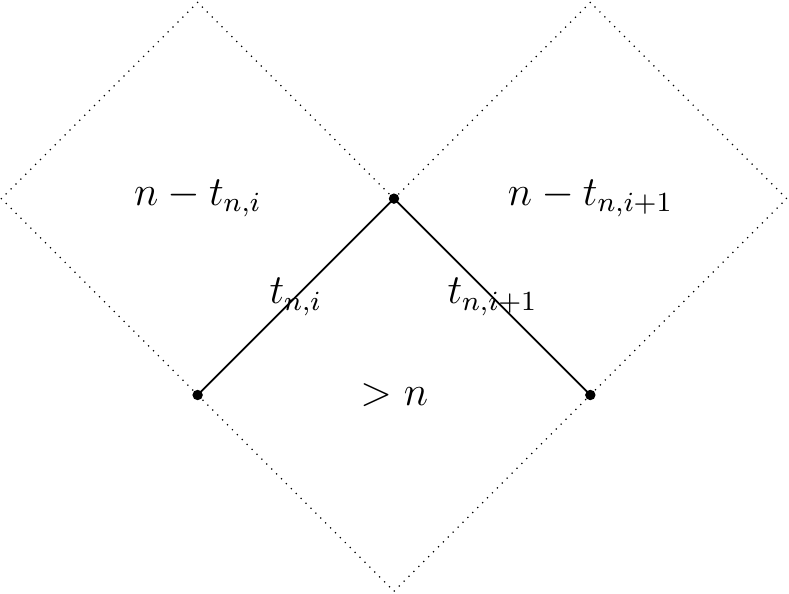} &
      {\begin{tabular}{c}
         (a) \\ \smallskip
         \includegraphics{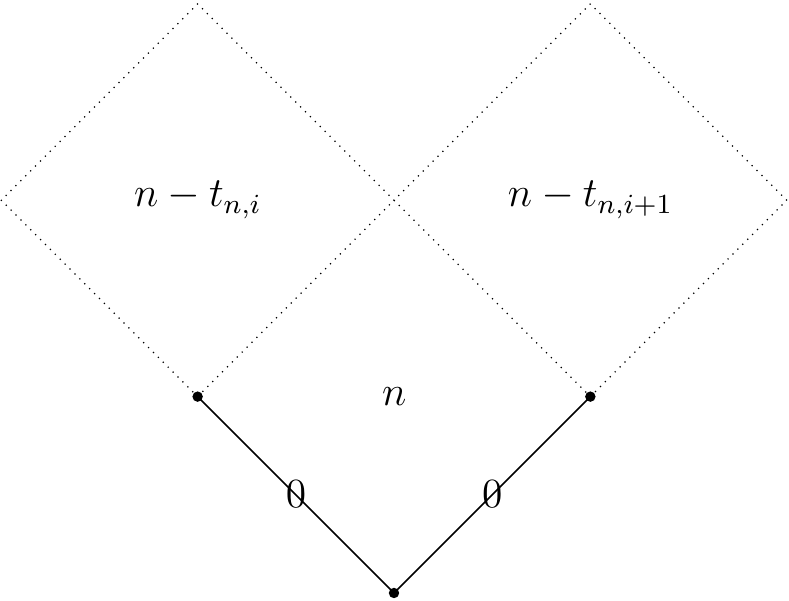}\\ \smallskip
         w.p. $p_{1 + \min(t_{n,j},t_{n,j+1})}^{|t_{n,j}-t_{n,j+1}|}$\\
         \hline
         (b) \\ \smallskip
         \includegraphics{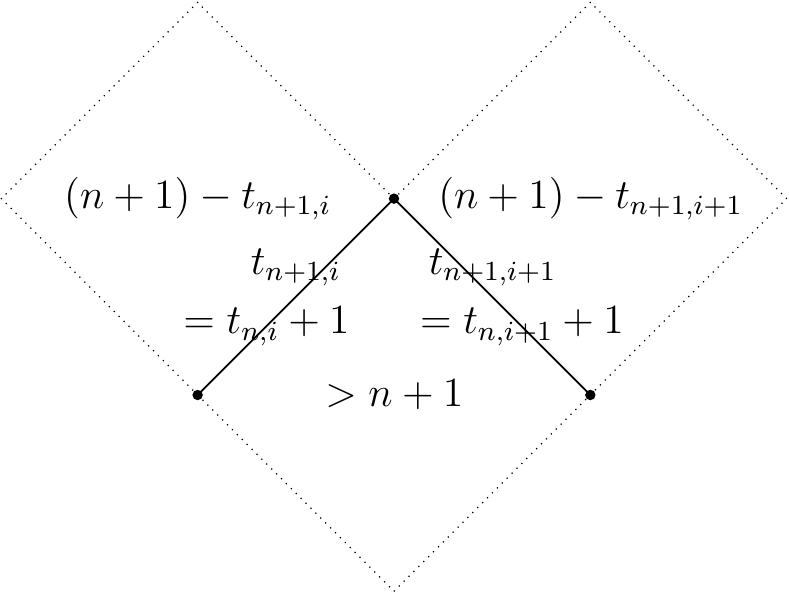} \\ \smallskip
         w.p. $1-p_{1 + \min(t_{n,j},t_{n,j+1})}^{|t_{n,j}-t_{n,j+1}|}$
      \end{tabular}}
      \\\hline
    \end{tabular}
  \end{center}
  \caption{Case 2 of the dynamic}
  \label{fig:dyn2}
\end{figure}

Why is it the same dynamic as the definition of $(\tilde{F}_n)_{n \in \NN}$?
The dynamic is obviously the same in the case 1, illustrated in Figure~\ref{fig:dyn1}. In the case 2, we need to justify the value of $p_{m}^{\delta}$. Suppose that we are in the second case as illustrated in Figure~\ref{fig:dyn2}. Let us define $\TL_j = n-t_{n,j}$ and $\TL_{j+k} = n - t_{n,j+k}$, and denote by $z_1$ the face adjacent to edge $j$ such that $\TL_1 = \TL(z_1) \leq n$, $z_2$ the one adjacent to edge $j+k$ such that $\TL_{j+k} = \TL(z_2) \leq n$ and $z_3$ the one that is adjacent to both edges $j$ and $j+k$. On the GLPP, the fact that $\Front_n$ is $b_n$ means that $\TL(z_1) = \TL_j$, $\TL(z_2) = \TL_{j+k}$ and $\TL(z_3) > n$. Now,
\begin{align}
  & \prob{\TL(z_3) = n+1 | \TL(z_1) = \TL_j, \TL(z_2) = \TL_{j+k}, \TL(z_3) > n } \nonumber \\ 
  & = \frac{\mu_{|\TL_{j} - \TL_{j+k}|}((n+1) - \max(\TL_j,\TL_{j+k}))}{\displaystyle \sum_{s \geq (n+1) - \max(\TL_j,\TL_{j+k})} \mu_{|\TL_{j} - \TL_{j+k}|}(s)}\\
  & = \frac{\mu_{|t_{n,j} - t_{n,j+k}|}(1 + \min(t_{n,j},t_{n,j+k}))}{ \sum_{s \geq 1+ \min(t_{n,j},t_{n,j+l}) } \mu_{|t_{n,j} - t_{n,j+k}|}(s)}\\
  & = p_{m}^{\delta}
\end{align}
In this case, $\tilde{\Front}_{n+1}$ gets a local minimum between $j$ and $j+k$ as illustrated in Figure~\ref{fig:dyn2}(a) that corresponds to case 2(a). Else (with probability $1- p_{m}^{\delta}$), $\TL(z_3) > n+1$, and so we obtain case 2(b).

\begin{remark} \label{rem:nono-proof}
  It is exactly, for this proof, that we want the condition $\mu_{\Delta}(0) = 0$. Indeed, if for some $\Delta$ $\mu_\Delta(0) \in (0,1)$, the transition for the Markov chain $(\tilde{\Front}_n)_{n \in \NN}$ becomes much more complicated. Indeed, in that case, for any local maximum that becomes a local minimum, we have to check that the two possible new local maxima created do or do not become local minima in the same time slot, etc. Hence, instead of having a Markov kernel that is understandable for $(\tilde{\Front}_n)_{n \in \NN}$, we would get a very intricate kernel.
\end{remark}

\paragraph{Proof that $(\tilde{\Front}_n)_{n \in \NN}$ is ergodic:} 
Now, to conclude the proof of Theorem~\ref{thm:HMC}, we have to prove that the Markov chain $( \tilde{\Front}_n )_{n \in \NN}$ is ergodic when~\C~\ref{cond:conv} holds.\par

Firstly, $(\tilde{\Front}_n)_{n \in \NN}$ is irreducible because, from any state, the Markov chain can go to $b_i = (-1)^i$ and $t_i = 0$ by applying case 2(a) to any local maximum at each step of time during $\lceil L/2 \rceil$ time steps (or $\lceil L/2 \rceil +1$ time steps). And, from this state, it can go to any other element of $\tilde{\Bridges}_L$ by changing local maximum to local minimum at some precise moments.\par

Secondly, $(\tilde{\Front}_n)_{n \in \NN}$ is aperiodic. Indeed, from $( ((-1)^i)_{i\in \ZZL}, 0^{\ZZL})$, it can come back in two steps of time by going through $( ((-1)^{i+1})_{i\in \ZZL}, 0^{\ZZL})$, or in three steps of time by going through $( ((-1)^i)_{ i\in \ZZL}, 1^{\ZZL})$, then $( ((-1)^{i+1})_{i\in \ZZL}, 0^{\ZZL})$. So period divides $\text{gcd}(2,3) = 1$, so it is $1$.

The last point is obtained by using the Foster criterion, see~\cite[Theorem~1.1, Chapter~5, p.167]{Bremaud99}. The Lyapunov function that we take is $\displaystyle l((b,t)) = \sum_{i \in \ZZL} t_i$. Our finite refuge is
  \begin{equation}
    R_{2L} = \left \{(b,t) : l((b,t)) \leq 2L \left( \frac{2L}{\alpha} + 1 \right) \right \}.
  \end{equation}

  Now, take $\tilde{\Front}_n = (b,t) \notin R_{2L}$, then $\displaystyle \max_{i \in \ZZL} t_i \geq \frac{2L}{\alpha} + 1$. Denote by $j$ one of the index such that $t_i$ is maximum ($j = \text{argmax}_{i \in \ZZL}(t_i)$). By the fact that $\tilde{\Front}_n \in \tilde{\Bridges}_L$, we know that $b$ has ($b_j=1$ and $b_{j+1}=-1$) or ($b_{j-1}=1$ and $b_{j}=-1$). Suppose that it is $b_{j-1} = 1 = -b_j$. Now, by applying the dynamic of the Markov chain, the local maximum between $j-1$ and $j$ becomes a local minimum with probability $p_{t_{j-1}}^{t_{j} - t_{j-1}}$ and so
  \begin{align*}
    \esp{l(\tilde{\Front}_{n+1}) | \tilde{\Front}_n = (b,t)} & \leq (l(\tilde{\Front}_n) + 2L) (1- p_{t_{j-1}}^{t_{j} - t_{j-1}}) + p_{t_{j-1}}^{t_{j} - t_{j-1}} (l(\tilde{\Front}_n) + 2L - t_j - t_{j-1}) \\
                                           & \leq (l(\tilde{\Front}_n) + 2L) - p_{t_{j-1}}^{t_{j} - t_{j-1}} t_{j} \\
                                           & \leq l(\tilde{\Front}_n) + 2L - \alpha \left( \frac{2L}{\alpha} + 1 \right) = l(\tilde{\Front}_n) - \alpha
  \end{align*}
  Hence, $\tilde{\Front}_n$ is ergodic.
\qed

\subsection{Proof of Theorem~\ref{thm:meanspeed}} \label{sec:proofmeanspeed}
\begin{proof}
  We consider the projection $\sigma : \tilde{\Bridges}_L \to \NN$ that is the projection according to $t_1$: for any $(b = (b_i)_{i \in \ZZL}, t = (t_i)_{i \in \ZZL}) \in \tilde{\Bridges}_L$, $\sigma((b,t)) = t_1$. We denote by $\sigma_{\tilde{\nu}_L}$ the law of the random variable $t_1 = \sigma((b,t))$ when $(b,t) \sim \tilde{\nu}_L$. Moreover, for any $y \in \NN$, we denote by $\edgetime_y$ the time $\edgetime(e)$ as defined in~\eqref{eq:edgetime} where $e$ is the edge between the two squares $\left(\frac{1 - (-1)^y}{2},y\right)$ and $\left(\frac{1 - (-1)^{y+1}}{2},y+1\right)$.\par

Now, let's consider the hidden Markov chain $(t_{n,1})_{n \in \NN} = \left(\sigma(\tilde{\Front}_n)\right)_{n \in \NN}$ where $\tilde{\Front}_0$ is taken under its invariant measure $\tilde{\nu}_L$. Under this invariant law, the sequence of times $(\edgetime_j)_{j \in \NN}$ is simply $\left(t_{n_j-1,1}+1\right)_{j \in \NN}$ where $n_j$ is the $j$th time such that $t_{n_j,1} = 0$, i.e.\ $n_j = \min \{n > n_{j-1}: t_{n,1} = 0 \}$. Now, the proof is quite simple. Indeed, we remark that, for any $j \geq 1$, for any $i \leq j-1$,
  \begin{equation}
    \prob{t_{0,1} = i | \edgetime_0 = j} = \frac{\prob{t_{0,1} = i \text{ and } \edgetime_0 = j}}{\sum_{k=0}^{j-1} \prob{t_{0,1} = k \text{ and } \edgetime_0 = j}} = \frac{1}{j}.
  \end{equation}
  That comes from the fact that, for any $j \geq 2$, for any $0 \leq k \leq j-2$, 
  \begin{equation}
    \prob{t_{0,1} = k \text{ and } \edgetime_0 = j} = \prob{t_{1,1} = k+1 \text{ and } \edgetime_0 = j} = \prob{t_{0,1} = k+1 \text{ and } \edgetime_0 = j}.
  \end{equation}
  The first equality comes from the fact that the value of $t_{n,1}$ increases by 1 at each step of time until $n$ reaches one of the $n_j$ where $t_{n_j,1} = 0$ and so on, and the second equality comes because $(t_{n,1})_{n \in \NN}$ is taken under the invariant law.

  Hence, knowing $\edgetime_0$, the law of $t_{0,1}$ is uniform on $\{0,1,\dots,\edgetime_0-1\}$. So
  \begin{equation}
    \esp{t_{0,1} | \edgetime_0} = \frac{\edgetime_0 - 1}{2}.
  \end{equation}
  And, so,
  \begin{equation}
    \esp{t_{0,1}}= \frac{\esp{\edgetime_0} - 1}{2}.
  \end{equation}

  Finally, $c_L = \displaystyle \frac{1}{\esp{\edgetime_0}} =  \frac{1}{2 \esp{t_{0,1}} + 1}$ and $t_{0,1}$ is distributed according to $\sigma_{\tilde{\nu}_L}$. The first equality is obtained by a law of large number for ergodic Markov chain, or the Birkhoff theorem.
\end{proof}

\begin{remark}
  From the ergodicity of $(\tilde{F}_n)_{n \in \NN}$, we can deduce the one of $(t_{n,1})_{n \in \NN}$. Moreover, the step of $(t_{n,1})_{n \in \NN}$ are adding $1$ or returning to $0$, so the mean of the return time to $0$ is $\esp{t_1}$ with $t_1 \sim \sigma_{\tilde{\nu}_L}$. But $(t_{n,1})_{n \in \NN}$ is ergodic, so its return time to $0$ is finite. That's why we can deduce that $\esp{t_1} < \infty$ if $(\tilde{\Front}_n)_{n \in \NN}$ is ergodic.
\end{remark}

\subsection{Proof of Theorems~\ref{thm:int} and~\ref{thm:int-speed}} \label{sec:proofint}
Suppose that $(\mu_\Delta)_{\Delta \in \NN}$ satisfies \C~\ref{cond:int}. Because the dynamic on $(\tilde{\Front}_n)_{n \in \NN}$ is known, see Section~\ref{sec:proofHMC}, we can just check that the conjectured $\tilde{\nu}_L((b,t))$ given by~\eqref{eq:wbt-simple} and~\eqref{eq:conj} is invariant for the dynamic. Suppose that $\tilde{\Front}_n \sim \tilde{\nu}_L$, then, for any $(b,t) \in \tilde{\Bridges}_L$,
\begin{itemize}
\item if, for any $i$, $t_i > 0$, then
  \begin{align}
    & \prob{\tilde{\Front}_{n+1} = (b,t)}\\
    =\ & \prob{\tilde{\Front}_n = (b,(t_1-1,t_2-1,\dots,t_{2L}-1))} \prod_{i:b_i=1,b_{i+1}=-1} 1-\frac{\mu_{|t_i-t_{i+1}|}(\min(t_i,t_{i+1}))}{\sum_{s \geq 0} \mu_{|t_i-t_{i+1}|}(s + \min(t_i,t_{i+1})}\\
    =\ & \frac{1}{Z_L} \left( \prod_{i: b_i=b_{i+1}} \sqrt{\mu_0(|t_{i+1}-t_i|)} \right) \left( \prod_{i: b_i=1,b_{i+1}=-1} \left(\sum_{s > \max(t_i,t_{i+1})} \sqrt{\mu_0(s-t_i)}  \sqrt{\mu_0(s-t_{i+1})}\right)  \right.  \nonumber \\
    & \left. \left( \sum_{s \geq 1} \mu_{|t_i-t_{i+1}|}(\min(t_i,t_{i+1}) - 1 + s) \left(1 - \frac{\mu_{|t_i-t_{i+1}|}(\min(t_i,t_{i+1}))}{\sum_{s \geq 0} \mu_{|t_i-t_{i+1}|}(s + \min(t_i,t_{i+1})}\right)\right) \right)     
  \end{align}
  
  But,
  \begin{equation}
    \left( \sum_{s \geq  0} \mu_{\Delta}(m + s) \right)\left(1 - \frac{\mu_{\Delta}(m)}{\sum_{s \geq 0} \mu_{\Delta}(m + s)}\right) 
    = \sum_{s \geq 1} \mu_{\Delta}(m + s).
  \end{equation}
  Hence, for any $(b,t)$ such that, for any $i$, $t_i> 0$, $\prob{\tilde{\Front}_{n+1} = (b,t)} = \tilde{\nu}_L((b,t))$.
  
\item we suppose that there exists some $i$ such that $t_i=0$. Because $t \in \Time_b$, this implies that
  \begin{enumerate}
  \item there is an even number of such $i$: $I = \{i_1,i_2,\dots,i_{2m}\}$ with $m > 1$,
  \item and we can pair them $\{(i_1,i_2),\dots,(i_{2m-1},i_{2m})\}$ such that, for any $j \in \{1,\dots,m\}$, $i_{2j} = i_{2j-1} +1$ and $b_{i_{2j-1}} = -1$ and $b_{i_{2j}} = 1$.
  \end{enumerate}
  Then,
  \begin{align*}
    & \prob{\tilde{\Front}_{n+1} = (b,t)}\\
    = & \sum_{(s_{i})_{i \in I} \in \NN^I} \prob{\tilde{\Front}_n = ((b_i \ind{i \notin I} - b_i \ind{i \in I})_{i \in \ZZL},((t_i-1) \ind{i \notin I} + s_i \ind{i \in I})_{i \in \ZZL})} \nonumber \\
    & \quad \left( \prod_{i:t_i=t_{i+1} = 0} \frac{\mu_{|s_i-s_{i+1}|}(1 + \min(s_i,s_{i+1}))}{\sum_{s \geq 1 + \min(s_i,s_{i+1})} \mu_{|s_i-s_{i+1}|}(s)}\right) \\
    & \quad \left( \prod_{i:b_i=1,b_{i+1}=-1,t_i \neq 0 \neq t_{i+1}} 1-\frac{\mu_{|t_i-t_{i+1}|}(\min(t_i,t_{i+1}))}{\sum_{s \geq \min(t_i,t_{i+1})} \mu_{|t_i-t_{i+1}|}(s)}\right)
  \end{align*}
  
  Now, we decompose the product according to the 9 different cases illustrated in Figure~\ref{fig:9cases}. Note that cases 1 and 2 are presented both in the same factor (the first one).
  \begin{figure}
    \begin{center}
      \begin{tabular}{|c|c|c||c|c|c|} \hline
        Case & $n+1$ & $n$ & Case & $n+1$ & $n$ \\ \hline
        1 & \includegraphics{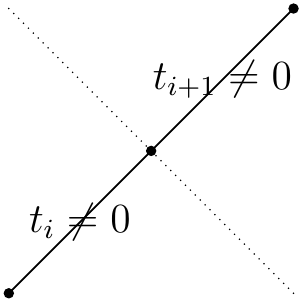} & \includegraphics{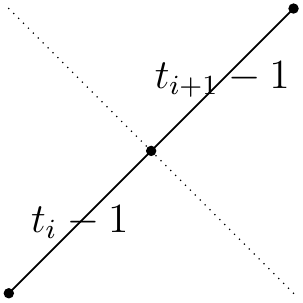} & 4 & \includegraphics{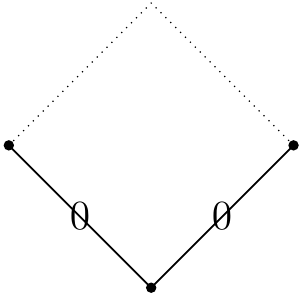} & \includegraphics{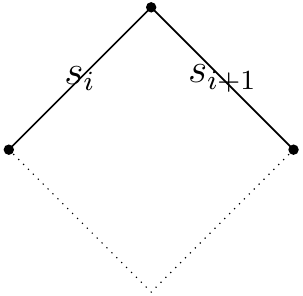}  \\ \hline
        2 & \includegraphics{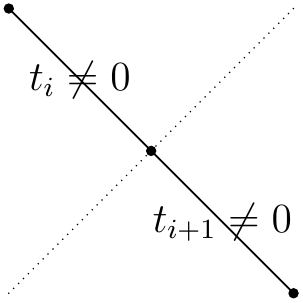} & \includegraphics{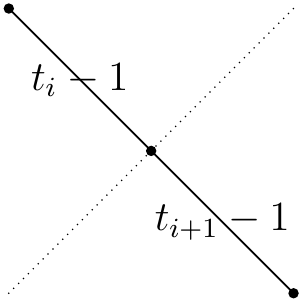} & 5 & \includegraphics{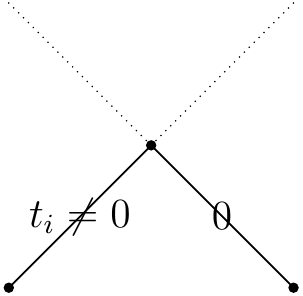} & \includegraphics{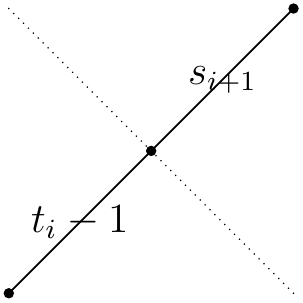}  \\ \hline
        3 & \includegraphics{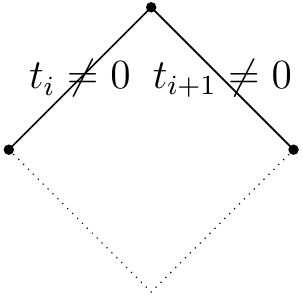} & \includegraphics{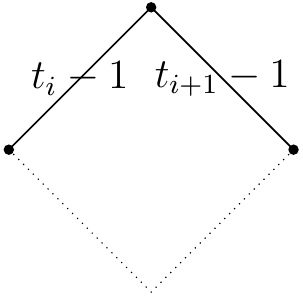} & 6 & \includegraphics{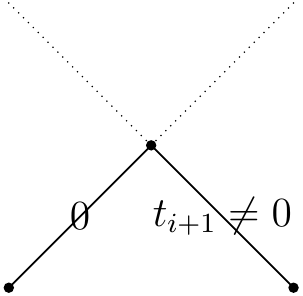} & \includegraphics{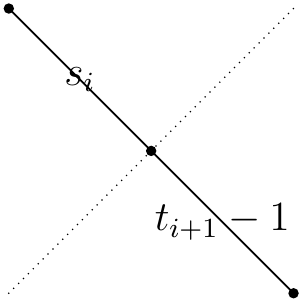}  \\ \hline
      \end{tabular}\newline
      \begin{tabular}{|c|c|c|} \hline
        Case & $n+1$ & $n$ \\ \hline
        7 & \includegraphics{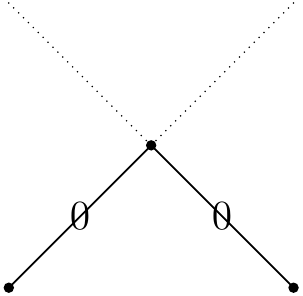} & \includegraphics{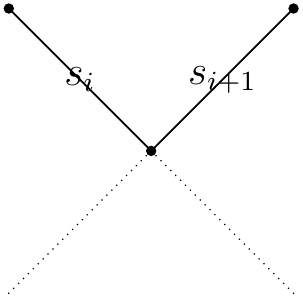} \\ \hline
        8 & \includegraphics{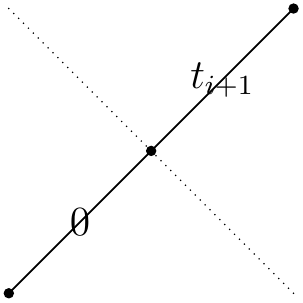} & \includegraphics{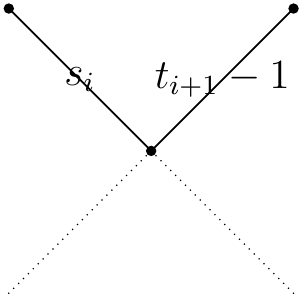} \\ \hline
        9 & \includegraphics{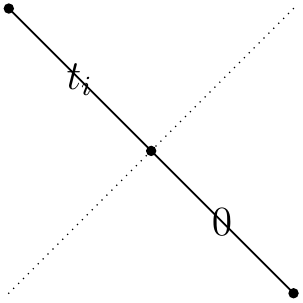} & \includegraphics{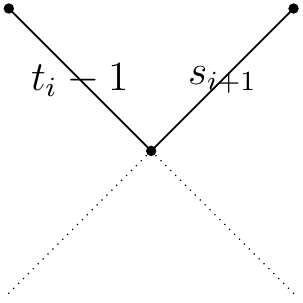} \\ \hline
      \end{tabular}
    \end{center}
    \caption{The 9 different cases.}
    \label{fig:9cases}
  \end{figure}

  \begin{align*}
    & \prob{\tilde{\Front}_{n+1} = (b,t)}\\
    = & \frac{1}{Z_L} \left( \prod_{i: b_i=b_{i+1},t_i \neq 0 \neq t_{i+1}} \sqrt{\mu_0(|t_{i+1}-t_i|)} \right) \\
    & \left[ \prod_{i: b_i=1,b_{i+1}=-1, t_i \neq 0 \neq t_{i+1}} \left(\sum_{s > \max(t_i,t_{i+1})} \sqrt{\mu_0(s-t_i)} \sqrt{\mu_0(s-t_{i+1})}\right)  \right. \\
    &  \quad \left. \left( \sum_{s \geq 1} \mu_{|t_i-t_{i+1}|}(\min(t_i,t_{i+1}) - 1 + s) \left(1 - \frac{\mu_{|t_i-t_{i+1}|}(\min(t_i,t_{i+1}))}{\sum_{s \geq \min(t_i,t_{i+1})} \mu_{|t_i-t_{i+1}|}(s)}\right)\right) \right] \\
    &  \left\{ \sum_{(s_{i})_{i \in I} \in \NN^I} \left( \prod_{i : b_i=b_{i+1}=1, t_i=0 \neq t_{i+1}} \ind{s_i=t_{i+1}-1}  \right) \left( \prod_{i:b_i=b_{i+1}=-1,t_i \neq 0 = t_{i+1}} \ind{s_{i+1}=t_i-1} \right) \right.\\
    & \quad \left(\prod_{b_i = 1, b_{i+1} = -1, t_i \neq 0 = t_{i+1}} \sqrt{\mu_0(s_{i+1}-t_i+1)}\right) \left(\prod_{b_i = 1, b_{i+1} = -1, t_i = 0 \neq t_{i+1}} \sqrt{\mu_0(s_{i}-t_{i+1}+1)}\right) \\
    & \quad \left(\prod_{i:b_i=1,b_{i+1}=-1, t_i=0=t_{i+1}} \ind{s_i=s_{i+1}} \right) \\
    & \quad \left[ \prod_{i: b_i=-1,b_{i+1}=1, t_i = 0 = t_{i+1}} \left(\sum_{s > \max(s_i,s_{i+1})} \sqrt{\mu_0(s-s_i)} \sqrt{\mu_0(s-s_{i+1})}\right)  \right. \\
    &  \quad \quad \left. \left. \left( \sum_{s \geq 1} \mu_{|s_i-s_{i+1}|}(\min(s_i,s_{i+1}) + s) \frac{\mu_{|s_i-s_{i+1}|}(1 + \min(s_i,s_{i+1}))}{\sum_{s \geq 1} \mu_{|s_i-s_{i+1}|}(\min(s_i,s_{i+1})+s)}\right)\right]\right\}
  \end{align*}
  
  Now, in case 4 that is $i$ such that $b_i=-1,b_{i+1}=1, t_i = 0 = t_{i+1}$, by the same computations that have be done to go from~\eqref{eq:deb} to~\eqref{eq:fin},
  \begin{equation}
    \left(\sum_{s > \max(s_i,s_{i+1})} \sqrt{\mu_0(s-s_i)} \sqrt{\mu_0(s-s_{i+1})} \right) \mu_{|s_i-s_{i+1}|}(1 + \min(s_i,s_{i+1}))
    = \sqrt{\mu_0(1+s_i)} \sqrt{\mu_0(1+s_{i+1})}.
  \end{equation}

  Moreover, we remark that any $s_i$ with $i \in I$ must appear twice: once in a case 4, and once and only once between cases $5$, $6$, $7$, $8$ and $9$. So, now, by reordering factor, we obtain 
  \begin{align*}
    & \prob{\tilde{\Front}_{n+1} = (b,t)}\\
    = & \frac{1}{Z_L} \left( \prod_{i: b_i=b_{i+1},t_i \neq 0 \neq t_{i+1}} \sqrt{\mu_0(|t_{i+1}-t_i|)} \right) \\
    & \hspace{-5mm}\left[ \prod_{i: b_i=1,b_{i+1}=-1, t_i \neq 0 \neq t_{i+1}} \left(\sum_{s > \max(t_i,t_{i+1})} \sqrt{\mu_0(s-t_i)} \sqrt{\mu_0(s-t_{i+1})}\right) \left( \sum_{s \geq 1} \mu_{|t_i-t_{i+1}|}(\min(t_i,t_{i+1}) + s) \right) \right] \\
    &  \left\{ \sum_{(s_{i})_{i \in I} \in \NN^I} \left( \prod_{i : b_i=b_{i+1}=1, t_i=0 \neq t_{i+1}} \sqrt{\mu_0(1+s_i)} \ind{s_i=t_{i+1}-1}  \right) \right. \\
    & \quad \left( \prod_{i:b_i= b_{i+1}=-1,t_i \neq 0 = t_{i+1}} \sqrt{\mu_0(1+s_{i+1})} \ind{s_{i+1}=t_i-1} \right) \\
    & \quad \left(\prod_{b_i = 1, b_{i+1} = -1, t_i \neq 0 = t_{i+1}} \sqrt{\mu_0(s_{i+1}-t_i+1)} \sqrt{\mu_0(1+s_{i+1})} \right) \\
    & \quad \left(\prod_{b_i = b_{i+1} = -1, t_i = 0 \neq t_{i+1}} \sqrt{\mu_0(1+s_i)} \sqrt{\mu_0(s_{i}-t_{i+1}+1)} \right) \\
    & \quad \left. \left(\prod_{i:b_i=1,b_{i+1}=-1, t_i=0=t_{i+1}} \sqrt{\mu_0(1+s_i)} \sqrt{\mu_0(1+s_{i+1})} \ind{s_i=s_{i+1}} \right) \right\}
  \end{align*}
  
  And, so,
  \begin{align*}
    & \prob{\tilde{\Front}_{n+1} = (b,t)}\\
    = & \frac{1}{Z_L} \left( \prod_{i: b_i=b_{i+1},t_i \neq 0 \neq t_{i+1}} \sqrt{\mu_0(|t_{i+1}-t_i|)} \right) \\
    & \hspace{-5mm}\left[ \prod_{i: b_i=1,b_{i+1}=-1, t_i \neq 0 \neq t_{i+1}} \left(\sum_{s > \max(t_i,t_{i+1})} \sqrt{\mu_0(s-t_i)} \sqrt{\mu_0(s-t_{i+1})} \right) \left( \sum_{s \geq 1} \mu_{|t_i-t_{i+1}|}(\min(t_i,t_{i+1}) + s) \right) \right] \\  
    &  \left\{ \sum_{(s_{i})_{i \in I} \in \NN^I} \left( \prod_{i : b_i=b_{i+1}=1, t_i=0 \neq t_{i+1}} \sqrt{\mu_0(t_{i+1})} \right) \left( \prod_{i:b_i=b_{i+1}=-1,t_i \neq 0 = t_{i+1}} \sqrt{\mu_0(t_i)} \right) \right. \\
    & \quad \left(\prod_{b_i = 1, b_{i+1} = -1, t_i \neq 0 = t_{i+1}} \sqrt{\mu_0(s_{i+1}-t_i+1)} \sqrt{\mu_0(1+s_{i+1})} \right) \\
    & \quad \left(\prod_{b_i = 1, b_{i+1} = -1, t_i = 0 \neq t_{i+1}} \sqrt{\mu_0(1+s_i)} \sqrt{\mu_0(s_{i}-t_{i+1}+1)} \right) \\
    & \quad \left. \left(\prod_{i:b_i=1,b_{i+1}=-1, t_i=0=t_{i+1}} \mu_0(1+s_i) \right) \right\}
  \end{align*}
  
  Now, we distribute the sum of $s_i$ on each concerning term, 
  \begin{align*}
    & \prob{\tilde{\Front}_{n+1} = (b,t)}\\
    = & \frac{1}{Z_L} \left( \prod_{i: b_i=b_{i+1},t_i \neq 0 \neq t_{i+1}} \sqrt{\mu_0(|t_{i+1}-t_i|)} \right) \\
    & \hspace{-5mm}\left[ \prod_{i: b_i=1,b_{i+1}=-1, t_i \neq 0 \neq t_{i+1}} \left(\sum_{s > \max(t_i,t_{i+1})} \sqrt{\mu_0(s-t_i)} \sqrt{\mu_0(s-t_{i+1})} \right) \left( \sum_{s \geq 1} \mu_{|t_i-t_{i+1}|}(\min(t_i,t_{i+1}) + s) \right) \right] \\  
    &  \left( \prod_{i : b_i=b_{i+1}=1, t_i=0 \neq t_{i+1}} \sqrt{\mu_0(t_{i+1})} \right) \left( \prod_{i:b_i=b_{i+1}=-1,t_i \neq 0 = t_{i+1}} \sqrt{\mu_0(t_i)} \right) \\
    & \left(\prod_{b_i = 1, b_{i+1} = -1, t_i \neq 0 = t_{i+1}} \sum_{s_{i+1} > t_i-1} \sqrt{\mu_0(s_{i+1}-t_i+1)} \sqrt{\mu_0(1+s_{i+1})} \right) \\
    & \left(\prod_{b_i = b_{i+1} = -1, t_i = 0 \neq t_{i+1}} \sum_{s_{i} > t_{i+1}-1} \sqrt{\mu_0(1+s_i)} \sqrt{\mu_0(s_{i}-t_{i+1}+1)} \right) \\
    & \left(\prod_{i:b_i=1,b_{i+1}=-1, t_i=0=t_{i+1}} \sum_{s_i \geq 0} \mu_0(1+s_i) \right)
  \end{align*}
  
  By some changes of variables, we get
  \begin{align*}
    & \prob{\tilde{\Front}_{n+1} = (b,t)}\\
    = & \frac{1}{Z_L} \left( \prod_{i: b_i=b_{i+1},t_i \neq 0 \neq t_{i+1}} \sqrt{\mu_0(|t_{i+1}-t_i|)} \right) \\
    & \hspace{-5mm}\left[ \prod_{i: b_i=1,b_{i+1}=-1, t_i \neq 0 \neq t_{i+1}} \left(\sum_{s > \max(t_i,t_{i+1})} \sqrt{\mu_0(s-t_i)} \sqrt{\mu_0(s-t_{i+1})} \right) \left( \sum_{s \geq 1} \mu_{|t_i-t_{i+1}|}(\min(t_i,t_{i+1}) + s) \right) \right] \\  
    &  \left( \prod_{i : b_i=b_{i+1}=1, t_i=0 \neq t_{i+1}} \sqrt{\mu_0(t_{i+1})} \right) \left( \prod_{i:b_i=b_{i+1}=-1,t_i \neq 0 = t_{i+1}} \sqrt{\mu_0(t_i)} \right) \\
    & \left[\prod_{b_i = 1, b_{i+1} = -1, t_i \neq 0 = t_{i+1}} \left( \sum_{s > t_i} \sqrt{\mu_0(s-t_i)} \sqrt{\mu_0(s)}\right) \underbrace{\left(\sum_{s \geq 1} \mu_{|t_i|}(s) \right)}_{=1}\right] \\
    & \left[\prod_{b_i = 1, b_{i+1} = -1, t_i = 0 \neq t_{i+1}} \left( \sum_{s > t_{i+1}} \sqrt{\mu_0(s)} \sqrt{\mu_0(s-t_{i+1})} \right) \underbrace{\left(\sum_{s \geq 1} \mu_{|t_{i+1}|}(s) \right)}_{=1}\right] \\
    & \left[\prod_{i:b_i=1,b_{i+1}=-1, t_i=0=t_{i+1}} \left(\sum_{s \geq 1} \sqrt{\mu_0(s)} \sqrt{\mu_0(s)} \right) \underbrace{\left(\sum_{s \geq 1} \mu_{0}(s) \right)}_{=1}  \right] \\
    = & \tilde{\nu}_L((b,t))
  \end{align*}
\end{itemize}

That proves that $\tilde{\nu}_L$ is an invariant law of the Markov chain $(\tilde{\Front}_n)_{n \in \NN}$. And, because it is ergodic, we deduce that $\tilde{\nu}_L((b,t)) = \frac{1}{Z_L} W_{(b,t)}$ is the asymptotic law of $(\tilde{\Front}_n)_{n \in \NN}$. And, by projection on the first coordinate, that $\nu_L$ defined by $\nu_L(b) = \frac{1}{Z_L} \sum_{t \in \Time_b} W_{(b,t)}$ is the asymptotic law of $(\Front_n)_{n \in \NN}$ when \C~\ref{cond:int} holds.\qed

Theorem~\ref{thm:int-speed} is now just a corollary of Theorems~\ref{thm:meanspeed} and~\ref{thm:int}.

The main problem with this direct proof is that we do not understand from where the conjectured form~\eqref{eq:conj} comes from as well as~\C~\ref{cond:int}. In the next section, we reveal tools and ideas used to obtain these two formulas.

\section{How to conjecture~\eqref{eq:conj} and~\C~\ref{cond:int}} \label{sec:ideas} 
The main goal of this section is to explain ideas that permit us to conjecture the form~\eqref{eq:conj} of $\tilde{\nu}_L$ and the integrability condition~\C~\ref{cond:int}. For these, we use probabilistic cellular automata  (PCA) and adapt new results of PCA, see~\cite{BGM69,DKT90,CM15,Casse16,CM19} and references therein. Because, in this section, our goal is to establish a conjecture, already proved true on the previous section, we allow us to be sometimes less formal and to skip some proofs if that can clarify the ideas and avoid to lose the reader on some formal details. Nevertheless, we hope that this section enhances the reader by giving it an ``almost true'' alternative proof of Theorem~\ref{thm:int}. 

\subsection{Transformation of $\cyl{L}$ to $\tore{L} \times \NN$}
Before to start, we define a one-to-one transformation $\trans$ from $\tore{L} \times \NN$ to $\cyl{L}$ in the following way:
\begin{equation}
  \trans((x,y)) = (2x+y,y) \text{ and } \trans^{-1}((x,y)) = ((x-y)/2,y).
\end{equation}

This transformation is important because it is more natural to consider the LPP as a law on $\ZZ^{\cyl{L}}$ and the space-time diagram of a PCA as a law on $\ZZ^{\tore{L} \times \NN}$.

\subsection{PCA related to GLPP} \label{sec:PCALPP}
For any $\mu = (\mu_\Delta)_{\Delta \in \NN} \in \mprob{\NN^*}^\NN$, we consider the following PCA $A_{\mu}$ with $E=\ZZ$, $\LL = \tore{L}$, $N(i) = (i,i+1)$ and whose transitions are, for any $a,b,c \in \ZZ$,
\begin{equation} \label{eq:LPPPCA}
  T_{\mu}(a,b;c) = \mu_{|b-a|}(c - \max(a,b)). 
\end{equation}
To $A_\mu$, we associate a law $H_\mu$ on $\ZZ^{\tore{L} \times \NN}$ called its space-time diagram with initial law the Dirac law on $0^{\tore{L}}$: $(\eta((x,y)))_{(x,y) \in \tore{L} \times \NN} \sim H_\mu$ if $(\eta((x,y)) : x \in \tore{L})_{y \in \NN}$ is a Markov chain on $\ZZ^{\tore{L}}$ such that
\begin{itemize}
\item for any $x \in \tore{L}$, $\eta((x,0)) = 0$ and
\item for any $y \in \NN$, $(\eta((x,y+1)))_{x \in \tore{L}}$ is the image of $(\eta((x,y)))_{x \in \tore{L}}$ by $A_\mu$, i.e.\ for any $t = (t_x)_{x \in \tore{L}} \in \ZZ^{\tore{L}}$, 
  \begin{equation}
    \prob{(\eta((x,y+1)))_{x \in \tore{L}} = t} = \sum_{s = (s_x)_{x \in \tore{L}} \in \ZZ^{\tore{L}}} \prob{(\eta((x,y))_{x \in \tore{L}}) = s} \prod_{x \in \tore{L}} T_\mu(s_x,s_{x+1};t_x).
  \end{equation}
\end{itemize}

\begin{lemma} \label{lem:model}
  For any $\mu \in \mprob{\NN^*}^\NN$, denote by $G_\mu$ the law of $(\TL(z))_{z \in \cyl{L}}$, the GLPP on $\cyl{L}$ with parameter $\mu$ as defined in Section~\ref{sec:gLPPdef}. If $(\TL(z))_{z \in \cyl{L}} \sim G_\mu$ and $(\eta(z))_{z \in \tore{L} \times \NN} \sim H_\mu$, then
  \begin{equation}
    \left(\TL(z)\right)_{z \in \cyl{L}} \eqd \left(\eta \left(\trans^{-1}(z)\right)\right)_{z \in \cyl{L}}.
  \end{equation}
\end{lemma}

\begin{proof}
  The proof is done by induction on $y$. When $y=0$, for any $x \in \tore{L}$, $\eta((x,0)) = 0$ because the initial law is the Dirac one on $0^{\tore{L}}$ and $\TL(2x,0) = 0$ by~\eqref{eq:LPPinit}. That ends the case $y=0$.\par
  Now, we suppose that $\left(\eta \left(\trans^{-1}((x,y'))\right)\right)_{(x,y') \in \cyl{L}, y'\leq y} \eqd \left(\TL((x,y'))\right)_{(x,y') \in \cyl{L},y'\leq y}$.

  For any $t =(t_x)_{x \in \ZZL + (y+1 \mod 2)}$,
  \begin{align*}
     & \prob{(\TL((x,y+1)))_{x \in \ZZL + (y+1 \mod 2)}=t | (\TL((x,y')))_{(x,y') \in \cyl{L}, y'\leq y}} \\
    =\ & \prob{(\TL((x,y+1)))_{x \in \ZZL + (y+1 \mod 2) } = t | (\TL((x,y)))_{x \in \ZZL + (y \mod 2)}} \\
    =\ & \prod_{x \in \ZZL + (y+1 \mod 2)} \mu(t_x - \max((\TL(x-1,y-1)),\TL((x+1,y-1)))) & \text{ (cf~\eqref{eq:LPPdyn})} \\
    =\ & \prod_{x \in \ZZL + (y+1 \mod 2)} \TT_{\mu}(\TL((x-1,y-1)),\TL((x+1,y-1));t_x) \\
    =\ & \prod_{x \in \ZZL + (y+1 \mod 2)} \TT_{\mu}(\eta(\trans^{-1}((x-1,y-1))), \eta(\trans^{-1}((x+1,y-1))) ; t_x) \\
    =\ &  \prob{(\eta(\trans^{-1}((x,y+1))))_{x \in \ZZL + (y+1 \mod 2) } = t | (\eta(\trans^{-1}((x,y))))_{x \in \ZZL + (y \mod 2)}} \\
    =\ &  \prob{(\eta(\trans^{-1}((x,y+1))))_{x \in \ZZL + (y+1 \mod 2) } = t | (\eta(\trans^{-1}((x,y'))))_{x \in \ZZL + (y \mod 2),y'\leq y}}.
  \end{align*}
  Hence, by induction, for any $y \in \NN$, $\left(\eta\left(\trans^{-1}((x,y'))\right)\right)_{(x,y') \in \cyl{L}, y'\leq y} \eqd \left(\TL((x,y'))\right)_{(x,y') \in \cyl{L},y'\leq y}$. And then the Kolmogorov's extension theorem concludes the proof.
\end{proof}

\subsection{Integrable PCA} \label{sec:intPCA}
First, remark that PCA related to GLPP could not get an invariant probability measure because their values increase line by line. Nevertheless, we could use recent results and ideas developed in~\cite{DKT90,CM15,Casse16} about PCA whose one invariant probability measure is Markovian. Here, instead of finding invariant \emph{probability} measures, we will look for invariant measures, but \emph{not probabilistic}. Due to that, this section is dedicated to adaptations in that context of previous results that could be found in~\cite{CM15,CM19}.\par

Due to previous works~\cite{CM15,Casse16} about PCA, we focus on measures that have a particular form, introduced here and called cyclic-HZMM (cyclic-Horizontal Zigzag Markovian Measure). Indeed, cyclic-HZMC and HZMC (Horizontal Zigzag Markov Chain) are the only sets of probability measures for which there exist necessary and sufficient conditions that characterise PCA whose one invariant probability measure is in the set. We refer the interested reader in invariant cyclic-HZMC and HZMC of PCA to~\cite{CM15,Casse16,CM19}.
\begin{definition}
  Let $E$ be a discrete set and let $M^+$ and $M^-$ be two Markov kernels from $E$ to $E$ such that $M^+ M^- = M^- M^+$. The measure $\theta_{(M^-,M^+)}$ on $E^{\tore{L}} \times E^{\tore{L}}$ is the cyclic-HZMM with parameter $(M^-,M^+)$ if for any $(s;t) = (s_1,\dots,s_L;t_1,\dots,t_L) \in E^{\tore{L}}\times E^{\tore{L}}$,
  \begin{equation}
    \theta_{(M^-,M^+)}((s;t)) = \prod_{i \in \tore{L}} M^-(s_i;t_i) M^+(t_i;s_{i+1}).
  \end{equation}
\end{definition}  

An illustration is given in Figure~\ref{fig:HZMC}.
\begin{figure}
  $\left. {\begin{minipage}[c]{12cm} \begin{center} \includegraphics{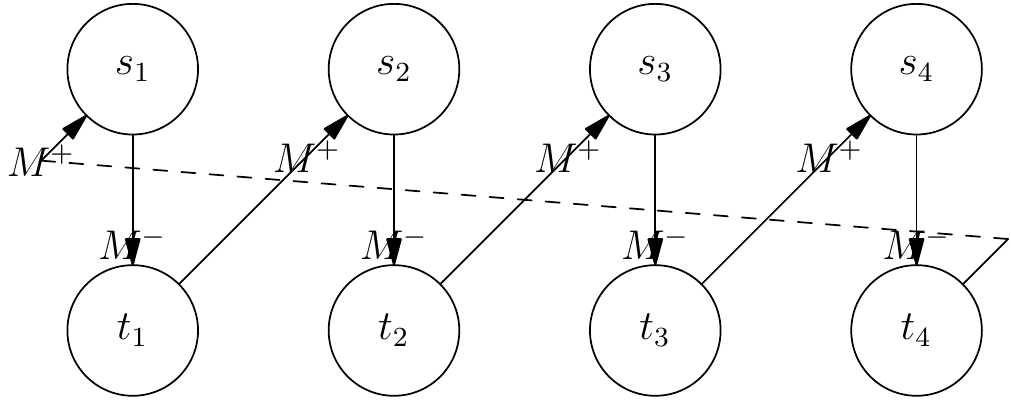} \end{center} \end{minipage}} \right \}$ weight is $\theta_{(M^-,M^+)}$
  \caption{An illustration of the cyclic-HZMM $\theta_{(M^-,M^+)}$ with $L=4$.}
  \label{fig:HZMC}
\end{figure}

\begin{lemma}
  For any $(M^-,M^+)$, $\theta_{(M^-,M^+)}$ is unique and $\theta_{(M^-,M^+)}$ is $\sigma$-finite.
\end{lemma}

\begin{proof}
  The first point is obvious and the second one also because $E$ is discrete.
\end{proof}

\begin{lemma} \label{lem:stable}
  Let $E$ be a discrete set. Let $A$ be a PCA with transition $T$ and $(M^-,M^+)$ be a couple of transition matrices from $E$ to $E$. $\theta_{(M^-,M^+)}$ is an invariant cyclic-HZMM of $A$ iff, for any $s,t,u \in E$, 
  \begin{equation} \label{eq:stable}
    (M^- M^+)(s;t) T(s,t;u) = M^-(s;u)M^+(u;t).
  \end{equation}
\end{lemma}
Such a PCA is called an integrable PCA and we say that $\theta_{(M^-,M^+)}$ is invariant by $A$.
  
\begin{proof}
  Proof of this lemma is the same as the ones of~\cite[Lemma~16.2]{DKT90},~\cite[Theorem~2.3]{CM15} and \cite[Theorem~1]{Casse16}.
\end{proof}

\begin{remarks}
  \begin{itemize}
  \item We could work with any measure proportional to $\theta_{(M^-,M^+)}$ because they are all invariant by $A$.
  \item In Remark~\ref{rem:alpha}, we show an example of two HZMM that are proportional but with different kernels, i.e.\ $\theta_{(M^-,M^+)} = k \theta_{(N^-,N^+)}$ with $(M^-,M^+) \neq (N^-,N^+)$.
  \end{itemize}
\end{remarks}

From~\eqref{eq:stable}, we obtain the following necessary condition: for any $s,t,u,s',t',u' \in E$ such that $(M^-M^+)(s;t) > 0$, $(M^-M^+)(s;t') > 0$, $(M^-M^+)(s';t) > 0$ and $(M^-M^+)(s';t') > 0$,
\begin{equation} \label{eq:Belyaev}
  T(s,t;u) T(s',t';u) T(s',t;u') T(s,t';u') = T(s',t';u') T(s,t;u') T(s,t';u) T(s',t;u).
\end{equation}
This condition is a very well-known condition in the integrable PCA literature. It was found first in~\cite{BGM69} when $E = 2$, and extended for any finite $E$ in~\cite{CM15}, and for any Polish space $E$ in~\cite{Casse16}.\par
\medskip

We finish this section in a very informal way. Indeed, we use notations as if we manipulate probabilistic measures whereas we are manipulating $\sigma$-finite measures that are not probabilistic.

First, we need to adapt the definition of the space-time diagram of a PCA (defined in Section~\ref{sec:PCALPP}) to see it as a $\sigma$-finite measure on $E^{\cyl{L}}$, but not necessarily probabilistic. Let $A$ be a PCA and $\theta$ be any $\sigma$-finite measure on $E^{\tore{L}} \times E^{\tore{L}}$, the space-time diagram of $A$ under its initial measure $\theta$ is the (formally, we should say ``a'' because uniqueness is not proved) measure $H_{A,\theta}$ on $E^{\tore{L} \times \NN}$ such that if $(\eta((x,y)))_{(x,y) \in \tore{L} \times \NN} \sim H_{A,\theta}$ then, for any $y \geq 1$, the measure of $(\eta((x,y')))_{x \in \tore{L}, 0 \leq y' \leq y}$ is
\begin{equation} \label{eq:HAT}
  \theta( (\eta((x,0)))_{x \in \tore{L}} , (\eta((x,1)))_{x \in \tore{L}} ) \prod_{y' =1}^{y-1} \prod_{x \in \tore{L}} T( \eta((x,y')), \eta((x+1,y')); \eta((x,y'+1))).
\end{equation}

In the following, we are mostly interested when $A$ is an integrable PCA and $\theta$ is its invariant cyclic-HZMM. Indeed, in that case, we are able to give the (non probabilistic) measure of times on any bridge. For any $b = (b_i)_{i \in \ZZL} \in \Bridges_L$ and $z = (x,y) \in \tore{L} \times \NN$, the bridge $b$ with origin $z$, denoted by $\BB_{(b,z)}$, is the sequence of vertices $(x_i,y_i)_{i \in \ZZL}$ such that $(x_1,y_1) = (x,y)$ and, for any $i \in \ZZL$,
\begin{equation}
  (x_{i+1},y_{i+1}) = \begin{cases}
    (x_i,y_i+1) & \text{if } b_i = -1;\\
    (x_i+1,y_i-1) & \text{if } b_i = +1.
  \end{cases}
\end{equation}
Note that we need a condition on $b$ and $z = (x,y)$ to get $\BB_{(b,z)}$ entirely contained on $\tore{L} \times \NN$. The condition is, for any $j$, $y - \sum_{i=1}^j b_i \geq 0$.

\begin{remark}
  Let $A$ be an integrable PCA and $\theta_{(M^-,M^+)}$ be one of its invariant cyclic-HZMM. Let $\eta = (\eta_z)_{z \in \ZZ/L\ZZ \times \NN} \sim H_{A,\theta_{(M^-,M^+)}}$. For $b \in \Bridges_L$ and any $z = (x,y) \in \tore{L} \times \NN$ such that, for any $j$, $y - \sum_{i=1}^j b_i \geq 0$, for any $(t_i)_{i \in \ZZL} \in \ZZ^{\ZZL}$,
  \begin{equation} \label{eq:HAT-bridges}
    H_{A,\theta_{(M^-,M^+)}} \left( \left\{(\eta_{z'})_{z'\in \BB_{(b,z)}} = (t_i)_{i \in \ZZL}\right\} \right) = \prod_{i:b_i=1}M^+(t_i,t_{i+1}) \prod_{i:b_i=-1}M^-(t_i,t_{i+1}).
  \end{equation}
  This is illustrated in Figure~\ref{fig:flip}. This remark is the counterpart of~\cite[Proposition~27]{CM19} when we consider $\sigma$-finite measures instead of probabilistic measures.
\end{remark}

\begin{figure}
  \begin{center} \includegraphics{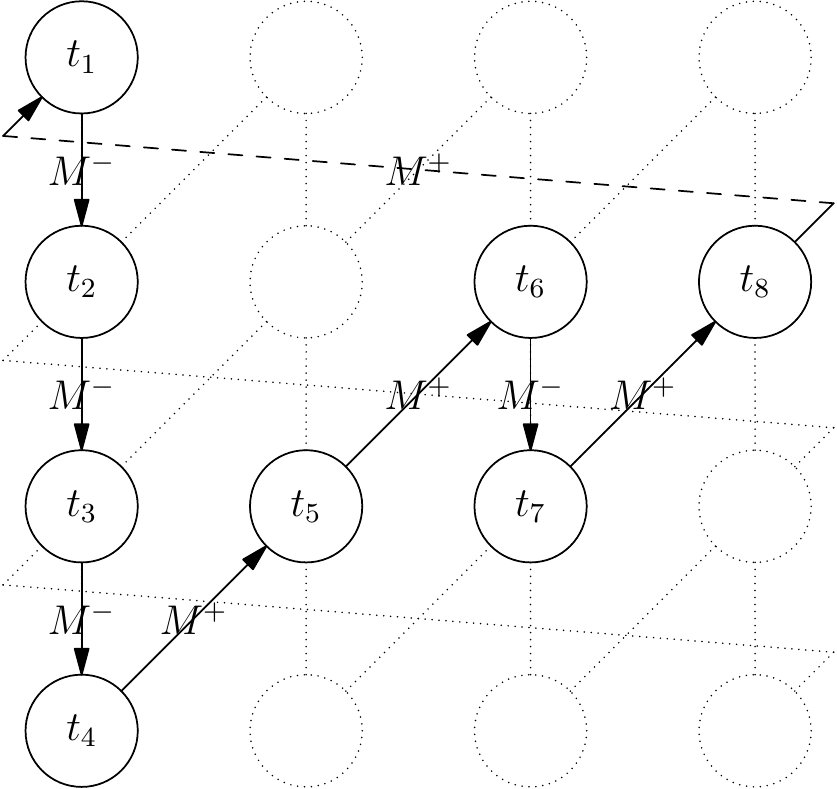} \end{center}
  \caption{The measure to get $(t_1,\dots,t_8)$ along a line $(-1,-1,-1,1,1,-1,1,1)$ is $M^-(t_1,t_2)M^-(t_2,t_3)M^-(t_3,t_4)M^+(t_4,t_5)M^+(t_5,t_6)M^-(t_6,t_7)M^+(t_7,t_8)M^+(t_8,t_1)$.}
  \label{fig:flip}
\end{figure}

\subsection{Integrable GLPP} \label{sec:intLPP}
First, we explain from where the integrable condition~\C~\ref{cond:int} on $(\mu_{\Delta})_{\Delta \in \NN}$ comes.
For PCA related to GLPP (see~\eqref{eq:LPPPCA}), the equation~\eqref{eq:Belyaev} implies (by taking $s=t=0 \leq s'=t'= \Delta < u = \Delta +1 \leq u'= \Delta + v$): for any $\Delta\in \NN$, $v \in \NN^*$,
\begin{equation} 
  \mu_{0}(\Delta+1) \mu_0(1) \mu_{\Delta}(v) \mu_{\Delta}(v) = \mu_0(v) \mu_0(\Delta+v) \mu_\Delta(1) \mu_\Delta(1).
\end{equation}
In particular, for any $\Delta \in \NN$, $v \in \NN^*$,
\begin{equation}
  \mu_{\Delta}(v) = \frac{\mu_\Delta(1)}{\sqrt{\mu_0(1) \mu_0(\Delta+1)}} \sqrt{\mu_0(v) \mu_0(\Delta+v)},
\end{equation}
and using the fact that $\displaystyle \sum_{v \in \NN^*} \mu_{\Delta}(v) = 1$, we obtain~\C~\ref{cond:int}.

Now let us find one couple $(M^-,M^+)$ that is associated to an integrable GLPP.
\begin{proposition} \label{prop:FB}
  Let $\mu_0 \in \mprob{\NN^*}$ be such that $\sum_{v \in \NN^*} \sqrt{\mu_0(v)} < \infty$. Let's define $(\mu_\Delta)_{\Delta \in \NN}$ by~\C~\ref{cond:int}. Let $A_\mu$ be the integrable PCA related to the GLPP with parameter $\mu$, see~\eqref{eq:LPPPCA}. Let's define $(M^-,M^+)$ such that, for any $s,t,u \in \ZZ$,
  \begin{equation} \label{eq:MM}
    M^-(s;u) = \frac{\sqrt{\mu_0(u-s)}}{\sum_{v \in \NN^*} \sqrt{\mu_0(v)}} \text{ and }
    M^+(u;t) = \frac{\sqrt{\mu_0(u-t)}}{\sum_{v \in \NN^*} \sqrt{\mu_0(v)}}.
  \end{equation}
  Then, $\theta_{(M^-,M^+)}$ is an invariant measure of $A_{\mu}$.
\end{proposition}

\begin{proof}
  To prove that it is invariant, we use Lemma~\ref{lem:stable}. We consider the case $s \leq t$, the case $s \geq t$ is similar. For any $s \leq t < u$,
  \begin{align*}
    & (M^- M^+)(s;t) T(s,t;u) \\
    =\ & \left( \sum_{u > t} \frac{\sqrt{\mu_0(u-s)}}{\sum_{v \geq 1} \sqrt{\mu_0(v)}} \frac{\sqrt{\mu_0(u-t)}}{\sum_{v \geq 1} \sqrt{\mu_0(v)}} \right) \mu_{t-s}(u - t) \\
    =\ & \frac{\sum_{v \geq 1} \sqrt{\mu_0(v + t - s) \mu_0(v+t-t)}}{\left( \sum_{v \geq 1} \sqrt{\mu_0(v)} \right) \left(\sum_{v \geq 1} \sqrt{\mu_0(v)}\right)} \frac{\sqrt{\mu_0(u-s) \mu_0(u-t)}}{\sum_{v \geq 1} \sqrt{\mu_0(v) \mu_0(v + t -s)}} \\
    =\ & \frac{\sqrt{\mu_0(u-s)}}{\sum_{v \geq 1} \sqrt{\mu_0(v)}} \frac{\sqrt{\mu_0(u-t)}}{\sum_{v \leq 1} \sqrt{\mu_0(v)}}\\
    =\ & M^-(s;u) M^+(u;t).
  \end{align*}
  The fact that $M^- M^+ = M^+ M^-$ is a similar computation.
\end{proof}

Here, because $M^-(s,u) = M^-(s+v,u+v)$ and $M^+(u,t) = M^+(u+v,t+v)$, we can express this invariant measure $\theta_{(M^-,M^+)}$ accordingly to the counting measure on $\ZZ$ and the following probability measure $\theta'_{(M^-,M^+)}$ on $\ZZ^{L-1} \times \ZZ^L$:
\begin{equation}
  \theta'_{(M^-,M^+)}((s_2,s_3,\dots,s_L;t_1,t_2,\dots,t_L)) = \frac{1}{Z} M^+(t_L;0) M^-(0;t_1) \prod_{i=2}^L M^+(t_{i-1};s_i) M^-(s_i;t_i)
\end{equation}
where $\displaystyle Z = \sum_{\ZZ^{L-1} \times \ZZ^L} M^+(t_L;0) M^-(0;t_1) \prod_{i=2}^L M^+(t_{i-1};s_i) M^-(s_i;t_i)$, so
\begin{equation}
  \theta_{(M^-,M^+)}((s_1,\dots,s_L;t_1,\dots,t_L)) = Z \sum_{u \in \ZZ}\ind{s_1=u} \theta'_{(M^+,M^-)}((s_2-u,s_3-u,\dots,s_L-u;t_1-u,t_2-u,\dots,t_L-u)).
\end{equation}

\begin{remark}
  The value $Z \leq 1$ and so finite. Indeed, for any $t_L \in \ZZ$, $M^+(t_L;0) \leq 1$, so
  \begin{align*}
    & \sum_{\ZZ^{L-1} \times \ZZ^L} M^+(t_L;0) M^-(0;t_1) \prod_{i=2}^L M^+(t_{i-1};s_i) M^-(s_i;t_i) \\
    \leq & \sum_{\ZZ^{L-1} \times \ZZ^L} M^-(0;t_1) \prod_{i=2}^L M^+(t_{i-1};s_i) M^-(s_i;t_i) \\
    = & \sum_{t_1 \in \ZZ}  M^-(0;t_1) \sum_{s_2 \in \ZZ} M^+(t_1;s_2) \dots \underbrace{\sum_{s_L \in \ZZ} M^+(t_{L-1};s_L) \underbrace{\sum_{t_L \in \ZZ} M^-(s_L;t_L)}_{= 1}}_{=1} \\
    = & 1.
  \end{align*}
\end{remark}

With $\theta'_{(M^-,M^+)}$, we can give another expression of $c_L$. 
Indeed, the probability that $\TL((1,1)) = t_1$ under the ergodic measure when $\TL((0,0)) = 0$ is
\begin{equation}
  \frac{1}{Z} \sum_{s_2,\dots,s_n,t_2,\dots,t_n \in \ZZ} \theta'_{(M^-,M^+)}((s_2,s_3,\dots,s_L;t_1,t_2,\dots,t_L))
\end{equation}
and, by~\eqref{eq:MM},
\begin{equation}
  \theta'_{(M^-,M^+)}((s_2,s_3,\dots,s_L;t_1,t_2,\dots,t_L)) = \frac{\sqrt{\mu_0(t_1) \left( \prod_{i=2}^L \mu_0(t_{i-1} - s_i) \mu_0(t_i - s_i) \right) \mu_0(t_L)}}{Z \left(\sum_{s \geq 1} \sqrt{\mu_0(s)}\right)^{2L}}.
\end{equation}
From these two equations, we can deduce the value of
\begin{equation}
  \esp{\edgetime_0} = \sum_{t_1 \in \NN} \frac{t_1}{Z'} \sqrt{\mu_0(t_1) \left( \prod_{i=2}^L \mu_0(t_{i-1} - s_i) \mu_0(t_i - s_i) \right) \mu_0(t_L)}
\end{equation}
where $Z' = \sum_{(s_2,s_3,\dots,s_L;t_1,\dots,t_L) \in \ZZ^{L-1} \times \ZZ^L} \sqrt{\mu_0(t_1) \left( \prod_{i=2}^L \mu_0(t_{i-1} - s_i) \mu_0(t_i - s_i) \right) \mu_0(t_L)}$ and $\edgetime_0$ is the same notation as the one used on Section~\ref{sec:proofmeanspeed}, and then
\begin{equation} \label{eq:int-speed-alt}
  \meanspeed_L = \frac{\displaystyle \sum_{t_1,s_2,t_2,\dots,s_L,t_L \in \ZZ} \sqrt{\mu_0(t_1) \left( \prod_{i=2}^L \mu_0(t_{i-1} - s_i) \mu_0(t_i - s_i) \right) \mu_0(t_L) } }{\displaystyle \sum_{t_1,s_2,t_2,\dots,s_L,t_L \in \ZZ} t_1 \sqrt{\mu_0(t_1) \left( \prod_{i=2}^L \mu_0(t_{i-1} - s_i) \mu_0(t_i - s_i) \right) \mu_0(t_L) }}.
\end{equation}
This gives a different expression of $c_L$ as a function of $\mu_0$ than the one of~\eqref{eq:int-speed}.\par
\medskip

\begin{remark} \label{rem:alpha}
  For any $\alpha \in (0,\infty)$, if we define, for any $s,t \in \ZZ$ with $t > s$,
  \begin{equation}
    M^-_\alpha(s;t) = \frac{\sqrt{\mu_0(t-s)} \alpha^{t-s}}{\sum_{v \in \NN^*} \sqrt{\mu_0(v)} \alpha^{v}} \text{ and } M^+_\alpha(t;s)  = \frac{\sqrt{\mu_0(t-s)} \alpha^{-(t-s)}}{\sum_{v \in \NN^*} \sqrt{\mu_0(v)} \alpha^{-v}}.
  \end{equation}
  We can check that, for any $\alpha \in (0,\infty)$, $\theta_{(M^-_\alpha,M^+_\alpha)}$ and $\theta_{(M^-_{1/2},M^+_{1/2})}$ are proportional and are invariant measures of $A_{\mu}$.\par
  This remark is not important for the GLPP on the cylinders, but more important for the study of the GLPP on the half-plane. In that case, $\alpha$ parameterises some of the invariant probability measures invariant by translation (the parameter $\alpha$ is then related to the mean slope of the front line), and even maybe all of them. Nowadays, we are not able to answer the last remark, because there exist few works about ergodicity of PCA and, in particular, nothing about the one of that kind of PCA. We suggest the reading of~\cite{Vasilyev78},~\cite{BMM13} and~\cite{DLR02} where one can find the three leading ideas about ergodicity of PCA.\par
  Parameterising invariant measures with $\alpha$ could also play a role to study the GLPP on the quarter-plane.
\end{remark}

\subsection{From PCA to $\nu_L$} \label{sec:PCAnu}
Now, we are about to conclude about explanations of how we have conjectured formula~\eqref{eq:wbt}. Let $\mu_0 \in \mprob{\NN^*}$ be such that~\C~\ref{cond:conv} holds. Define $\mu_{\Delta}$ by~\C~\ref{cond:int}, $T_{\mu}$ by~\eqref{eq:LPPPCA}, and $M^-$ and $M^+$ by~\eqref{eq:MM}. Consider the measure $H_{A,\theta_{(M^-,M^+)}}$ on $\ZZ^{\tore{L} \times \NN}$ defined by~\eqref{eq:HAT}.\par
Let $(b,t) \in \tilde{\Bridges}_L$. In this section, we look at the ``probability'' that $\tilde{F}_n$, the front line at time $n$, is $(b,t)$ on $\left(\TL(z) \right)_{z \in \cyl{L}} = \left( \eta(\trans^{-1}(z)) \right)_{z \in \cyl{L}}$ when the measure of $\left(\eta(z)\right)_{z \in \tore{L} \times \NN}$ is $H_{A,\theta_{(M^-,M^+)}}$ and when this front line is contained on $\cyl{L}$.\par

To do that, just consider what happens for the PCA. For the PCA (so on $\eta$ under $H_{A,\theta_{(M^-,M^+)}}$), that consists of having one of its bridges $B_{(b,(0,y))} = ((x_1,y_1)=(0,y),(x_2,y_2),\dots,(x_{2L},y_{2L}))$ with the following properties: for any $i \in \ZZL$,
\begin{enumerate}
\item $\eta((x_i,y_i)) = n - t_i$, if $b_i=1$,
\item $\eta((x_i,y_i)) = n - t_{i-1}$, if $b_{i-1}=-1$,
\item $\eta((x_i,y_i)) < \min(n-t_{i-1},n-t_i)$, if $b_{i-1}=1,b_i=-1$, and
\item $\eta((x_i,y_i) + (0,1)) > n$.
\end{enumerate}
But, the last condition is equivalent to: for $i$ such that $(b_{i-1}=1,b_i=-1)$, $\eta((x_i,y_i) + (0,1)) > n$. Indeed, the GLPP construction implies that if it is true for all $i$ such that $(b_{i-1}=1,b_i=-1)$, it is then true for any $i$.

Now, by~\eqref{eq:HAT-bridges}, the measure of a bridge that satisfies conditions 1, 2 and 3 is
\begin{align}
  & \left( \prod_{i: b_i=b_{i+1}} \frac{\sqrt{\mu_0(|(n-t_{i+1})-(n-t_i)|)}}{\sum_{s \geq 1} \sqrt{\mu_0(s)}} \right) \nonumber \\
  & \quad \left( \prod_{i: b_i=1,b_{i+1}=-1} \left(\sum_{u \in \ZZ, u < \min(n-t_i,n-t_{i+1})} \frac{\sqrt{\mu_0((n-t_i)-u)}}{\sum_{s \geq 1} \sqrt{\mu_0(s)}} \frac{\sqrt{\mu_0((n-t_{i+1})-u)}}{\sum_{s \geq 1} \sqrt{\mu_0(s)}}\right)\right) \label{eq:red}
\end{align}
and to add condition 4, we multiply by
\begin{equation} \label{eq:blue}
  \prod_{i : b_i=1,b_{i+1}=-1} \sum_{s > n} T(n-t_i,n-t_{i+1};s) = \prod_{i : b_i=1,b_{i+1}=-1} \sum_{v > n} \mu_{|t_i-t_{i+1}|}(v - \max(n-t_i,n-t_{i+1})).
\end{equation}

By simplification and by the changes of variables $u' = n-u$ and $v' = v-n$ in their respective sums, we obtain the conjectured formula~\eqref{eq:wbt}.

This is illustrated in Figure~\ref{fig:explanation}.

\begin{figure}
  \includegraphics{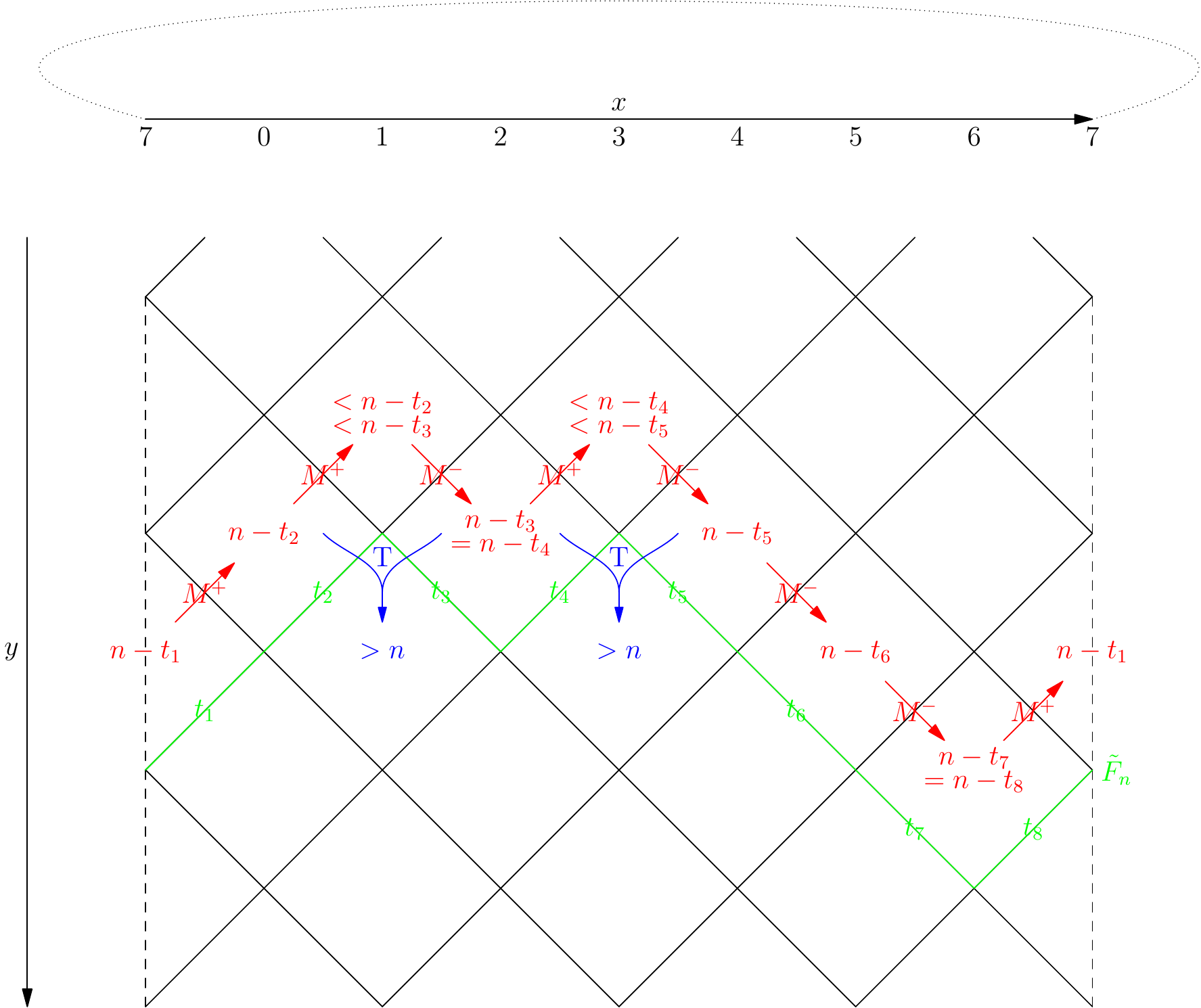}
  \caption{This figure illustrates Section~\ref{sec:PCAnu}. To get the green line, we need to get all the conditions in red (conditions 1, 2 and 3 of Section~\ref{sec:PCAnu}) and blue (condition 4 of Section~\ref{sec:PCAnu} for $i$ such that $(b_{i-1}=1,b_i=-1)$). Getting conditions in red is given by~\eqref{eq:red} and in blue by~\eqref{eq:blue} under the measure $H_{A,\theta_{(M^-,M^+)}}$.\newline Remember that, by the definition of $\tilde{\Bridges}_L$, $t_3=t_4$ and $t_7=t_8$.} \label{fig:explanation}
\end{figure}

\section{GLPP on cylinders: continuous time} \label{sec:cont}
This section is dedicated to GLPP in continuous time. We explain the main difference with the discrete time. In particular, we give some few sufficient conditions (not optimal in general) on the sequences $(\mu_\Delta)_{\Delta \in \RR_+}$ such that the GLPP is well defined and such that the front line $(\tilde{F}_n)_{n \in \RR_+}$ is a non-explosive Markov process. Then, we establish properties of this front line as we have done in the discrete time case.\par
We suppose that, for any $\Delta \in \RR_+$, $\mu_\Delta$ is absolutely continuous according to the Lebesgue measure on $\RR_+^*$. In particular, this implies that we do not consider measures with atoms. Moreover, for any $\Delta \in \RR_+$, the density of $\mu_\Delta$ on $\RR_+^*$ is denoted by $f_\Delta$ and we impose that, for any $\Delta \in \RR_+$, $x \in \RR_+^*$, $f_{\Delta}(x) > 0$ and that $f_{\Delta}$ is $\mathcal{C}_1$, that means it is differentiable and its derivative is continuous.\par 
The first point is that the definition of GLPP is strictly the same as the one in Section~\ref{sec:gLPPdef}, except that we consider $(\mu_{\Delta})_{\Delta \in \RR_+} \in \mprob{\RR_+^*}^{\RR_+}$ instead of $\mprob{\NN^*}^\NN$ and the time $n$ is no more discrete but continuous, so, now, $n \in \RR_+$. Just to be sure, we give quickly the new definition of $\tilde{\Bridges}_L$ in that context. For any $b \in \Bridges_L$, we define the set
\begin{align} 
  \Time_b = \{(t_i & : i \in \ZZL) \in (\RR_+)^{\ZZL} : \nonumber \\
                   & \text{if } b_i = b_{i+1} = 1, \text{then } t_i<t_{i+1}; \nonumber \\
                   & \text{if } b_i = b_{i+1} = -1, \text{then } t_i>t_{i+1}; \nonumber \\
                   & \text{if } b_i = -1 \text{ and } b_{i+1} = 1, \text{then } t_i=t_{i+1}\} \label{eq:timeset}.
\end{align}
Then, $\tilde{\Bridges}_L = \{(b,t) : b \in \Bridges_L, t\in \Time_b\}$.
\par

\begin{lemma}
  Let $(\mu_\Delta)_{\Delta \in \RR_+}$. The process $(\tilde{\Front}_n)_{n \in \RR_+}$ is a Markov process on $\tilde{\Bridges}_L$. 
\end{lemma}

\begin{proof}
  The dynamic of $(\tilde{\Front}_n)_{n \in \RR_+}$ has two components: one deterministic and continuous and one that is a jump process on local maximum. Hence,
  \begin{itemize}
  \item for any $j \in \cyl{L}$, $\di{b_{n,j}} = 0$ and $\di{t_{n,j}} = \di{n}$ and, 
  \item for $j \in \cyl{L}$ such that $b_{n,j} = 1$ and $b_{n,j+1} = -1$ (i.e. $b_n$ has a local maximum between $j$ and $j+1$), then $(b_{n,j},b_{n,j+1},t_{n,j},t_{n,j+1})$ jumps to  the value $(-1,1,0,0)$ at rate $\displaystyle \beta_{m}^{\delta} = \frac{f_\delta(m)}{\int_{\RR_+^*} f_\delta(s + m) \di s}$ where $m = \min \left( t_{n,j},t_{n,j+1} \right)$ and $\delta= \left| t_{n,j}-t_{n,j+1} \right|$. \qedhere
  \end{itemize}
\end{proof}


This dynamics corresponds to the following \emph{generator} $G$ on $\mathcal{C}^{1} \left(\tilde{\Bridges}_L \right)$ (the set of function of class $C^1$ from $\tilde{\Bridges}_L$ to $\RR$): for any $h \in \mathcal{C}^1 \left(\tilde{\Bridges}_L \right)$, any $(b,t) \in \tilde{\Bridges}_L$,
\begin{equation} \label{eq:generator}
(Gh)((b,t)) = \frac{\partial h}{\partial \vec{u}}((b,t)) + \sum_{i : b_i=1,b_{i+1}=-1} \frac{f_{\delta}(m)}{\displaystyle \int_{s \geq 0} f_{\delta}(m+s) \di s} \left[ h((b^{(i)},t^{(i)})) - h((b,t)) \right]
\end{equation}
where $\vec{u} = (0^{2L},1^{2L})$, $\delta = |t_{i+1}-t_i|$, $m = \min(t_i,t_{i+1})$, $b^{(i)} = (b_1,\dots,b_{i-1},-1,1,b_{i+2},\dots,b_{2L})$ and $t^{(i)} = (t_1,\dots,t_{i-1},0,0,t_{i+2},\dots,t_{2L})$.

The first difficulty that could not happen in discrete time is that we have to check that the process $(\tilde{\Front}_n)_{n \in \RR_+}$ does not explode. The explosion of the process is here an infinite number of jumps in a finite time that implies an infinite asymptotic mean speed $c_L = \infty$. The notion of explosion is different from the usual one on Markov processes (see~\cite{MT93-part3}). Indeed, the usual one is when the process goes to infinity in a finite time, whereas here this kind of explosion is not possible because times on edges grow at most linearly. And, reciprocally, when $(\tilde{\Front}_n)$ explodes in the GLPP context, it does not explode in the usual context because $(\tilde{\Front}_n)_{n \in \RR_+}$ is then in a compact set of the form $\{(b,t) : b \in \Bridges_L, \forall i, 0 \leq t_i \leq \epsilon\} \cap \tilde{\Bridges}_L$ for a small $\epsilon$.\par 

A sufficient condition to avoid the explosion is the following one:
\begin{cond} \label{cond:noexplosion}
  there exists $\alpha \in (0,1)$, $\epsilon > 0$ such that
  \begin{equation}
    \sup_{\Delta \in \RR_+} \{ \mu_{\Delta}([0,\epsilon]) \} \leq \alpha.
  \end{equation}
\end{cond}
  
\begin{lemma} \label{lem:noexplosion}
  Let $(\mu_\Delta)_{\Delta \in \RR_+} \in \mprob{\RR_+^*}^{\RR_+}$ be such that~\C~\ref{cond:noexplosion} holds. Then, $(\tilde{\Front}_n)_{n \in \RR_+}$ does not explode.
\end{lemma}
This following condition is probably very far from optimal, in particular, it is uniform on $\Delta$, whereas an optimal condition should consider dependence on it.\par
  
\begin{proof}
  The idea of the proof is to bound the mean number of squares that can arrive during $\epsilon$ units of time. To bound it, we use a coupling between our GLPP with parameter $\mu$ such that \C~\ref{cond:noexplosion} holds and the product measure (with Bernouilli's random variables of parameter $1-\alpha$) on sites. An illustration of what could happen on any time interval of size $\epsilon$ is given in Figure~\ref{fig:proofNoExplosion}.\par
  Let $(\mu_\Delta)_{\Delta \in \RR_+}$ be such that~\C~\ref{cond:noexplosion} holds. Now define, for each site $z$ in $\cyl{L}$, the random variable $w_z$ such that $w_z$ is a Bernoulli variable of parameter $1-\alpha$ (i.e.\ $\prob{w_z = 1} = 1- \alpha$) and $(w_z)_{z \in \cyl{L}}$ are independent. Now, by a coupling, it is easy to see that any square $z$ such that $w_z=1$ waits at least a time $\epsilon$ to come in our GLPP.\par
  Now, we choose one square $z_1$ such that $w_{z_1}=1$. We wait $\epsilon$ units of time before it comes. During this duration, there is at most $(L-1)^2$ squares that can arrive before it arrives: in Figure~\ref{fig:proofNoExplosion} it corresponds to the number of squares between the two red lines $R_1$ and $R_2$. We bound this number of squares by $(2L-3)L$ that is the number of squares between the two blue lines $B_1$ and $B_2$ in Figure~\ref{fig:proofNoExplosion}.\par 
  When the square $z_1=(x_1,y_1)$ has arrived, full lines of squares with $w_z=0$ (in green in Figure~\ref{fig:proofNoExplosion}) can arrive until we reach another square $z_2=(x_2,y_2)$ such that $w_{z_2}=1$ and $y_2-y_1>L$. But this number of lines is a geometric random variable, of success parameter $1-\alpha^L \neq 0$, whose mean is $\alpha^{-L} < \infty$. So we get in mean $\alpha^{-L}L$ squares between the two blue lines $B_2$ and $B_3$ in Figure~\ref{fig:proofNoExplosion}.\par
  Hence, during any interval of time of size $\epsilon$, there are less than the number of squares between the two red lines $R_1$ and $R_4$ in Figure~\ref{fig:proofNoExplosion} that can arrive, that is less than the number of squares between the two blue lines $B_1$ and $B_4$ whose mean number is $2L(2L-3) + L \alpha^{-L}$ finite. \qedhere
  
  \begin{figure}
    \begin{center}
      \includegraphics{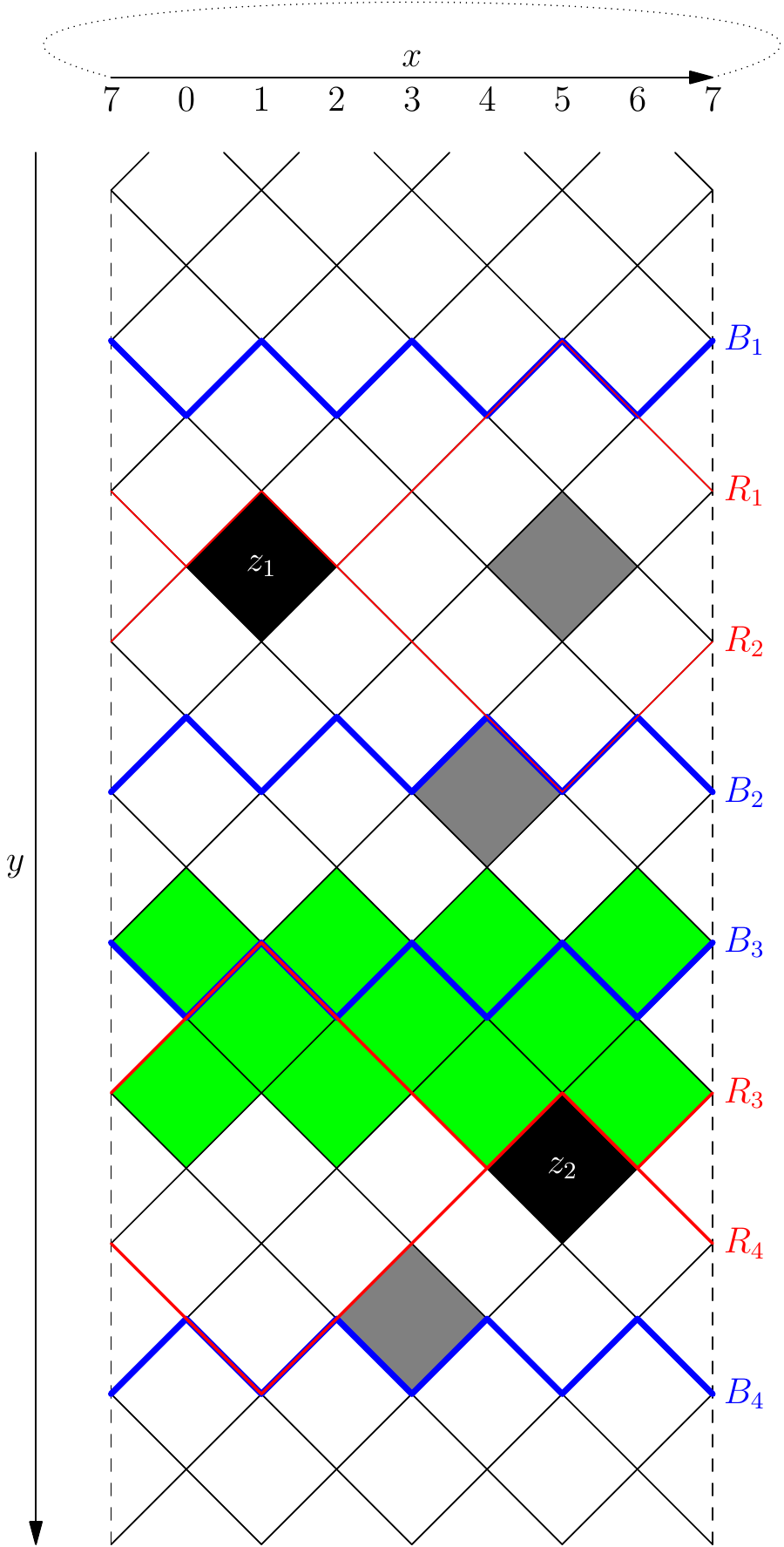}
    \end{center}
    \caption{Here $L=4$. White and green squares are squares such that $w_z=0$, grey squares such that $w_z=1$ but not use to bound, and black squares such that $w_z=1$ and use to find the boundary. A square could be both grey and green because, in that case, it should be black. The total maximum number of squares that can arrive during $\epsilon$ units of time is bounded by the number of squares between the two blue lines $B_1$ and $B_4$.} \label{fig:proofNoExplosion}
  \end{figure}
\end{proof}

After the non-explosion condition, we would like to generalise \C~\ref{cond:conv} to obtain ergodicity of $(\tilde{\Front}_n)_{n \in \RR_+}$. For that, we work with the densities. The equivalent for~\C~\ref{cond:conv} is
\begin{cond} \label{cond:conv-cont}
  there exists $\alpha> 0$ such that
  \begin{equation}
    \inf_{\Delta \in \RR_+, t \in \RR_+^*} \frac{f_\Delta(t)}{\mu_\Delta \left( [t,\infty) \right)} \geq \alpha.
  \end{equation}
\end{cond}
and then we obtain
\begin{proposition} \label{prop:HMC-cont}
  For any $\left( \mu_\Delta \right)_{\Delta \in \RR_+} \in \mprob{\RR_+^*}^{\RR_+}$ such that~\C~\ref{cond:noexplosion} and~\C~\ref{cond:conv-cont} hold, then $\tilde{\Front}_n$ is (exponentially) ergodic.
\end{proposition}
  
\begin{proof}
  Ergodicity of Markov processes in continuous space and time is much more technical that in discrete space and time. Very good references on it are the series of articles by Meyn and Tweedie~\cite{MT93-part1,MT93-part2,MT93-part3} and references therein. Here, we use Theorem~5.2 in~\cite{DMT95} that completes this series of articles. Moreover, here, the state-space considered is $\tilde{\Bridges}_L$ that adds some complexities because of its structure. Hence, we do not give in the following all the formal details of the proof, but the main ideas that permit to understand and check that it is correct.\par

  We begin here by the construction of a measure $\Psi$ on $\tilde{\Bridges}_L$.
  First, we need to define, for any $b \in \Bridges_L$, a measure $\Psi_b$ on $\Time_b$. Let $k_b = \card{i: b_i = -1,b_{i+1}=1}$ that is the number of local minima in $b$. We recall that, if $b_i=-1$ and $b_i=1$, then $t_i=t_{i+1}$. Then, we define $\Psi_b$ as a measure on $\Time_b$ that has a density $\psi_b$ according to the Lebesgue measure on $\RR^{2L-k_b}$:
  \begin{equation} \label{eq:HaarMeasure}
    \di \Psi_b(t) = \ind{t \in \Time_b} \prod_{i: b_{i-1} \neq -1 \text{ or } b_i \neq 1} \di t_i.
  \end{equation}
  Now, the measure $\Psi$ on $\tilde{\Bridges}_L$ is defined by
  \begin{equation} \label{eq:psi}
    \Psi = \sum_{b \in \Bridges_L} \delta_b \otimes \Psi_b  
  \end{equation}
  where $\delta_b$ is the Dirac measure on the finite space $\Bridges_L$ and $\otimes$ denotes the product measure.\par
  
  \begin{lemma}
    Let $(\mu_\Delta)_{\Delta \in \RR_+} \in \mprob{\RR_+^*}^{\RR_+}$ be such that their densities $f_{\Delta}$ satisfy, for any $\Delta \in \RR_+$, $x \in \RR_+^*$, $f_\Delta(x) > 0$. The Markov process $(\tilde{\Front}_n)_{n\in \RR_+}$ is $\Psi$-irreducible and aperiodic. 
  \end{lemma}
  
  \begin{proof}
    The definitions of $\Psi$-irreducibility and aperiodicity we use here are the ones given in Section~3 of~\cite{DMT95}.
      
    The $\Psi$-irreducibity is: for any $A \in \borel{\tilde{\Bridges}_L}$, if $\Psi(A) > 0$, then for any $(b,t) \in \tilde{\Bridges}_L$, there exists $t \in \RR_+$ such that $P^t(x,A) > 0$. The idea to prove it, it is just to say that from any configuration $(b,t) \in \Bridges_L$, we could go to the compact set $C_1 = \{ (b,t) : \forall i\ b_{2i}= 1, b_{2i+1} = -1, t_{2i} = t_{2i+1} \leq 1 \}   \cap \tilde{\Bridges}_L$ with a positive probability (for example, if many well chosen squares come during a short period of time), and from the set $C_1$ to any other compact set of $\tilde{\Bridges}_L$ with a positive probability (by choosing well when new squares come). This is possible because we have imposed that $f_{\Delta}(x) > 0$ for any $\Delta \in \RR_+$, $x \in \RR_+^*$.\par
    
    To prove the aperiodicity of the $\Psi$-irreducible $(\tilde{\Front}_n)_{n \in \RR_+}$, we define, for any $\epsilon > 0$, $C_\epsilon = \{ (b,t) : \forall i\ b_{2i}= 1, b_{2i+1} = -1, t_{2i} = t_{2i+1} \leq \epsilon \}   \cap \tilde{\Bridges}_L$. For any $\epsilon > 0$, $C_\epsilon$ is a small set and, for any $x \in C_\epsilon$, for any $t \in \RR_+$, $P^t(x,C_\epsilon) > 0$ (because $f_{\Delta}(x) > 0$ for any $\Delta \in \RR_+$, $x \in \RR_+^*$). 
  \end{proof}

  Now, to conclude we prove the condition~$(\tilde{\mathcal{D}})$ of Theorem~5.2 in~\cite{DMT95} that is an analogue of the Foster criterion in continuous space and time. We choose the same Lyapunov function $V$ as in discrete time (but we need to add a $1$) and we use the generator given in~\eqref{eq:generator}:
  \begin{equation}
    V((b,t)) = 1 + \sum_{i=1}^{2L} t_i.
  \end{equation}
  Then, for any $(b,t) \in \tilde{\Bridges}_L$, for any $n \in \RR_+$,
  \begin{align*}
    (GV)((b,t)) & = \frac{\partial V}{\partial \vec{u}} ((b,t)) + \sum_{i : b_i = 1 , b_{i+1} = -1} \frac{f_{|t_i-t_{i+1}|}(\min(t_i,t_{i+1}))}{\int_{s \geq 0} f_{|t_i-t_{i+1}|}(\min(t_i,t_{i+1}+s)) \di s} \left[ V((b^{(i)},t^{(i)})) - V((b,t)) \right] \\
    & = 2L - \sum_{i : b_i = 1, b_{i+1}=-1} \frac{f_{|t_i-t_{i+1}|}(\min(t_i,t_{i+1}))}{\int_{s \geq 0} f_{|t_i-t_{i+1}|}(\min(t_i,t_{i+1}+s)) \di s} (t_i+t_{i+1}) \\
    & \leq 2L - \alpha \sum_{i : b_i = 1, b_{i+1}=-1} (t_i+t_{i+1}) \\
  \end{align*}
  The last line is obtained because we suppose~\C~\ref{cond:conv-cont}. Now, if we set $T = \max\{t_1,\dots,t_{2L}\}$, we know that $T$ corresponds to a $t_i$ or $t_{i+1}$ where $b_i=1,b_{i+1}=-1$ and that $\displaystyle T \geq \frac{\sum_{i=1}^{2L} t_i}{2L}$. Hence
  \begin{align*}
    (GV)((b,t)) \leq & 2L - \alpha \frac{\sum_{i=1}^{2L} t_i}{2L} \\
    \leq & - \frac{\alpha}{2L} V((b,t)) + \left(2L + \frac{\alpha}{2L} \right)
  \end{align*}
  In addition, we remark that, for any $T \in \RR_+^*$, $\{(b,t) \in \tilde{\Bridges}_L: V((b,t)) \leq T\}$ are petite sets. Hence, that permits to prove condition~$(\tilde{\mathcal{D}})$ in~\cite{DMT95}. Now, we can apply Theorem~5.2 of~\cite{DMT95} that gives us that $(\tilde{\Front}_n)_{n \in \RR_+}$ is ergodic and even exponentially ergodic. 
\end{proof}

In the following, we denote by $\tilde{\nu}_L$ the invariant measure of $(\tilde{\Front}_n)_{n \in \RR_+}$. 

Now, we could express the asymptotic mean speed.
\begin{proposition}
  Let $\mu = (\mu_\Delta)_{\Delta \in \RR_+} \in \mprob{\RR_+^*}^{\RR_+}$ be such that~\C~\ref{cond:conv} holds. We denote by $\tilde{\nu}_L$ the invariant measure of the Markov chain $(\tilde{\Front}_n)_{n \in \RR_+}$. Let $(B,(T_1,\dots,T_{2L})) \sim \tilde{\nu}_L$. The asymptotic mean speed $\meanspeed_L$ of the front line of the LPP with parameter $\mu$ on the cylinder of size $L$ is
  \begin{equation}
    \meanspeed_L = \frac{1}{2 \esp{T_1}} = \frac{1}{\displaystyle 2 \int_{\tilde{\Bridges}_L} t_1 \di \tilde{\nu}_L((b,t))}
  \end{equation}
\end{proposition}

\begin{proof}
  The proof is the same as the one in discrete time (see~Section~\ref{sec:proofmeanspeed}) but in continuous time. The main difference is that now $(t_{n,1})_{n \in \RR_+}$ varies from $0$ to $\edgetime_j$ (and not $\edgetime_j -1$ as in discrete time). Hence, by the notations of~Section~\ref{sec:proofmeanspeed},
  \begin{equation}
    \esp{t_{0,1}|\edgetime_{0}} = \edgetime_0/2.
  \end{equation}
  That's why $\esp{t_{0,1}} = \esp{\edgetime_0}/2$.
\end{proof}

The integrable case is similar to the one in discrete time. The integrability condition is:
\begin{cond} \label{cond:int-cont}
  For any $\Delta \in \RR_+$, for any $i \in \RR_+^*$,
  \begin{equation}
    f_{\Delta}(i) = \frac{\sqrt{f_{0}(i) f_{0}(i+\Delta)}}{\displaystyle \int_{\RR_+^*}\sqrt{f_0(s) f_0(s+\Delta)} \di s}.
  \end{equation}
\end{cond}

\begin{proposition} \label{prop:int-c}
  Let $\left( \mu_{\Delta}\right)_{\Delta \in \RR_+} \in \mprob{\RR_+^*}^{\RR_+}$ be such that \C~\ref{cond:conv-cont} and~\C~\ref{cond:int-cont} hold, then $\tilde{\nu}_L$ is explicit and its density according to $\Psi$ (defined in~\eqref{eq:psi}) is
  \begin{equation}
    g((b,t)) = \frac{1}{Z} \prod_{i: b_i =b_{i+1}} \sqrt{f_0(|t_{i+1}-t_i|)} \prod_{i: b_i=1,b_{i+1} = -1} \left( \int_{\RR_+^*} \sqrt{f_0(s+t_i)} \sqrt{f_0(s+t_{i+1})} \di s \right) 
  \end{equation}
  where
  \begin{equation}
    Z = \sum_{b \in \Bridges_L} \int_{t \in \Time_b}  \prod_{i: b_i =b_{i+1}} \sqrt{f_0(|t_{i+1}-t_i|)} \prod_{i: b_i=1,b_{i+1} = -1} \left( \int_{\RR_+^*} \sqrt{f_0(s+t_i)} \sqrt{f_0(s+t_{i+1})} \di s \right) \di{t}.
  \end{equation}
\end{proposition}

\paragraph{Before the proof}\par
Before to make the proof, we need to fix some notations. First, for any $(b,t) \in \tilde{\Bridges}_L$ and $i \in \tore{2L}$ such that $b_i = 1,b_{i+1}=-1$, we let
\begin{equation}
  b^{(i)} = (b_1,\dots,b_{i-1},-1,1,b_{i+2},\dots,b_{2L}) \text{ and } t^{(i)} = (t_1,\dots,t_{i-1},0,0,t_{i+2},\dots,t_{2L}).
\end{equation}
Moreover, for any $b \in \Bridges_L$, for any $i \in \tore{2L}$ such that $b_i= -1, b_{i+1}=1$, we let
\begin{equation}
  \Time^{(i)}_b = \{t \in \Time_b : t_{i}=t_{i+1}= 0\}.
\end{equation}
As for $\Time_b$ in~\eqref{eq:HaarMeasure}, we define a measure $\Psi^{(i)}_b$ on $\Time^{(i)}_b$ by
\begin{equation}
  \di{\Psi^{(i)}_b}(t) = \ind{t \in \Time^{(i)}_b} \prod_{j: b_{j-1} \neq -1 \text{ or } b_i \neq 1 \text{ or } j \neq i} \di{t_j}.
\end{equation}
We can remark that
\begin{equation}
  \di{\Psi_b}((t_1,\dots,t_{2L})) = \sum_{i:b_{i} = -1,b_{i+1}=1} \di{\Psi^{(i)}_b}((t_1-t_i,\dots,t_{2L}-t_i)) \di{t_i}
\end{equation}
because, for any $b \in \Bridges_L$, $i, j \in \tore{2L}$ such that $b_i=b_j=-1, b_{i+1}=b_{j+1}=1$ and $i \neq j$, $\Time_b^{(i)} \cap \Time_b^{(j)}$ are null sets for the measure $\Psi_b$.
Finally, for any $(b,t) \in \tilde{\Bridges}_L$ and any $t' \in \RR_+$, we let 
\begin{align}
  F_{(b,t)}(t') & = \prod_{j: b_j =b_{j+1}} \sqrt{f_0(|t_{j+1}-t_j|)} \prod_{j: b_j=1,b_{j+1} = -1} \left( \int_{\RR_+^*} \sqrt{f_0(s+t_j+t')} \sqrt{f_0(s+t_{j+1}+ t')} \di s \right)\\
                & = \prod_{j: b_j =b_{j+1}} \sqrt{f_0(|t_{j+1}-t_j|)} \prod_{j: b_j=1,b_{j+1} = -1} \left( \int_{(t',\infty)} \sqrt{f_0(s+t_j)} \sqrt{f_0(s+t_{j+1})} \di s \right)\\
                & = Z\ g(b,(t_1+t',\dots,t_{2L}+t'))
\end{align}
and its derivative is 
\begin{align}
  \frac{\di{F_{(b,t)}}}{\di{t'}}(t') =  & - \sum_{i : b_i = 1 , b_{i+1} = -1} \sqrt{f_0(t'+t_i)} \sqrt{f_0(t'+t_{i+1})} \prod_{j: b_j =b_{j+1}} \sqrt{f_0(|t_{j+1}-t_j|)} \nonumber \\
                                        & \qquad \prod_{j: b_j=1, b_{j+1} = -1, j \neq i} \int_{s \in (t',\infty)} \sqrt{f_0(s+t_j)} \sqrt{f_0(s+t_{j+1})} \di{s}\\
  = &  - \sum_{i : b_i = 1 , b_{i+1} = -1} \frac{\sqrt{f_0(t'+t_i)} \sqrt{f_0(t'+t_{i+1})}}{\int_{s \in (t',\infty)} \sqrt{f_0(s+t_i)} \sqrt{f_0(s+t_{i+1})}} F_{(b,t)}(t')\\
  = & -  \sum_{i : b_i = 1 , b_{i+1} = -1} f_{|t_i-t_{i+1}|}(\min(t_i,t_{i+1}) + t')\ F_{(b,t)}(t').
\end{align}

\begin{proof}[Proof of Proposition~\ref{prop:int-c}]
  $\bullet$ To prove that $\tilde{\nu}_L$ is invariant, it is sufficient to prove that, for any $h \in \mathcal{C}^1_K(\tilde{\Bridges}_L)$ (the set of function in $\mathcal{C}^1(\tilde{\Bridges}_L)$ whose support is compact)
  \begin{displaymath}
    \int_{(b,t) \in \tilde{\Bridges}_L} (Gh)((b,t)) \di{\tilde{\nu}_L}((b,t)) = 0. 
  \end{displaymath}

  \begin{align}
    & \int_{(b,t) \in \tilde{\Bridges}_L} (Gh)((b,t)) \di{\tilde{\nu}_L}((b,t))\\
    = & \frac{1}{Z} \int_{\tilde{\Bridges}_L} \left(\frac{\partial h}{\partial \vec{u}}((b,t)) + \sum_{i : b_i=1,b_{i+1}=-1} f_{|t_i-t_{i+1}|}(\min(t_i,t_{i+1})) \left(h((b^{(i)},t^{(i)})) - h((b,t)) \right) \right) F_{(b,t)}(0) \di{\Psi}((b,t))\\
    = & \frac{1}{Z} \int_{\tilde{\Bridges}_L} \frac{\partial h}{\partial \vec{u}}((b,t)) F_{(b,t)}(0) \di{\Psi}((b,t)) \nonumber \\
    & + \frac{1}{Z} \int_{{\tilde{\Bridges}}_L}\ \sum_{i : b_i=1,b_{i+1}=-1} h((b^{(i)},t^{(i)}))\ f_{|t_i-t_{i+1}|}(\min(t_i,t_{i+1})) F_{(b,t)}(0) \di{\Psi}((b,t)) \nonumber \\
    & - \frac{1}{Z} \int_{{\tilde{\Bridges}}_L}\ \sum_{i : b_i=1,b_{i+1}=-1} h((b,t))\ f_{|t_i-t_{i+1}|}(\min(t_i,t_{i+1})) F_{(b,t)}(0) \di{\Psi}((b,t)). \label{eq:sum1}
  \end{align}

  $\bullet$ Now, we rewrite the first term of the sum~\eqref{eq:sum1}
  \begin{align*}
    & \int_{\tilde{\Bridges}_L} \frac{\partial h}{\partial \vec{u}}((b,t)) F_{(b,t)}(0) \di{\Psi}((b,t)) \\
    = & \sum_{b \in \Bridges_L} \sum_{i : b_i=-1,b_{i+1}=1} \int_{t \in \Time^{(i)}_b} \int_{t'\in \RR_+} \frac{\partial h}{\partial \vec{u}}((b,(t_1+t',\dots,t_{2L}+t'))) F_{(b,t)}(t') \di{t'} \di{\Psi^{(i)}_b}(t).
  \end{align*}
  Now, we integrate by part the integral in $t'$, for that we could consider $h_{(b,t)}(t') = h((b,(t_1+t',\dots,t_{2L}+t'))$, hence $\displaystyle \frac{\partial h}{\partial \vec{u}}((b,(t_1+t',\dots,t_{2L}+t'))) = \frac{\di{h_{(b,t)}}}{\di{t'}}(t')$.
  \begin{align*}
    & \int_{t'\in \RR_+} \frac{\di{h_{(b,t)}}}{\di{t'}}(t') F_{(b,t)}(t') \di{t'} \\
    = & \left[ h_{(b,t)}(t') F_{(b,t)}(t') \right]_{0}^{\infty} - \int_{t'\in \RR^+} h_{(b,t)}(t') \frac{\di{F_{(b,t)}}}{\di t'}(t') \di{t'}\\
    = & 0 - h((b,t)) F_{(b,t)}(0) + \int_{t' \in \RR_+} \sum_{j:b_j=1,b_{j+1}=-1} h_{(b,t)}(t')\ f_{|t_j-t_{j+1}|}(\min(t_j,t_{j+1})+t')\ F_{(b,t)}(t') \di{t'}.
  \end{align*}
  So,
  \begin{align*}
    & \int_{\tilde{\Bridges}_L} \frac{\partial h}{\partial \vec{u}}((b,t)) F_{(b,t)}(0) \di{\Psi}((b,t)) \\
    = & - \sum_{b \in \Bridges_L} \sum_{i : b_i=-1,b_{i+1}=1} \int_{t \in \Time^{(i)}_b} h((b,t)) F_{(b,t)}(0) \di{\Psi^{(i)}_b}(t)\\
    & + \sum_{b \in \Bridges_L} \sum_{i : b_i=-1,b_{i+1}=1} \int_{t \in \Time^{(i)}_b}  \int_{t' \in \RR_+} \sum_{j:b_j=1,b_{j+1}=-1} h_{(b,t)}(t')\ f_{|t_j-t_{j+1}|}(\min(t_j,t_{j+1})+t')\ F_{(b,t)}(t') \di{t'} \di{\Psi^{(i)}_b}(t)\\
    = & - \sum_{b \in \Bridges_L} \sum_{i : b_i=-1,b_{i+1}=1} \int_{t \in \Time^{(i)}_b} h((b,t)) F_{(b,t)}(0) \di{\Psi^{(i)}_b}(t)\\
    & + \int_{\tilde{\Bridges}_L} \sum_{j:b_j=1,b_{j+1}=-1} h((b,t)) f_{|t_j-t_{j+1}|}(\min(t_j,t_{j+1})) F_{(b,t)}(0) \di{\Psi}((b,t))
  \end{align*}
  
  The second term of this last equation cancels with the third of~\eqref{eq:sum1}. So to end the proof, we need to show that
  \begin{align}
    & \int_{{\tilde{\Bridges}}_L}\ \sum_{i : b_i=1,b_{i+1}=-1} h((b^{(i)},t^{(i)}))\ f_{|t_i-t_{i+1}|}(\min(t_i,t_{i+1})) F_{(b,t)}(0) \di{\Psi}((b,t)) \nonumber \\
    = & \sum_{b \in \Bridges_L} \sum_{i : b_i=-1,b_{i+1}=1} \int_{t \in \Time^{(i)}_b} h((b,t)) F_{(b,t)}(0) \di{\Psi^{(i)}_b}(t).
  \end{align}

  $\bullet$ But,
  \begin{align*}
  & \int_{{\tilde{\Bridges}}_L}\ \sum_{i : b_i=1,b_{i+1}=-1} h((b^{(i)},t^{(i)}))\ f_{|t_i-t_{i+1}|}(\min(t_i,t_{i+1})) F_{(b,t)}(0) \di{\Psi}((b,t)) \\
    = & \sum_{b \in \Bridges_L} \int_{t \in \Time_b} \sum_{j : b_j=1,b_{j+1}=-1} h((b^{(j)},t^{(j)})) f_{|t_j-t_{j+1}|}(\min(t_j,t_{j+1})) F_{(b,t)}(0) \di{\Psi_b}(t).
  \end{align*}
  Here, we do a change a variable passing from $b$ to $c$ where $c = b^{(j)}$ and $t$ passing from $u$ to $t^{(j)}$.
  \begin{align*}
    = & \sum_{c \in \Bridges_L} \sum_{j : c_j=-1,c_{j+1}=1} \int_{u \in \Time^{(j)}_c} \di{\Psi^{(j)}_c}(u) h((c,u)) F_{(c,u)}(0) \\
      & \left( \frac{\ind{c_{j-1}=-1}}{\sqrt{f_0(u_{j-1})}} + \frac{\ind{c_{j-1}=1}}{\int_{s \in \RR_+} \sqrt{f_0(u_{j-1}+s)}\sqrt{f_0(s)} \di{s}} \right) \left( \frac{\ind{c_{j+2}=1}}{\sqrt{f_0(u_{j+2})}} + \frac{\ind{c_{j+2}=-1}}{\int_{s \in \RR_+} \sqrt{f_0(u_{j+2}+s)}\sqrt{f_0(s)} \di{s}} \right) \\
      & \left( \ind{c_{j-1}=-1}\ \ind{c_{j+2}=1}\ \underbrace{f_{|u_{j-1}-u_{j+2}|}(\min(u_{j-1},u_{j+2})) \int_{s \in \RR^+} \sqrt{f_0(u_{i-1}+s)} \sqrt{f_0(u_{i+2}+s)} \di{s}}_{= \sqrt{f_0(u_{j-1})} \sqrt{f_0(u_{j+2})}} \right. \\ 
      & \left. + \ind{c_{j-1}=1}\ \ind{c_{j+2}=1} \int_{s_{j}\in \RR_+} \right.\\
      & \qquad \underbrace{f_{|u_{j-1}+s_j-u_{j+2}|}(\min(u_{j-1}+s_j,u_{j+2})) \int_{s \in \RR^+} \sqrt{f_0(u_{i-1}+s_j+s)} \sqrt{f_0(u_{i+2}+s)} \di{s}}_{= \sqrt{f_0(u_{j-1}+s_j)} \sqrt{f_0(u_{j+2})}}\\
      & \qquad \qquad \sqrt{f_0((u_{j-1}+s_{j})-u_{j-1})}\di{s_{j}}\\ 
      & \left. + \ind{c_{j-1}=-1}\ \ind{c_{j+2}=-1} \int_{s_{j+1}\in \RR_+} \sqrt{f_0(u_{j-1})} \sqrt{f_0(u_{j+2}+s_{j+1})} \sqrt{f_0(s_{j+1})} \di{s_{j+1}} \right.\\
      & \left. + \ind{c_{j-1}=1}\ \ind{c_{j+2}=-1} \int_{s_{j}\in \RR_+} \int_{s_{j+1} \in \RR_+} \sqrt{f_0(u_{j-1}+s_j)} \sqrt{f_0(u_{j+2}+s_{j+1})} \sqrt{f_0(s_j)} \sqrt{f_0(s_{j+1})} \di{s_{j}} \di{s_{j+1}} \right) \\
    = &  \sum_{c \in \Bridges_L} \sum_{j : c_j=-1,c_{j+1}=1} \int_{u \in \Time^{(j)}_c} h((c,s)) F_{(c,s)}(0) \di{\Psi^{(j)}_c}(s). 
  \end{align*}
  In previous computation, we split the sum in 4 terms because when $c_{j-1} = -1$, then in $b$, $b_{j-1}=-1$ and $b_j=1$, so $t_j=t_{j-1}=u_{j-1}$ because $t \in \Time_b$ and $u = t^{(i)}$. Similarly, when $c_{j+2}=1$. But, when $c_{j-1}=1$, then in $b$, $t_j > t_{j-1} = u_{j-1}$, and we integrate on all possible value of $s_j = t_j-t_{j-1}$. Similarly, when $c_{j+2} = -1$.  
\end{proof}  

\section{Examples} \label{sec:examples}
In the first part of this section, we apply our previous results to the integrable LPP on the cylinders. That permits us to find a very simple expression of the asymptotic law of the front line. In the second part, we prove that the directed edge-LPP (as defined in Remarks~\ref{rem:banalite}) is a GLPP. And, in the third part, we discuss GLPP on the quarter-plane. In particular, we present some simulations of integrable GLPP on the quarter-plane for different $\mu_0$.\par

\subsection{Integrable LPP on the cylinders} \label{sec:classical}
In this section, we consider that $\mu_\Delta = \mu_0$ for any $\Delta \in \NN$ that corresponds to the LPP and that $\mu_0$ is a geometrical law (on $\NN^*$) of success parameter $p \in (0,1)$, i.e.\ for any $i \in \NN^*$,
\begin{equation}
  \mu_0(i) = p (1-p)^{i-1}.
\end{equation}
Hence, we get the integrable LPP in discrete time.

\begin{lemma} \label{lem:classical}
  Let $\mu_0 \in \mprob{\NN^*}$ and define $\mu = \left(\mu_\Delta\right)_{\Delta \in \NN}$ by~\C~\ref{cond:int}. The two following conditions are equivalent:
  \begin{itemize}
  \item for any $\Delta \in \NN$, $\mu_\Delta = \mu_0$,
  \item there exists $p \in (0,1)$ such that $\mu_0$ is a geometrical law (on $\NN^*$) of success parameter $p$.
  \end{itemize}
\end{lemma}

\begin{proof}
  $1 \Leftarrow 2$: for any $\Delta \in \NN$ and $i \in \NN^*$:
  \begin{align*}
    \mu_{\Delta}(i) & = \frac{\sqrt{\mu_0(i) \mu_0(i+\Delta)}}{\sum_{s \geq 1} \sqrt{\mu_0(i) \mu_0(i+\Delta)}} = \frac{\sqrt{p (1-p)^{i-1} p (1-p)^{i+\Delta-1}}}{\sum_{s \geq 1} \sqrt{p (1-p)^{s-1} p (1-p)^{s+\Delta-1}}} \\
                    & = \frac{p (1-p)^{i-1} (1-p)^{\Delta/2}}{\sum_{s \geq 1} p (1-p)^{s-1} (1-p)^{\Delta/2}} = \frac{p (1-p)^{i-1}}{\sum_{s \geq 1} p (1-p)^{s-1}} = p (1-p)^{i-1}.
  \end{align*}\par

  $1 \Rightarrow 2$: for any $i \in \NN^*$:
  \begin{equation}
    \frac{\mu_0(i)}{\mu_0(1)} = \frac{\mu_1(i)}{\mu_1(1)} = \frac{\sqrt{\mu_0(i) \mu_0(i+1)}}{\sqrt{\mu_0(1) \mu_0(2)}}.
  \end{equation}
  Hence,
  \begin{equation}
    \sqrt{\frac{\mu_0(i)}{\mu_0(1)}} = \sqrt{\frac{\mu_0(i+1)}{\mu_0(2)}}.
  \end{equation}
  So, for any $i \in \NN^*$, by denoting $p = 1 - \frac{\mu_0(2)}{\mu_0(1)}$,
  \begin{equation}
    \frac{\mu_0(i+1)}{\mu_0(i)} = \frac{\mu_0(2)}{\mu_0(1)} = 1-p.
  \end{equation}
\end{proof}

In this particular case,
\begin{lemma}
  Let $p \in (0,1)$ and let $\mu_0$ be a geometrical law of success parameter $p$. For any $\Delta \in \NN$, take $\mu_\Delta = \mu_0$. Let's define the front line $(\Front_n)_{n \in \NN}$ as in Section~\ref{sec:gLPPdef}. In this case, the front line $(\Front_n)_{n \in \NN}$ is a Markov chain on $\Bridges_L$.
\end{lemma}

\begin{proof}
  It is not difficult to see and check that the Markov kernel $M = (M_{b,c})_{b,c \in \Bridges_L}$ is the following one: for any $b = (b_i)_{i \in \ZZL} \in \Bridges_L$, for any $c = (c_i)_{i \in \ZZL} \in \Bridges_L$,
  \begin{equation}
    M_{b,c} = \prod_{i : b_i = b_{i+1} = 1} \ind{b_i=c_i} \prod_{i:b_i = b_{i-1} = -1} \ind{b_i=c_i} \prod_{i:b_i=1=-b_{i+1}} ((1-p) \ind{c_i=1=-c_{i+1}} + p \ind{c_i=-1=-c_{i+1}}). 
  \end{equation}
  In words, nothing changes except on local maxima. Each local maximum becomes a local minimum independently with probability $p$.
\end{proof}

The Markov chain $(F_n)_{n \in \NN}$ is ergodic and its invariant measure is
\begin{proposition} \label{prop:clasCyl}
  Let $\mu_0$ be a geometrical law of success parameter $p \in (0,1)$ and take $\mu_\Delta = \mu_0$ for any $\Delta \in \NN$. The invariant law $\nu_L$ of the LPP with parameter $\mu$ is, for any $b \in \Bridges_L$,
  \begin{equation}
    \nu_L(b) = \frac{1}{Z} \frac{1}{(1-p)^{k_b}}
  \end{equation}
  with $\displaystyle Z = \sum_{b \in \Bridges_L} \frac{1}{(1-p)^{k_b}}$.
\end{proposition}

\begin{proof}
  It is a result known in the folklore of LPP. In Annex~\ref{sec:ann-LPP}, we deduce it as a corollary of Theorem~\ref{thm:int}.
\end{proof}

When $p \to 0$, we obtain the uniform measure on $\Bridges_L$. This suggests the following proposition:
\begin{proposition} \label{prop:clasCylR}
  Let $\mu_0$ be an exponential law of parameter $\lambda \in (0,\infty)$ and take $\mu_\Delta = \mu_0$ for any $\Delta \in \NN$. The invariant law $\nu_L$ of the LPP with parameter $\mu$ is the uniform law on $\Bridges_L$.
\end{proposition}

\begin{proof}
  It is also a very well-known result about LPP on cylinders. A simple way to prove it is to remark that: for any $b$ with $k$ local maxima (and $k$ local minima), during a short period of time $\di n$,
  \begin{itemize}
  \item if $\Front_n = b$, it goes out of $b$ with probability $k \lambda \di n + o(\di n)$ and
  \item if $\Front_n \neq b$, then there is $k$ ways to become $b$ (it corresponds to the $k$ bridges where one and only one of the minimum local of $b$ is a maximum local), and so under the uniform measure on $\Bridges_L$, the probability for $\Front_{n+\di n}$ to be $b$ is $k \lambda \di n + o(\di n)$. \qedhere
  \end{itemize}
\end{proof}

\subsection{Classical edge-LPP} \label{sec:edge-classical}
In this section, we just want to prove that the classical edge-LPP is just a particular case of the GLPP.\par

\begin{lemma} \label{lem:edge-classical}
  Let $\mu \in \mprob{\NN^*}$. The ``classical directed edge-LPP'' with weight law $\mu$ on edges is the GLPP with parameter $(\mu_\Delta)_{\Delta \in \NN}$ where, for any $\Delta \in \NN$, any $i \in \NN^*$,
  \begin{equation}
    \mu_{\Delta}(i) = \mu(i+\Delta) \sum_{j=1}^{i-1} \mu(j) + \mu(i) \sum_{j=1}^{i-1+\Delta} \mu(j) + \mu(i) \mu(i + \Delta).
  \end{equation}
\end{lemma}

\begin{proof}
  Let's suppose that $(i,j)$ arrived at time $a$ and $(i+2,j)$ arrived at time $b = a + \Delta$. Then, $(i+1,j+1)$ arrives at time $t = \max(a + \zeta_1,a + \Delta + \zeta_2)$ where $(\zeta_1,\zeta_2)$ are i.i.d.\ with law $\mu$. Now suppose that $t$ could be rewritten as $a+\Delta + \zeta$ where $\zeta \sim \mu_\Delta$. Hence,
  \begin{align*}
    & \mu_{\Delta}(i) = \prob{\zeta = i} \\
    = & \prob{a + \Delta + i = a + \zeta_1 > a + \Delta + \zeta_2} + \prob{a + \Delta + i = a + \Delta + \zeta_2 > a + \zeta_1} \\
    & \quad + \prob{a + \Delta + i = a + \Delta + \zeta_2 = a + \zeta_1} \\
    = & \prob{\zeta_1 = i + \Delta \text{ and } \zeta_2 < i} + \prob{\zeta_1 < i + \Delta \text{ and } \zeta_2 = i} + \prob{\zeta_1 = i + \Delta \text{ and } \zeta_2 = i} \\
    = & \mu(i+\Delta) \sum_{j=1}^{i-1} \mu(j) + \mu(i) \sum_{j=1}^{i-1+\Delta} \mu(j) + \mu(i) \mu(i + \Delta). \qedhere
  \end{align*}
\end{proof}

It seems that there does not exist an integrable model of classical directed edge-LPP via our methods.

\subsection{GLPP on the quarter-plane} \label{sec:exQuarter}
In this section, we make a few comments and remarks about the difference between the LPP on the quarter-plane and the GLPP on the quarter-plane.\par

We give first a formal definition of the GLPP on the quarter-plane. It is the same as the one given in the introduction but with a translation by the vector $(1,1)$.\par
Let $(\mu_{\Delta})_{\Delta \in \NN} \in \mprob{\NN^*}^{\NN}$. To each cell $(x,y) \in \NN^2$, we associate a number $\TL((x,y))$ such that
\begin{itemize}
  \item $\TL((x,0)) = 0$ for any $x \in \NN$,
  \item $\TL((0,y)) = 0$ for any $y \in \NN$,
  \item $\TL((x,y)) = \max(\,\TL((x,y-1))\, , \, \TL((x-1,y))\,) + \xi_{(x,y)}$ for any $(x,y) \in (\NN^*)^2$
\end{itemize}
where $\xi_{(x,y)} \sim \mu_{|\TL(x,y-1) - \TL(x-1,y)|}$ and $\left( \xi_{(x,y)} \right)_{(x,y) \in \cyl{L}}$ are independent. \par

As before, we could be interested in the study of the curve $\Front_n$ that splits $\{z \in (\NN^*)^2 : T(z) \leq n\}$ and $\{z \in (\NN^*)^2 : T(z) > n\}$, and, in particular, by its asymptotic shape when $n \to \infty$.\par

Contrary to the classical case (see Theorem~\ref{thm:classicalQuarter}), in GLPP, in most cases, the superadditive property of $(\TL(z))_{z \in \NN^2}$ does not hold. But we could obtain it in some special (and restrictive?) cases:
\begin{proposition} \label{prop:concave}
  Let $\mu = (\mu_\Delta)_{\Delta \in \NN} \in \mprob{\NN^*}^{\NN}$. 
  If, the following condition holds
  \begin{cond} \label{cond:subadd}
    for any $\Delta \in \NN$,
    \begin{equation}
      \mu_{\Delta}([0,n]) \leq \mu_{\Delta+1}([0,n]) \leq \mu_{\Delta}([0,n+1]),
    \end{equation}
  \end{cond}
  then, for any $z_1,z_2 \in \NN^2$,
  \begin{equation}
    \TL(z_1+z_2) \geq \TL(z_1) + \TL(z_2)\ a.s.
  \end{equation}
\end{proposition}

\begin{remark}
  Before to do the proof, we define a GLPP on the quarter-plane \emph{with boundary condition} $\omega = (\omega_i)_{i \in \ZZ}$ by taking $\TL((x,0)) = \omega_x$ for any $x \in \NN$ and $\TL((0,y)) = \omega_{-y}$ for any $y \in \NN$ in the definition above. In this case, we denote by $(\TL_\omega(z))_{z \in \NN^2}$, the arrival times of squares, and we keep $\TL(z)$ for $\TL_{0^\ZZ}(z)$.
\end{remark}

\begin{proof}
  Obviously, $\TL(z_1+z_2) = \TL(z_1) + (\TL(z_1+z_2) - \TL(z_1)) = \TL(z_1) + \TL_{\omega}(z_2)$ where $\omega = (\omega_i)_{i \in \ZZ}$ with
  \begin{equation}
    \omega_i = \begin{cases}
      \TL(z_1+(i,0)) - \TL(z_1) & \text{ if } i \geq 0, \\
      \TL(z_1+(0,-i)) - \TL(z_1) & \text{ if } i \leq 0.
    \end{cases}
  \end{equation}

  Now, to obtain the superadditivity property, we prove the following lemma
  \begin{lemma}
    Let $\mu = (\mu_\Delta)_{\Delta \in \NN} \in \mprob{\NN^*}^{\NN}$ be such that \C~\ref{cond:subadd} holds, and let $(\omega_i)_{i \in \ZZ} \in \NN^\ZZ$ be any sequence that decreases on $(-\infty,0] \cap \ZZ$ and increases on $[0,\infty) \cap \ZZ$. Then, for any $z \in \NN^2$, $\TL_\omega(z) \geq \TL(z)$ (stochastically).
  \end{lemma}

  \begin{proof}
    The proof is done by induction. First, for any $z = (0,y)$ or $z=(x,0)$, $\TL(z) = 0 \leq \TL_\omega(z)$. Now take $z=(x+1,y+1)$ with $(x,y) \in \NN^2$,
    \begin{equation}
      \TL_\omega(z) = \max(\TL_\omega((x,y+1)) , \TL_\omega((x+1,y))) + \xi_{z}
    \end{equation}
    with $\xi_z \sim \mu_{|\TL_\omega((x,y+1)) - \TL_\omega((x+1,y))|}$.

    Now, by induction, we know that $\TL_\omega((x,y+1)) = \TL((x,y+1) + a$ and $\TL_\omega((x+1,y)) = \TL((x+1,y) + b$ with $a,b \geq 0$. Hence,
    \begin{equation}
      \TL_\omega(z) = \max(\TL((x,y+1)) + a , \TL((x+1,y)) + b) + \xi_{z}.
    \end{equation}

    Now, we have to split into 4 cases, but only 2 by symmetry (we suppose that $\TL((x,y+1)) \geq \TL((x+1,y))$, the case $\TL((x,y+1)) \leq \TL((x+1,y))$ is similar).
    \begin{itemize}
    \item If $\max(\TL((x,y+1)) + a , \TL((x+1,y)) + b) = \TL((x,y+1)) + a$, then
      \begin{equation}
        \TL_\omega(z) = \max(\TL((x,y+1)),\TL((x+1,y))) + (a + \xi_{z}). 
      \end{equation}
      In that case, we have to prove that $a+\xi_z$ is stochastically greater than $\tilde{\xi}_z$ where $\tilde{\xi}_z \sim \mu_{\TL((x,y+1)) - \TL((x+1,y)}$, that is: for any $i$,
      \begin{align}
        \prob{\tilde{\xi_z} \leq i} \geq &\ \prob{a+\xi_z \leq i} \nonumber \\
        \mu_{\TL((x,y+1)) - \TL((x+1,y)}([0,i]) \geq &\ \mu_{\TL_\omega((x,y+1)) - \TL_\omega((x+1,y)}([0,i-a]) \nonumber \\
        \mu_{\TL((x,y+1)) - \TL((x+1,y)}([0,i + a ]) \geq &\ \mu_{\TL((x,y+1)) - \TL((x+1,y) + (a-b)}([0,i]). \nonumber
      \end{align}
      \begin{itemize}
      \item If $a-b =0$, it is obvious because $a \geq 0$.
      \item If $a-b < 0$, we use $(b-a)$ times the left size of~\C~\ref{cond:subadd}, to get that
        \begin{equation}
          \mu_{\TL((x,y+1)) - \TL((x+1,y)}([0,i + a ]) \geq \mu_{\TL((x,y+1)) - \TL((x+1,y) + (a-b)}([0,i + a ])
        \end{equation}
        and we conclude as in the case $a-b=0$.
      \item If $a-b > 0$, then we use $(a-b)$ times the right size of~\C~\ref{cond:subadd} to get that 
        \begin{equation}
          \mu_{\TL((x,y+1)) - \TL((x+1,y)}([0,i + a ]) \geq \mu_{\TL((x,y+1)) - \TL((x+1,y) + (a-b)}([0,i +b])
        \end{equation}
        and we conclude as in the case $a-b=0$ using the fact that $b \geq 0$.
      \end{itemize}

    \item If $\max(\TL((x,y+1)) + a , \TL((x+1,y)) + b) = \TL((x+1,y)) + b = \TL((x,y+1)) + (b-a)$, then
      \begin{equation}
        \TL_\omega(z) = \max(\TL((x,y+1)),\TL((x+1,y))) + ((b-a) + \xi_{z}). 
      \end{equation}
      In that case, we have to prove that, $(b-a)+\xi_z$ is stochastically greater than $\tilde{\xi}_z$ where $\tilde{\xi}_z \sim \mu_{\TL((x,y+1)) - \TL((x+1,y)}$, that is, for any $i$,
      \begin{align}
        \prob{\tilde{\xi_z} \leq i} \geq &\ \prob{(b-a)+\xi_z \leq i} \nonumber \\
        \mu_{\TL((x,y+1)) - \TL((x+1,y)}([0,i]) \geq &\ \mu_{\TL_\omega((x,y+1)) - \TL_\omega((x+1,y)}([0,i-(b-a)]) \nonumber \\
        \mu_{\TL((x,y+1)) - \TL((x+1,y)}([0,i + (b-a)]) \geq &\ \mu_{\TL((x,y+1)) - \TL((x+1,y) + (b-a)}([0,i]). \nonumber
      \end{align}
      The last condition is obtained by applying the right size of~\C~\ref{cond:subadd} $(b-a)$ times.
    \end{itemize}
    That permits to conclude, that, for any $z \in \NN^2$, $\TL_\omega(z) \geq \TL(z)$ (stochastically). 
  \end{proof}

  So, by this lemma, we find that
  \begin{equation}
    \TL(z_1+z_2) \geq \TL(z_1) + \TL(z_2) \text{ a.s.}
  \end{equation}
  The ``almost sure'' is obtainable by choosing a good coupling between $\xi_z$ and $\tilde{\xi}_z$. We can use the most naive one: let $U$ be uniform on $[0,1]$, $\xi_z = F^{-1}(U)$ and $\tilde{\xi}_z = \tilde{F}^{-1}(U)$ where $F$ and $\tilde{F}$ are the cumulative distribution function of $\xi_z$ and $\tilde{\xi}_z$.\par

  Hence, we get the superadditivity property.
\end{proof}

To conclude this section, we present three simulations of GLPP on the quarter-plane. In any case, we are under the integrability condition~\C~\ref{cond:int} and we choose $\mu_0$ is a Poisson law, a geometrical law (classical LPP), and a Zeta law of parameter $\alpha > 2$ (i.e.~$\prob{i} = \frac{1}{Z} \frac{1}{i^\alpha}$). See~Figure~\ref{fig:simu}.\par

\begin{figure}
  \begin{center}
    \includegraphics[width=0.3 \textwidth]{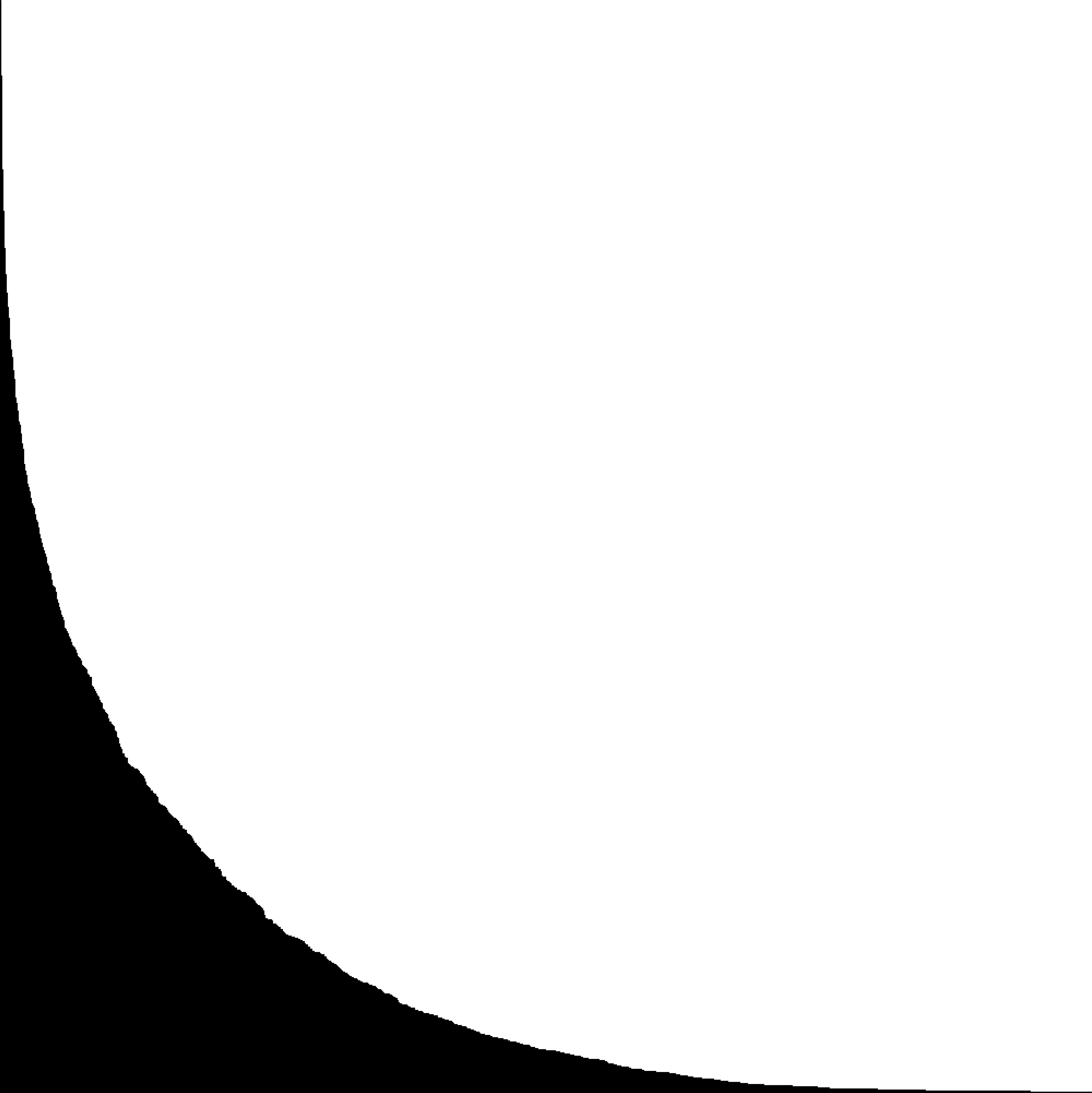} \hfill
    \includegraphics[width=0.3 \textwidth]{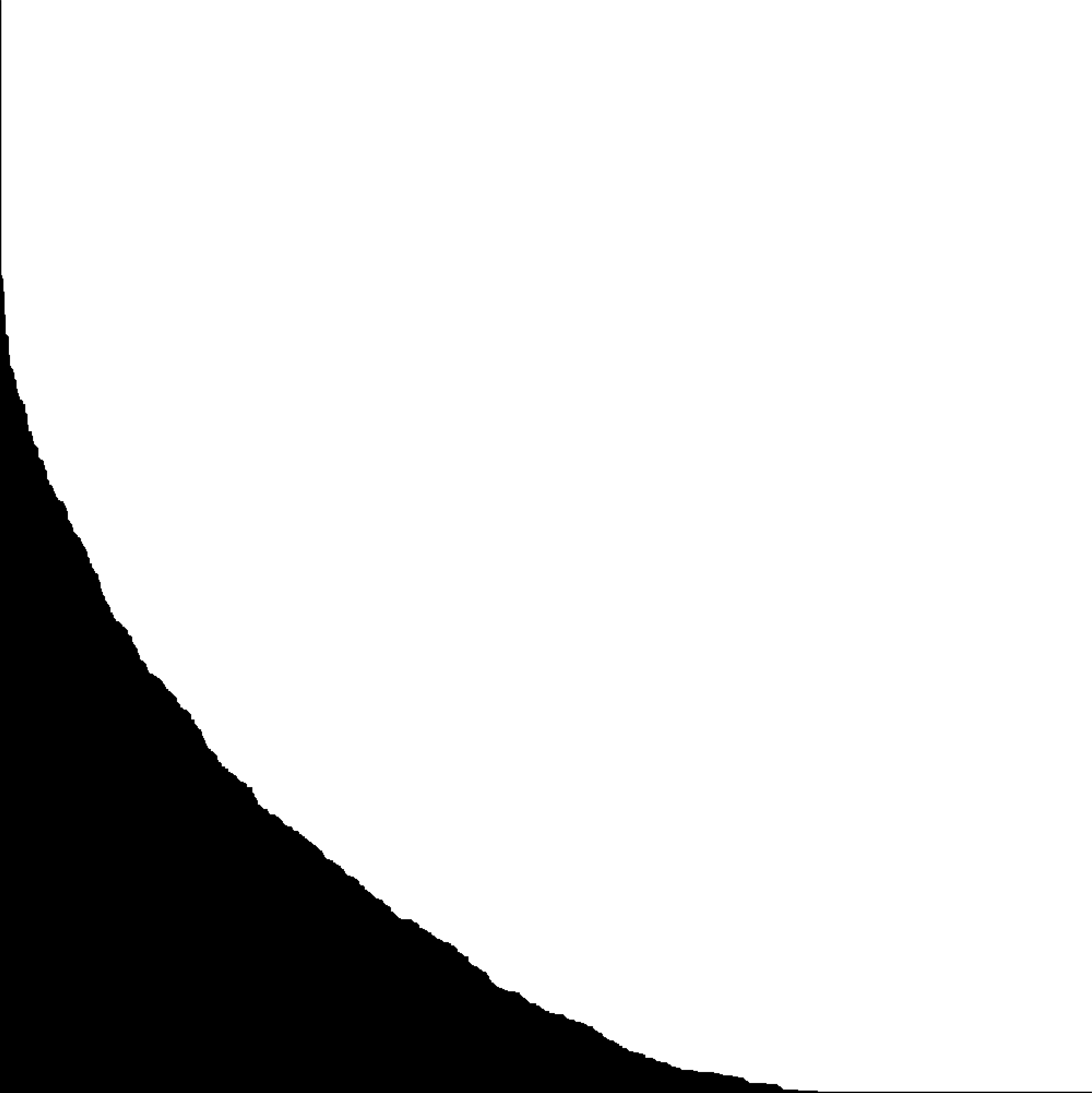} \hfill
    \includegraphics[width=0.3 \textwidth]{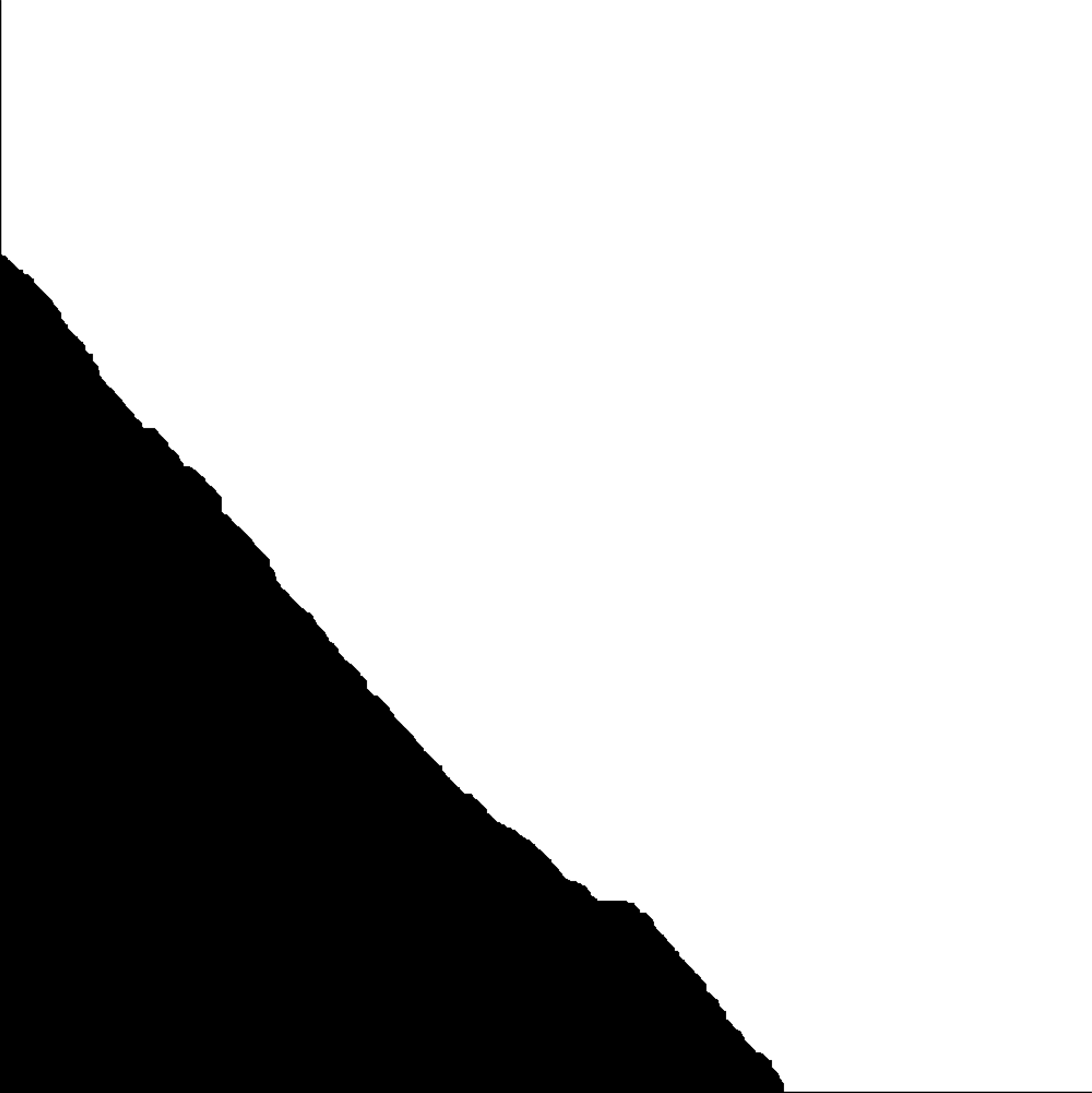}
  \end{center}
  \caption{Simulation of the GLPP on the quarter-plane under integrability condition when, from left to right, $\mu_0$ is a Poisson law of parameter $1$ (left), geometrical law of parameter $0.75$ (middle) and Zeta law of parameter $6$ (right). The growth model is represented at time $1000$ in a $1000 \times 1000$ box.} \label{fig:simu}
\end{figure}

We can remark that all the three lines seem asymptotically more or less concave. When $\mu_0$ is a Poisson law, it is easy to prove the left size in~\C~\ref{cond:subadd}, we try to check the right size, but it's still open. When $\mu_0$ is a Zeta law, the line seems to be straight. Moreover, when $\mu_0$ is a Zeta law, the left size in~\C~\ref{cond:subadd} does not hold.

\section{Open questions} \label{sec:open}
To conclude this article, we would like to give some interesting directions and open questions about this new model of LPP.

\begin{itemize}
\item The first one is to determine the asymptotic of $\nu_L$ when $L \to \infty$, firstly when the model is integrable and, maybe after, for any parameter $\mu \in \mprob{\NN^*}^\NN$. This could be interesting to know if these GLPP converge all to the Brownian bridges (as we can deduce from~Propositions~\ref{prop:clasCyl} and~\ref{prop:clasCylR} for integrable LPP on the cylinders) or not.
  
\item The second one is to determine the asymptotic shapes of the front line when we study integrable GLPP on the quarter-plane. That is done when $\mu_0$ is an exponential law or a geometrical law~\cite{Rost81,CEP96}. But we could ask what happens for any other values of $\mu_0$. In Figure~\ref{fig:simu}, we simulate the case where $\mu_0$ is a Poisson law and when $\mu_0$ is a Zeta law. In the Zeta law case, the asymptotic shape seems to be a straight line.
    
\item Another interesting question is to ask about the invariant laws invariant by translation when we consider the LPP on the half-plane. For now, we can describe, as said in Remark~\ref{rem:alpha}, some of them, that are parameterised by $\alpha \in \RR_+^*$. Probably, they are the only ones, but we are not able to prove it due to a lack of ergodicity results. So, more works should be done about ergodicity of PCA or just about ergodicity of these models of GLPP on the half-plane. 
  
\item Moreover, this kind of generalisation could be done for the directed First Passage Percolation, we just need to replace the $\max$ by the $\min$ in the definition of PCA. Unfortunately, none of these new models is integrable. Also, if we use PCA of memory 2 (see~\cite{CM19}) we can model First Passage Percolation on the triangular lattice and can probably define some new and interesting generalisations, but none of them could be integrable via our methods.

\item Finally, what we have done could be done maybe for other functions $f(a,b)$ different of $f(a,b) = |a-b|$. The approach should not be too different, but we are not sure about the physical interest and meaning of doing it.
\end{itemize}

\section*{Acknowledgement}
I am very grateful to Nathana\"el Enriquez, Jean-François Marckert and Irène Marcovici. Their comments and suggestions have been of great benefit. I am also very grateful to the Laboratoire Mathématiques d'Orsay for the support and supply.

\bibliographystyle{plain}
\bibliography{aaa}

\section{Annex} \label{sec:ann-LPP}
In this annex, we give a proof of Proposition~\ref{prop:clasCyl} as a corollary of Theorem~\ref{thm:int}, i.e.\  we compute $\sum_{t \in \Time_b} W_{(b,t)}$ as given in~\eqref{eq:wbt-usimple} for the LPP case.

\begin{proof}[Proof of Proposition~\ref{prop:clasCyl}]
  Because~\C~\ref{cond:conv} holds with $\alpha = 1-p$, we know that the invariant probability measure $\nu_L$ is unique.\par
  First, we compute $\tilde{\nu}_L$ that is, according to a multiplicative constant,
\begin{align*}
  W_{(b,t)} & = \left( \prod_{i: b_i=b_{i+1}} \sqrt{p (1-p)^{|t_{i+1}-t_i|-1}} \right) \left( \prod_{i: b_i=1,b_{i+1}=-1} \left(\sum_{s \geq 1} \sqrt{p (1-p)^{s+t_i-1}} \sqrt{p (1-p)^{s+t_{i+1}-1}} \right) \right) \\
            & = p^{L} (1-p)^{-L}\left( \prod_{i: b_i = b_{i+1}} \sqrt{1-p}^{|t_{i+1}-t_i|} \right) \left( \prod_{i: b_i=1,b_{i+1}=-1} \sqrt{1-p}^{t_i+t_{i+1}} \sum_{s \geq 1} (1-p)^s \right)\\
            & = \left( \frac{p}{1-p} \right)^{L} \left( \prod_{i: b_i = b_{i+1}} \sqrt{1-p}^{|t_{i+1}-t_i|} \right) \left( \prod_{i: b_i=1,b_{i+1}=-1} \sqrt{1-p}^{t_i+t_{i+1}} \frac{1-p}{p} \right)\\
            & = \left( \frac{p}{1-p} \right)^{L} \left(\frac{1-p}{p} \right)^{\card{i: b_i=1,b_{i+1}=-1}} \\
            & \quad  \left( \prod_{i: b_i = b_{i+1}} \sqrt{1-p}^{|t_{i+1}-t_i|} \right) \left( \prod_{i: b_i=1,b_{i+1}=-1} \sqrt{1-p}^{t_i+t_{i+1}} \right).
\end{align*}
Now, for any $b \in \Bridges_L$, to find $\nu_L(b)$, we have to sum on $t = (t_i) \in \Time_b$. Before to do it, let us introduce some few notations about the local maxima and minima of bridges. First, for any bridge $b$, we say that $i+1/2$ (shorted in $i$ in the following) is a local minimum if $b_i=-1$ and $b_{i+1}=1$ and a local maximum if $b_i=1$ and $b_{i+1}=-1$. Because $b$ is a bridge, the number $k_b$ (rewritten $k$ when confusion on $b$ could not occur) of local maxima is equal to the number of local minima
\begin{equation}
k_b = \card{i: b_i=-1,b_i=1} = \card{i: b_i=1,b_i=-1}.
\end{equation}
In addition, we denote by $(m_1,\dots,m_k)$ the sequence of the positions of local min and $(M_1,\dots,M_k)$ the sequence of the position of local max such that $m_1 < M_1 < m_2 < \dots < M_k < m_1 +2L$. In those terms, $W_{(b,t)}$ rewrites as
\begin{align*}
  W_{(b,t)} = & \left(\frac{p}{1-p} \right)^{L} \left(\frac{1-p}{p}\right)^{k} \prod_{l=1}^k \sqrt{1-p}^{t_{M_l} - t_{m_l+1}} \sqrt{1-p}^{t_{M_l+1} - t_{m_{l+1}}} \prod_{l=1}^k \sqrt{1-p}^{t_{M_l}+t_{M_l+1}}\\
  = & \left(\frac{p}{1-p} \right)^{L} \left(\frac{1-p}{p}\right)^{k} \prod_{l=1}^k (1-p)^{t_{M_l}} (1-p)^{t_{M_{l}+1}} (1-p)^{-t_{m_l}}.  
\end{align*}
To write the last line, we use the fact that $t_{m_l+1} = t_{m_l}$. 
Now, $W_{(b,t)}$ depends only on values of $t$ around local minima and local maxima. But, be careful, we have constraints on them induced by the constraints given in $\Time_b$, see~\eqref{eq:timeset}. They are, for any $l \in \ZZ/k\ZZ$,
\begin{enumerate}
\item $t_{m_l} = t_{m_l+1}$,
\item $t_{M_l} \geq t_{m_l+1} + (M_l-m_l-1)$ if $M_l >  m_l + 1$,
\item $t_{M_l} = t_{m_l+1}$ if $M_l = m_l + 1$ (i.e.\ $b_{m_l+2} = -1$), 
\item $t_{M_l+1} \geq t_{m_{l+1}} + (m_{l+1}-M_l-1)$ if $m_{l+1} > M_l +1$,
\item $t_{M_l+1} = t_{m_{l+1}}$ if $M_l +1 = m_{l+1}$ (i.e.\ $b_{m_{l+1}-1} = 1$). 
\end{enumerate}
Now, we sum on $(t_i)$ that are not local extrema. In the following, we consider $M_l$ and $m_{l}+1$, but the reasoning is the same for $M_{l}+1$ and $m_{l+1}$. For that, we need to enumerate the number of increasing sequences of length $M_l-m_l$ in $[t_{M_l},t_{m_l+1}] \cap \ZZ$:
\begin{itemize}
\item if $M_l-m_l= 1$, then we need that $t_{M_l} = t_{m_l+1}$ that is the constraint 3,
\item if $M_l-m_l=2$, then we have $1$ sequence that is $(t_{m_l+1},t_{M_l})$,
\item if $M_l-m_l=k$ (with $k \geq 3$), then we have $\binom{t_{M_l}-t_{m_l+1}-1}{k-2} = \binom{t_{M_l}-t_{m_l+1}-1}{M_l-m_l-2}$ sequences, indeed we just have to choose $k-2$ numbers in the set $\{t_{m_l+1}+1,\dots,t_{M_l}-1\}$ of cardinal $t_{M_l}-t_{m_l+1}-1$. Note that $\binom{t_{M_l}-t_{m_l+1}-1}{M_l-m_l-2}= 1$ if $M_l-n_l=2$.
\end{itemize}
Now we sum on $(t_i)$ such that $i \notin \mathcal{E} = \{M_1,M_1+1,\dots,M_{k},M_{k}+1,m_1,m_1+1,\dots,m_k,m_k+1\}$,
\begin{align*}
  \sum_{t_i: i \notin \mathcal{E}} W_{(b,t)} = & \left(\frac{p}{1-p} \right)^{L} \left(\frac{1-p}{p}\right)^{k} \prod_{l=1}^k (1-p)^{-t_{m_l}} \left(\ind{M_l=m_l+1} (1-p)^{t_{M_l}} + \ind{M_l>m_l+1} \binom{t_{M_l}-t_{m_l+1}-1}{M_l-m_l-2} (1-p)^{t_{M_l}} \right)\\
  & \qquad \left(\ind{m_{l+1} = M_l+1} (1-p)^{t_{M_l+1}} + \ind{m_{l+1}> M_l+1} \binom{t_{M_{l}+1}-t_{m_{l+1}}-1}{m_{l+1}-M_{l}-2} (1-p)^{t_{M_l+1}}\right). 
\end{align*}

Now, we sum on $(t_{M_l})_{1 \leq l \leq k}$ and $(t_{M_l+1})_{1 \leq l \leq k}$, for that, we need the following lemma:
\begin{lemma} \label{lem:comb}
  For any $k \in \NN^*$, for any $q \in (0,1)$,
  \begin{equation}
    \sum_{i=k}^\infty \binom{i-1}{k-1} q^i = \left( \frac{q}{1-q} \right)^k.
  \end{equation}
\end{lemma}
\begin{proof}
  The proof is done in Section~\ref{sec:lem-comb} just after this one.\par
\end{proof}

And, so, if $M_l > m_l+1$,
\begin{align*}
  \sum_{t_{M_l} = t_{m_l+1} + M_l-m_l-1}^{\infty} \binom{t_{M_l}-t_{m_l+1}-1}{M_l-m_l-2} (1-p)^{t_{M_l}}
  = & \sum_{u = M_l-m_l-1}^{\infty} \binom{u-1}{M_l-m_l-2} (1-p)^{u+t_{m_l+1}}\\
  = & \left(\frac{1-p}{p} \right)^{M_l-m_l-1} (1-p)^{t_{m_l+1}}.
\end{align*}
Note that if $M_l=m_l+1$, we obtain $\left(1-p\right)^{t_{m_l+1}}= \left(1-p\right)^{t_{M_L}}$. So, after the sum on $(t_{M_l})_{1 \leq l \leq k}$ and $(t_{M_l+1})_{1 \leq l \leq k}$, we obtain
 \begin{align*}
   \sum_{t_i: i \notin \{m_l: 1 \leq l \leq k\}} W_{(b,t)} = & \left(\frac{p}{1-p} \right)^{L} \left(\frac{1-p}{p}\right)^{k} \\
                                                             & \quad \prod_{l=1}^k (1-p)^{-t_{m_l}} \left(\frac{1-p}{p} \right)^{M_l-m_l-1} (1-p)^{t_{m_{l+1}}}  \left(\frac{1-p}{p} \right)^{m_{l+1}-M_l-1} (1-p)^{t_{m_{l+1}}}.
 \end{align*}
 We recall once again that $t_{m_l} = t_{m_l+1}$, so
 \begin{align*}
   \sum_{t_i: i \notin \{m_l: 1 \leq l \leq k\}} W_{(b,t)} = & \left(\frac{p}{1-p} \right)^{L} \left(\frac{1-p}{p}\right)^{k} \prod_{l=1}^k (1-p)^{t_{m_l}} \left( \frac{1-p}{p} \right)^{m_{l+1}-m_l-2}\\
   = &  \left(\frac{p}{1-p} \right)^{L} \left(\frac{1-p}{p}\right)^{k} \left(\frac{1-p}{p}\right)^{-2k} \left(\frac{1-p}{p} \right)^{(m_1+2L - m_k) + (m_k -m_{k-1}) + \dots + (m_2 - m_1) }  \prod_{l=1}^k (1-p)^{t_{m_l}}\\
   = &  \left(\frac{p}{1-p} \right)^{L} \left(\frac{1-p}{p}\right)^{k} \left(\frac{1-p}{p}\right)^{-2k} \left(\frac{1-p}{p} \right)^{2L}  \prod_{l=1}^k (1-p)^{t_{m_l}}\\
 \end{align*}

 Now, we sum on $(t_{m_l})_{1 \leq l \leq k}$ to find
 \begin{equation}
   \nu_L(b) = \frac{1}{Z} \left( \frac{p}{1-p} \right)^k \left( \frac{1}{p} \right)^k = \frac{1}{(1-p)^k}.  \qedhere
 \end{equation} 
\end{proof}

\subsection{Proof of Lemma~\ref{lem:comb}} \label{sec:lem-comb}
  In this annex, we prove Lemma~\ref{lem:comb}. First, we need to prove the following lemma on sums of binomials.
\begin{lemma} \label{lem:sumcomb}
  For any $n,k \in \NN$ such that $k \leq n$, 
  \begin{equation}
    \sum_{j=k}^{n} \binom{j}{k} = \binom{n+1}{k+1}.
  \end{equation}
\end{lemma}

 \begin{proof}
   It is proved by induction on $n$. If $n=0$ (and so $k=0$), it is $\binom{0}{0} = 1 = \binom{1}{1}$. Now, take $n \geq 0$. If $k=n+1$, it is $\binom{n+1}{n+1} = \binom{n+2}{n+2}$. Now, take any $k \leq n$, by induction hypothesis and Pascal's rule, 
   \begin{equation}
     \sum_{j=k}^{n+1} \binom{j}{k} = \sum_{j=k}^n \binom{j}{k} + \binom{n+1}{k} = \binom{n+1}{k+1} + \binom{n+1}{k} = \binom{n+2}{k+1}. \qedhere 
   \end{equation}
 \end{proof}

 Now Lemma~\ref{lem:comb} is a corollary of Lemma~\ref{lem:sumcomb}.
\begin{proof}[Proof of Lemma~\ref{lem:comb}]
  It is proved by induction on $k$. If $k=1$, the sum is
  \begin{equation}
    \sum_{i=1}^\infty \binom{i-1}{0} q^i = \sum_{i=1}^\infty q^i = \frac{q}{1-q}.
  \end{equation}
  
  Now, we suppose that $k \geq 1$, then, by induction and Lemma~\ref{lem:sumcomb},
  \begin{align}
    \left( \frac{q}{1-q} \right)^{k+1}
    = & \left( \sum_{j=1}^\infty q^j \right) \left(\sum_{i=k}^\infty \binom{i-1}{k-1} q^i \right) 
        = \sum_{i=k}^\infty \sum_{j=1}^\infty \binom{i-1}{k-1} q^{i+j} \\
    = & \sum_{l=k+1}^\infty q^{l} \left(\sum_{m=k}^{l-1} \binom{m-1}{k-1} \right)
        = \sum_{l=k+1}^\infty q^{l} \binom{l-1}{k}. \qedhere
  \end{align}
\end{proof}

\end{document}